\documentclass[prx,twocolumn,amsmath,amssymb,showpacs,superscriptaddress]{revtex4-2}
\UseRawInputEncoding
\usepackage{graphicx}% Include figure files
\usepackage{epstopdf}
\usepackage{amsthm}
\usepackage{dcolumn}% Align table columns on decimal point
\usepackage{bm}% bold math
\usepackage[FIGTOPCAP]{subfigure}
\usepackage[usenames,dvipsnames]{color}
\usepackage{mathtools}
\usepackage[all]{xy}
\usepackage{lipsum}
\usepackage[hidelinks]{hyperref}
\usepackage{algpseudocode}

\def \bea{\begin{eqnarray}}
	\def \eea{\end{eqnarray}}

\newtheorem{theorem}{Theorem}[section]

\newcommand{\J}{\mathbf{J}}
\newcommand{\Q}{\mathbf{Q}}
\newcommand{\I}{\mathbf{I}}
\newcommand{\Pb}{\mathbf{P}}

\DeclareMathOperator{\Tr}{Tr}
\newcommand{\G}{\mathbf{G}}
\newcommand{\Y}{\mathbf{Y}}
\newcommand{\Ss}{\mathbf{S}}
\newcommand{\Ub}{\mathbf{U}}

\DeclareMathAlphabet\mathbfcal{OMS}{cmsy}{b}{n}

\definecolor{mygray}{gray}{0.8}
\definecolor{plum}{rgb}{.5,0,1}
\begin{document}
	%
	% \title{Analytical Power Spectral Density of Stochastic Biological Systems: \\ Leveraging Element-wise Solutions and Recursive Algorithms}
 \title{Element-wise and Recursive Solutions for the Power Spectral Density\\ of Biological Stochastic Dynamical Systems at Fixed Points}
	\author{Shivang Rawat}
	\email{sr6364@nyu.edu}
    \affiliation{Courant Institute of Mathematical Sciences, New York University, New York 10003, USA}
    \affiliation{Center for Soft Matter Research, Department of Physics, New York University, New York 10003, USA}
	
	\author{Stefano Martiniani}
	\email{sm7683@nyu.edu}
    \affiliation{Courant Institute of Mathematical Sciences, New York University, New York 10003, USA}
    \affiliation{Center for Soft Matter Research, Department of Physics, New York University, New York 10003, USA}
    \affiliation{Simons Center for Computational Physical Chemistry, Department of Chemistry, New York University, New York 10003, USA}

\begin{abstract}
 Stochasticity plays a central role in nearly every biological process, and the noise power spectral density (PSD) is a critical tool for understanding variability and information processing in living systems. In steady-state, many such processes can be described by stochastic linear time-invariant (LTI) systems driven by Gaussian white noise, whose PSD is a complex rational function of the frequency that can be concisely expressed in terms of their Jacobian, dispersion, and diffusion matrices, fully defining the statistical properties of the system's dynamics at steady-state. Here, we arrive at compact element-wise solutions of the rational function coefficients for the auto- and cross-spectrum that enable the explicit analytical computation of the PSD in dimensions $n=2,3,4$. We further present a recursive Leverrier-Faddeev-type algorithm for the exact computation of the rational function coefficients. Crucially, both solutions are free of matrix inverses. We illustrate our element-wise and recursive solutions by considering the stochastic dynamics of neural systems models, namely Fitzhugh-Nagumo ($n=2$), Hindmarsh-Rose ($n=3$), Wilson-Cowan ($n=4$), and the Stabilized Supralinear Network ($n=22$), as well as an evolutionary game-theoretic model with mutations ($n=5, 31$). We extend our approach to derive a recursive method for calculating the coefficients in the power series expansion of the integrated covariance matrix for interacting spiking neurons modeled as Hawkes processes on arbitrary directed graphs.
\end{abstract}

\maketitle

Unlike in an abstract painting by Mondrian \cite{mondrian1943}, the world is not made of perfect forms. Rather, we are surrounded by imperfect, yet regular patterns emerging from the randomness of microscopic processes that collectively give rise to statistically predictable phenomena. A point in case is the derivation of the diffusion equation from the random (Brownian) motion of microscopic particles \cite{einstein1905brownian, langevin1908theory}. This is but one example of the many processes that can be recapitulated in terms of ordinary differential equations driven by random processes, known as stochastic differential equations (SDEs). For instance, the effect of thermal fluctuations on synaptic conductance (linked to the stochastic opening of voltage-gated ion channels) and on the voltage of an electrical circuit (Johnson's noise), or the fluctuations of a neuron's membrane potential driven by stochastic spike arrival, can all be modeled as SDEs. Alternatively, SDEs are commonly adopted to model processes for which the exact dynamics are not known, and it is thus appropriate to express the model in terms of random variables encoding our uncertainty about the state of the system. Examples are the trajectory of an aircraft as modeled by navigation systems, or the evolution of stock prices subject to market fluctuations \cite{sde_book}.

Among stochastic models, linear time-invariant (LTI) SDEs play a special role because (i) they are amenable to analytical treatment and (ii) the linear response of many nonlinear systems (e.g., around a fixed point of the dynamics) can be reduced to LTI SDEs. To make this point explicit, let us consider a non-linear autonomous dynamical system driven by additive white noise with finite variance, namely
\begin{equation}
    \mathrm{d}\mathbf{x} = \mathbf{f}(\mathbf{x})\mathrm{d}t + \mathbf{L} \mathrm{d}\mathbf{W}    \label{eq:main1}
\end{equation}
where $\textbf{x} \in \mathbb{R}^n$ is the state vector, $\textbf{f} : \mathbb{R}^{n} \to \mathbb{R}^{n}$ is a nonlinear drift function encoding the deterministic dynamics, $\mathbf{L} \in \mathbb{R}^{n \times m}$ is the \textit{dispersion} matrix defining how noise enters the system, and $\mathrm{d}\mathbf{W}\in \mathbb{R}^{m}$ is a vector of mutually independent Gaussian increments satisfying $\mathbb{E}[\mathrm{d}\textbf{W}\mathrm{d}\textbf{W}^\top] = \mathbf{D} \mathrm{d}t$ where $\mathbf{D}$ is the \textit{diffusion} matrix (or spectral density matrix) with entries $D_{ij} = \delta_{ij}\sigma_i^2$, where $\delta_{ij}$ is the Kronecker delta \cite{sde_book}. 

We are interested in the power spectral density (PSD) of the system's response, $\mathbf{x}(t)$, near a stable fixed point, $\mathbf{x}_0$, so we proceed by linearizing the drift function about this point, yielding the following LTI system
\begin{equation}
    \mathrm{d}\mathbf{x} = \mathbf{J}  \mathbf{x} \mathrm{d}t + \mathbf{L} \mathrm{d}\mathbf{W} \label{eq:main2}
\end{equation}
where we have redefined $\mathbf{x} (t) \equiv \mathbf{x}(t) - \mathbf{x}_0$ such that $\mathbf{x}(t_0) = \mathbf{0}$, and $\textbf{J} = \left.|\mathrm{d}\mathbf{f}/\mathrm{d}\mathbf{x}\right|_{\mathbf{x}_0} \in \mathbb{R}^{n \times n}$ is the Jacobian evaluated at the fixed point, that we take to be Hurwitz (real part of all eigenvalues less than zero) to assure stability. The solution to this linear SDE is given by the stochastic integral $\mathbf{x}(t) = \int_{-\infty}^t e^{\mathbf{J}(t-s)}\mathbf{L}\mathrm{d}\mathbf{W}(s)$ which is a zero-mean stationary process with covariance $\bf{\mathbfcal{R}}(\tau)$. 

The PSD of this stochastic process is known to be expressed in terms of the following matrix product (see \cite{gardiner1985handbook} and \textit{Supplemental Material}), 
\begin{equation}
    \mathbfcal{S}_\mathbf{x}(\omega) = (\dot{\iota}\omega \mathbf{I} + \mathbf{J})^{-1} \, \mathbf{L} \, \mathbf{D} \, \mathbf{L}^\top (-\dot{\iota}\omega \mathbf{I} + \mathbf{J})^{-\top} \label{eq:main_mat_sol}
\end{equation}

For the sake of brevity, from now on we refer to $\mathbfcal{S}_\mathbf{x}(\omega)$ as $\mathbfcal{S}(\omega)$, whose $i,j^{\text{th}}$ entry is denoted $\mathcal{S}^{ij}(\omega)$. $\mathbfcal{S}(\omega)$ is a product of the spectral factors (the transfer function matrix) of the dynamical system and can be expressed as a rational function of the following form (see \cite{kailath1974view, sayed2001survey} \& \textit{Appendix~\ref{sec:appendix_rational}}),
\begin{equation}
\begin{split}
    \mathbfcal{S}(\omega) &= \frac{\Pb(\omega) + \dot{\iota}\omega \Pb^\prime(\omega)}{Q(\omega)} \\
    &= \frac{\sum\limits_{\alpha=0}^{n-1}\Pb_\alpha \omega^{2\alpha} + \dot{\iota}\omega \sum\limits_{\alpha=0}^{n-2}\Pb_\alpha^\prime \omega^{2\alpha}}{\sum\limits_{\alpha=0}^{n}q_\alpha \omega^{2\alpha}} \label{eq:recursive2} 
\end{split}
\end{equation}
and the corresponding PSD matrix elements are of the form,
\begin{equation}
    \mathcal{S}^{ij}(\omega) = \frac{P^{ij}(\omega) + \dot{\iota} \delta_{ij} \omega   P^{ij\prime}(\omega)}{Q(\omega)}
    \label{eq:main4}
\end{equation}
where $P^{ij}(\omega) = \sum_{\alpha=0}^{n-1} p^{ij}_{\alpha}\omega^{2\alpha}$, $P^{ij\prime}(\omega)=\sum_{\alpha=0}^{n-2} p^{ij\prime}_{\alpha}\omega^{2\alpha}$ and $Q(\omega) = \sum_{\alpha=0}^{n} q_{\alpha}\omega^{2\alpha}$ are even powered polynomials with real coefficients ($p^{ij}_{\alpha}$, $p^{ij\prime}_{\alpha}$ and $q_\alpha$).

In this work, we derive closed-form expressions for the individual coefficients of the rational function of a given element of $\mathbfcal{S}(\omega)$, i.e., $\mathcal{S}^{ij}(\omega)$. In Sec.~\ref{sec:element-wise} we provide the formulae for the auto-spectrum of $2$, $3$, and $4$-dimensional systems expressed in terms of elementary trace operations applied to the submatrices of $\mathbf{J}$. The element-wise solutions for the auto- and cross-spectrum of a generic $n$-dimensional system are presented in the \textit{Supplemental Material}, and they can be obtained directly in symbolic form through our software (see section \ref{sec:code}). 

In Sec.~\ref{sec:recursive-sol}, we derive a recursive algorithm (Theorem~\ref{theorem:0}) to calculate the coefficient matrices $\Pb_\alpha$ and $\Pb_\alpha^\prime$ and the coefficients of the denominator, $q_\alpha$. We follow an approach similar to the one adopted by Hou \cite{hou1998classroom} to derive the Leverrier-Faddeev (LF) method, which is a recursive approach to find the characteristic polynomial of a matrix in $O(n^4)$ operations, where $n$ is the dimensionality of the system \cite{faddeev1963computational, bernstein2009matrix}. Note that unlike Eq.~\ref{eq:main_mat_sol}, our solutions are free of matrix inverses, which is numerically advantageous when the system is ill-conditioned. To demonstrate the wide applicability of the recursive algorithm, we derive a recursive solution for the rational function of an auxiliary spectral matrix, Appendix~\ref{sec:appendix_recursive}. Using this solution, we develop a recursive method (Theorem~\ref{theorem:without_omega}) to find the power series expansion of the integrated covariance matrix of the Hawkes process \cite{hawkes1971point} with arbitrary connectivity (Sec.~\ref{sec:application}).

In Sec.~\ref{sec:application} we apply our solutions to several stochastic nonlinear models from neuroscience and evolutionary biology, providing model-specific solutions for the auto-spectrum when appropriate. In practice, readers seeking analytical expressions for the auto-spectrum of specific models (e.g., Fitzhugh-Nagumo \cite{fitzhugh1961impulses, nagumo1962active, yamakou2019stochastic}, or Hindmarsh-Rose \cite{hindmarsh1984model}), can refer directly to Sec.~\ref{sec:prelim-examples} and \ref{sec:application}.

\section{Preliminary examples}
\label{sec:prelim-examples}
\begin{figure}[t]
    \centering
    \includegraphics[width=\linewidth]{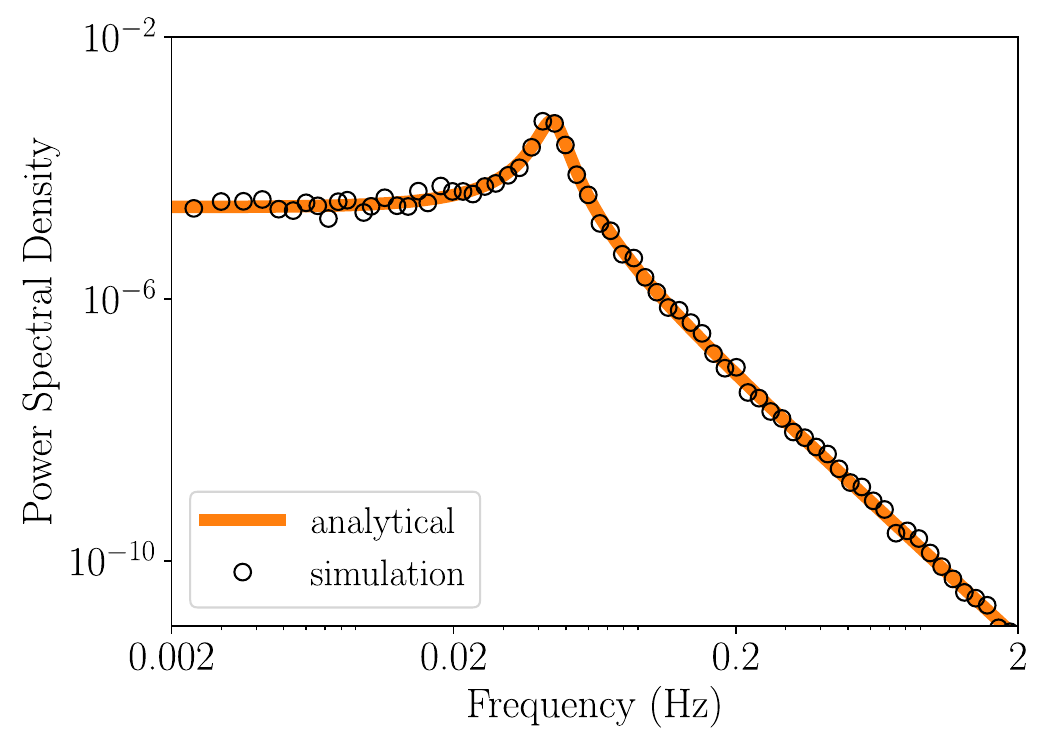}
    \caption{\textbf{Power spectral density of the membrane potential for the Fitzhugh-Nagumo model.} A comparison of the PSD of the membrane potential ($v$) calculated via simulation and the analytical rational function solution. The parameters used to simulate the system are $I = 0.265$, $\alpha = 0.7$, $\beta = 0.75$ and $\epsilon = 0.08$ that yields a fixed point solution $v_e = -1.00125$, $w_e = -0.401665$. Here, we plot the auto-spectrum for the case of multiplicative noise, i.e., $h_1(v)=0$ and $h_2(w)=\sigma w$, where $\sigma=0.001$.}
    \label{fig:psd_fhn}
\end{figure}
Before presenting the general solutions, it is instructive to review concrete 1- and 2-dimensional cases that are commonly encountered in the literature, for which we can derive the coefficients analytically with little effort, without resorting to our sophisticated solutions. To this end, let us consider the simplest and most common case: the auto-spectrum of an LTI system with mutually independent additive noise (i.e.,  $\textbf{L}=\textbf{I}$). For the $d=1$ case, also known as the Ornstein-Uhlenbeck process, we have
\begin{equation}
   \mathrm{d}x=-\tau_r^{-1}x\mathrm{d}t + \mathrm{d}W
\end{equation}
with $\mathrm{d}W^2=\sigma^2\mathrm{d}t$. It can be seen from Eq.~\ref{eq:main_mat_sol} that
\begin{equation}
   \mathcal{S}(\omega)= \frac{\sigma^2}{\tau_r^{-2} + \omega^2}
   \label{eq:main6}
\end{equation}
so that the coefficients in Eq.~\ref{eq:main4} are $p^{11}_{0}=\sigma^2$, $q_{0}=\tau_r^{-2}$ and $q_{1}=1$. \footnote{Note that Eq.~\ref{eq:main6} readily leads us to the correlation function $R(\tau)=\mathcal{F}^{-1}(S(\omega)) = \frac{\tau_r}{2}\sigma^2 e^{-|\tau|/\tau_r}$}
% or, alternatively, $C(\tau) = C(0)e^{-|\tau|/\tau_r}$ where $C(0) = (2\pi)^{-1}\int_{-\infty}^{\infty} S(\omega) = \sigma^2\tau_r/2$.

A slightly more sophisticated example is the stochastic Fitzhugh-Nagumo model \cite{fitzhugh1961impulses, nagumo1962active, yamakou2019stochastic}, which amounts to a simplification of the Hodgkin-Huxley equations for action potential generation \cite{hodgkin1952quantitative}. We simulate the dynamical evolution of a neuron's membrane potential in response to an input current by the following set of SDEs
\begin{equation}
\begin{split}
    \frac{\mathrm{d}v}{\mathrm{d}t} &= \biggl(v - \frac{v^3}{3} - w + I \biggr) + h_1(v) \eta_1\\ 
    \frac{\mathrm{d}w}{\mathrm{d}t} &= \epsilon (v + \alpha -\beta w ) + h_2(w) \eta_2 \label{eq:fhn-nonlinear}
\end{split}
\end{equation}

Here, $v$ represents the difference between the membrane potential and the resting potential, $w$ is the variable associated with the recovery of the ion channels after the generation of an action potential, and $I$ is the external current input. The parameters $\alpha$ and $\beta$ are related to the activation threshold of the fast sodium channel and the inactivation threshold of the slow potassium channels, respectively. $\epsilon \ll 1$ sets the time scale for the evolution of the activity, making the dynamics of $v$ much faster than those of $w$. In the subthreshold regime, the fixed point, $(v_e, w_e)$, can be found analytically as a function of the parameters of the model and the input current (details in \textit{Supplemental Material}). The stochasticity in $v$ and $w$ is represented by two uncorrelated Gaussian white noise sources, $\eta_1$ and $\eta_2$, with a standard deviation of 1. We consider the case of multiplicative noise which is realized by setting $h_1(v)$ and $h_2(w)$ proportional to the respective function variables. Specifically, we set $h_1(v)=0$ and $h_2(w)=\sigma w$. We choose the variance of the noise small enough that we are always in the vicinity of the fixed point. Then, upon linearizing Eq.~\ref{eq:fhn-nonlinear} around the fixed point, $(v_e, w_e)$, we arrive at the following LTI SDEs,
\begin{equation}
    \begin{bmatrix}
    \dot{v}\\
     \dot{w}
    \end{bmatrix} = 
    \begin{bmatrix}
    1-v_e^2 & -1\\
     \epsilon & -\beta\epsilon
    \end{bmatrix} 
    \begin{bmatrix}
    v - v_e\\
     w - w_e
    \end{bmatrix} + 
    \sigma\begin{bmatrix}
    0 & 0\\
    0 & w_e
    \end{bmatrix}
    \begin{bmatrix}
    \eta_1\\
     \eta_2
    \end{bmatrix}
\end{equation}

Plugging the Jacobian and noise matrices into Eq.~\ref{eq:main_mat_sol}, we obtain the PSD of the membrane potential, $v$, as
\begin{equation}
    \begin{aligned}
    \mathcal{S}_v(\omega) &= w_e^2 \sigma^2 {\Big /} \left [ (\epsilon+(v_e^2 - 1)\beta\epsilon)^2 \right. \\ 
    &+ \left. ((v_e^2 - 1)^2-2\epsilon+\beta^2\epsilon^2)\omega^2 + \omega^4 \right ]
    \end{aligned}
    \label{eq:psd_fhn}
\end{equation}
so that in Eq.~\ref{eq:main4} we have $p^{11}_{0}=w_e^2 \sigma^2$, $p^{11}_{1}=0$, $q_{0}=(\epsilon+(v_e^2 - 1)\beta\epsilon)^2$, $q_{1}=((v_e^2 - 1)^2-2\epsilon+\beta^2\epsilon^2)$ and $q_{2}=1$.

We verify the accuracy of this solution, Eq.~\ref{eq:psd_fhn}, by comparison with the PSD obtained numerically by simulating the nonlinear model, Eq.~\ref{eq:fhn-nonlinear}, with multiplicative noise, and find excellent agreement as shown in Fig.~\ref{fig:psd_fhn}.

This laborious approach works well for systems with dimensions $n=1,2$, but it becomes infeasible in higher dimensions. This is due to the multiplication of inverses of two matrices with complex entries, $(\dot{\iota}\omega \mathbf{I} + \mathbf{J})^{-1}$ and $(-\dot{\iota}\omega \mathbf{I} + \mathbf{J})^{-\top}$.

\section{Element-wise solution}
\label{sec:element-wise}

\begin{figure*}[htpb!]
    \centering
    \includegraphics[width=0.92\linewidth]{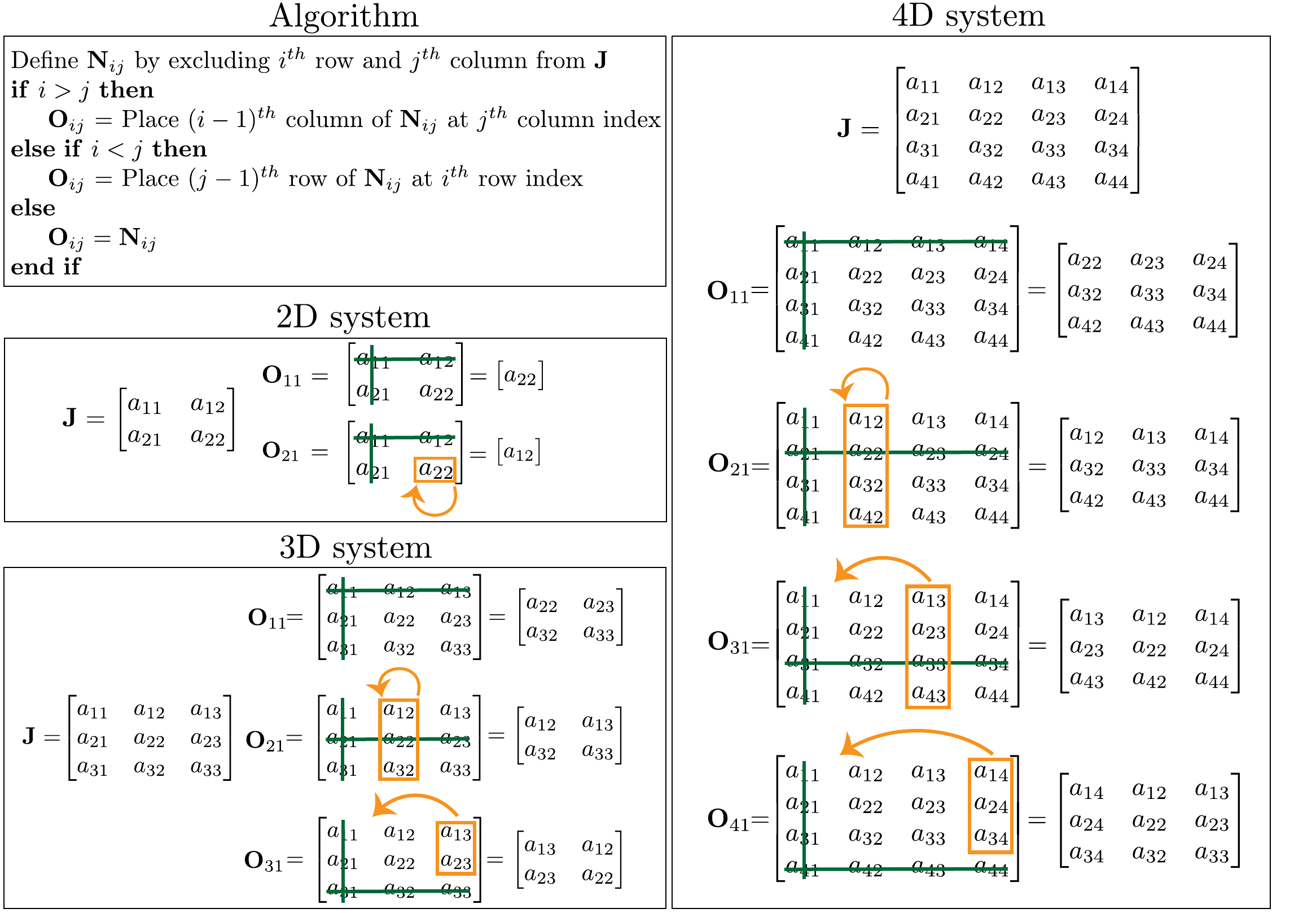}
    \caption{\textbf{Definitions of $\mathbf{O}_{ij}$ matrices.} We show the algorithm to generate the $\mathbf{O}_{ij}$ matrices with examples for the 2, 3 and 4D systems. To find $\mathbf{O}_{ij}$, we have to first find the excluded matrix by removing the $i^{th}$ row and $j^{th}$ column (denoted by green lines) from the Jacobian matrix and then perform the row/column change operations, denoted by orange arrows.} 
    \label{fig:O_algo}
\end{figure*}

We now introduce element-wise solutions for the auto-spectrum of 2, 3, and 4-dimensional systems. The general $n$-dimensional solutions for both auto- and cross-spectrum can be found in the \textit{Supplemental Material} alongside their derivations, and can be conveniently obtained in symbolic form for any dynamical system through our software (see section \ref{sec:code}). This method relies on expressing the rational function expression of Eq.~\ref{eq:main_mat_sol} in terms of elementary matrix operations on submatrices of the Jacobian, $\J$. Although this approach requires roughly $O(n^4\log n)$ (\textit{Supplemental Material, Fig.~S2}), operations for each element of the PSD matrix, i.e., $O(n^6\log n)$ total, it is particularly valuable to arrive at analytical closed-form solutions for the auto- or cross-spectra of specific dynamical variables, yielding greater insight into the dependence of the power spectral density on the model's functional form and its parameters. For simplicity, we consider the auto-spectrum of an LTI system driven by uncorrelated white noise (see \textit{Supplemental Material} for the most general case of correlated noise and/or the cross-spectrum). 

We start by noting that the denominator in Eq.~\ref{eq:main4} is independent of the type of noise driving the system, and it can be written as an even-powered expansion in $\omega$ with real coefficients. Similarly, the numerator of the auto-spectrum is an even-powered expansion in $\omega$ with real coefficients, but they depend directly on how the noise is added via the dispersion matrix, $\mathbf{L}$, whose elements we denote $l_{ij}$. For the particular case of independent noise, i.e., $l_{ij}=0$ for $i \neq j$, the coefficient of $\omega^{2\alpha}$ in the numerator of the rational function for the $i^{th}$ variable takes the simplified form, 
\begin{equation}
\label{eq:p-auto}
   p^{ii}_\alpha  = \sigma^2_i \, l_{i}^2 \,  d^\alpha(\mathbf{O}_{ii}) + \sum_{\substack{j=1 \\ j\neq i}}^n \sigma^2_j \,  l_{j}^2 \, f^\alpha(\mathbf{O}_{ji},\mathbf{e}_{\min( i,j )})
\end{equation}
where $d^\alpha(\mathbf{O}_{ii})$ and $f^\alpha(\mathbf{O}_{ji})$ are functions of the Jacobian submatrices, $\mathbf{O}_{ij}$, defined according to the algorithm in Fig.~\ref{fig:O_algo}. The functions themselves can be expressed in terms of Bell polynomials \cite{bell1934exponential}, and are defined in \textit{Supplemental Material (Sec.~5)}. $\mathbf{e}_\beta \in \{0,1\}^{n-1 \times 1}$ is the standard unit vector equal to 1 at index $\beta=\min(i, j)$ and 0 everywhere else. $\sigma_i^2 \equiv \sigma_{ii}^2$ is the $i$-th diagonal entry of the noise matrix $\mathbf{D}$, and, since $\mathbf{L}$ is diagonal, we also use the notation $l_i \equiv l_{ii}$.

Here, we only provide the solution for the auto-spectrum of the first variable, i.e., $i=1$. This solution can be generalized for any variable $i\neq 1$, as shown in \textit{Supplemental Material}. Alternatively, one can simply interchange the rows and columns of the LTI system so that the variable of interest appears at index $i=1$, and directly use the solution reported in this section.

Since the coefficients of the denominator are the same for all dimensions, $n$, we define them first,

\begin{equation}
\begin{split}
    q_{n-4} &= \frac{1}{24}\bigg[\Tr^{4}{\left(\mathbf{J}^{2} \right)} - 6 \Tr^{2}{\left(\mathbf{J}^{2} \right)} \Tr{\left(\mathbf{J}^{4} \right)} \\& + 8 \Tr{\left(\mathbf{J}^{2} \right)} \Tr{\left(\mathbf{J}^{6} \right)} + 3 \Tr^{2}{\left(\mathbf{J}^{4} \right)} - 6 \Tr{\left(\mathbf{J}^{8} \right)}\biggr]\\
    q_{n-3} &= \frac{\Tr^{3}{\left(\mathbf{J}^{2} \right)} - 3 \Tr{\left(\mathbf{J}^{2} \right)} \Tr{\left(\mathbf{J}^{4} \right)} + 2 \Tr{\left(\mathbf{J}^{6} \right)}}{6}  \\
    q_{n-2} &= \frac{\Tr^{2}{\left(\mathbf{J}^{2} \right)} - \Tr{\left(\mathbf{J}^{4} \right)}}{2} \\
    q_{n-1} &= \Tr{\left(\mathbf{J}^{2} \right)} \\
    q_{n} &= 1 
    \label{eq:ration_denominator_compiled}
\end{split}
\end{equation}
Note that for each $n$ we will only need to consider the coefficients $q_\alpha$ with $\alpha\geq 0$. 

\subsection{2-D solution}
For a 2D system, we can write the auto-spectrum of the first variable as,
\begin{equation}
    \mathcal{S}^{11}(\omega) = \frac{p_0 + p_1 \omega^2}{q_0 + q_1 \omega^2 + q_2 \omega^4}
\end{equation}
where the coefficients of the numerator are given by the equations
\begin{equation}
\begin{split}
    p_0 &= l_{1}^2\sigma_{1}^2\Tr(\mathbf{O}_{11}^2) + l_{2}^2\sigma_{2}^2\Tr(\mathbf{O}_{21}^2) \\
    p_1 &= l_{1}^2\sigma_{1}^2
\end{split}
\end{equation}

Here, the matrices $\mathbf{O}_{11} = [a_{22}$] and $\mathbf{O}_{21} = [a_{12}]$ are found from the 2D Jacobian using the algorithm summarized in Fig.~\ref{fig:O_algo} (\textit{n.b.}, we denote $a_{ij}$ the $i,j^\text{th}$ entry of the Jacobian). The coefficients of the denominator are given by the last three entries of Eq.~\ref{eq:ration_denominator_compiled}.

\subsection{3-D solution}
For a 3D system, the auto-spectrum of the first variable takes the form,
\begin{equation}
\mathcal{S}^{11}(\omega) = \frac{p_0 + p_1 \omega^2 + p_2 \omega^4}{q_0 + q_1 \omega^2 + q_2 \omega^4 + q_3 \omega^6} \label{eq:rational_fun_3D}
\end{equation}
where the coefficients of the numerator are given by the equations
% \begin{widetext}
\begin{equation}
\begin{split}
    p_0 =& \frac{1}{2} \biggl[l_{1}^{2} \sigma_{1}^{2} \bigl(\Tr^{2}{(\mathbf{O}_{11}^{2})}- \Tr{(\mathbf{O}_{11}^{4} )}\bigr) +\sum_{j=2}^3 l_{j}^{2} \sigma_{j}^{2} \bigl(\Tr^{2}{({\mathbf{O}^\prime_{j1}}^{2})}\\  &-  \Tr{({\mathbf{O}^\prime_{j1}}^{4} )} + \Tr^{2}{(\mathbf{O}_{j1}^{2} )} - \Tr{(\mathbf{O}_{j1}^{4} )}\bigr) \biggr]\\
    p_1 =& l_{1}^{2} \sigma_{1}^{2} \Tr{(\mathbf{O}_{11}^{2} )} + \sum_{j=2}^3 l_{j}^{2} \sigma_{j}^{2} \bigl(- 2 \Tr{(\mathbf{O}^\prime_{j1} )} \Tr{(\mathbf{O}_{j1})} \\ &+ \Tr{({\mathbf{O}^\prime_{j1}}^{2} )} + \Tr^{2}{(\mathbf{O}_{j1} )}\bigr)\\
    p_2 =& l_{1}^{2} \sigma_{1}^{2} \label{eq:rational_fun_3D_num}
\end{split}
\raisetag{5\baselineskip} 
\end{equation}
% \end{widetext}
Here, the matrices $\mathbf{O}_{11}$, $\mathbf{O}_{21}$ and $ \mathbf{O}_{31}$ are found from the 3D Jacobian matrix using the algorithm summarized in Fig.~\ref{fig:O_algo}. Furthermore, the matrices $\mathbf{O}^\prime_{ij}$ are found by excluding the $k^{th}$ row and column from the corresponding $\mathbf{O}_{ij}$ matrix, where $k=\min(i,j)$. Since we are interested in the spectrum of the first variable ($j=1$), $k=\min(i,1)=1$. Thus, $\mathbf{O}^\prime_{21}$ and $\mathbf{O}^\prime_{31}$ are found by removing the first row and column from the matrices $\mathbf{O}_{21}$ and $\mathbf{O}_{31}$, respectively. The coefficients of the denominator are given by the last four entries of Eq.~\ref{eq:ration_denominator_compiled}.

If the variance of the noise is the same for all the variables, i.e., $l_{1}^{2} \sigma_{1}^{2} = l_{2}^{2} \sigma_{2}^{2} = l_{3}^{2} \sigma_{3}^{2} = \sigma^2$, we can further simplify the coefficients of the numerator, which become
\begin{equation}
\begin{split}
    p_0 &= \sigma^2 \mathrm{det}(\mathbf{A}^\top\mathbf{A}) \\
    p_1 &= \sigma^2 \Tr(\mathbf{A}_1^\top \mathbf{A}_1 + \mathbf{A}_2^2) \\
    p_2 &= \sigma^2
\end{split}
\end{equation}
where, $\mathbf{A} \in \mathbb{R}^{3 \times 2}$ is the submatrix of $\J$ obtained by removing its first column. Then, $\mathbf{A}_1 \in \mathbb{R}^{1 \times 2}$ is the submatrix obtained by isolating the first row of $\mathbf{A}$,  and $\mathbf{A}_2 \in \mathbb{R}^{2 \times 2}$ is the submatrix obtained by removing the first row of $\mathbf{A}$.

\subsection{4-D solution}
For a 4D system, the auto-spectrum of the first variable becomes
\begin{equation}
\mathcal{S}^{11}(\omega) = \frac{p_0 + p_1 \omega^2 + p_2 \omega^4 + p_3 \omega^6}{q_0 + q_1 \omega^2 + q_2 \omega^4 + q_3 \omega^6 + q_4 \omega^8}
\end{equation}
where the coefficients of the numerator are given by Eq.~\ref{eq:rational_fun_4D_num}.
Here, the matrices $\mathbf{O}_{11}$, $\mathbf{O}_{21}$, $\mathbf{O}_{31}$ and $\mathbf{O}_{41}$  are found from the 4D Jacobian matrix using the algorithm summarized in Fig.~\ref{fig:O_algo}. Using an argument similar to the 3D case, $\mathbf{O}^\prime_{21}$, $\mathbf{O}^\prime_{31}$ and $\mathbf{O}^\prime_{41}$ can be found by removing the first row and column from the matrices $\mathbf{O}_{21}$, $\mathbf{O}_{31}$ and $\mathbf{O}_{41}$, respectively. The coefficients of the denominator are given by Eq.~\ref{eq:ration_denominator_compiled}.
\begin{widetext}
\begin{equation}
\begin{split}
    p_0 =& \frac{1}{6}\biggl[l_{1}^{2} \sigma_{1}^{2} \bigl(\Tr^{3}{(\mathbf{O}_{11}^{2} )} - 3 \Tr{(\mathbf{O}_{11}^{2} )} \Tr{(\mathbf{O}_{11}^{4} )} + 2 \Tr{(\mathbf{O}_{11}^{6} )}\bigr) + \sum_{j=2}^{4} l_{j}^{2} \sigma_{j}^{2} \biggl(\Tr^{3}{({\mathbf{O}^\prime_{j1}}^{2} )} - 3 \Tr{({\mathbf{O}^\prime_{j1}}^{2} )} \Tr{({\mathbf{O}^\prime_{j1}}^{4} )} \\& \qquad+ 2 \Tr{({\mathbf{O}^\prime_{j1}}^{6} )} + \Tr^{3}{(\mathbf{O}_{j1}^{2} )} - 3 \Tr{(\mathbf{O}_{j1}^{2} )} \Tr{(\mathbf{O}_{j1}^{4} )} + 2 \Tr{(\mathbf{O}_{j1}^{6} )}\biggr) \biggr]\\
    p_1 =& \frac{1}{6}\biggl[3 l_{1}^{2} \sigma_{1}^{2} \bigl(\Tr^{2}{(\mathbf{O}_{11}^{2} )} - \Tr{(\mathbf{O}_{11}^{4} )}\bigr) + \sum_{j=2}^{4} l_{j}^{2} \sigma_{j}^{2} \biggl(- 3 \bigl(\Tr^{2}{(\mathbf{O}^\prime_{j1} )}  - \Tr{({\mathbf{O}^\prime_{j1}}^{2} )}\bigr) \bigl(\Tr^{2}{(\mathbf{O}_{j1} )} - \Tr{(\mathbf{O}_{j1}^{2} )}\bigr) + 3 \Tr^{2}{({\mathbf{O}^\prime_{j1}}^{2} )}\\& + 2 \Tr{(\mathbf{O}^\prime_{j1} )}\bigl(\Tr^{3}{(\mathbf{O}_{j1} )} - 3 \Tr{(\mathbf{O}_{j1} )} \Tr{(\mathbf{O}_{j1}^{2} )} + 2 \Tr{(\mathbf{O}_{j1}^{3} )}\bigr)    - 3 \Tr{({\mathbf{O}^\prime_{j1}}^{4} )}  + 3 \Tr^{2}{(\mathbf{O}_{j1}^{2} )}  - 3 \Tr{(\mathbf{O}_{j1}^{4} )}\biggr)\biggr] \\
    p_2 =& l_{1}^{2} \sigma_{1}^{2} \Tr{(\mathbf{O}_{11}^{2} )} + \sum_{j=2}^{4} l_{j}^{2} \sigma_{j}^{2} \bigl(\Tr^{2}{(\mathbf{O}^\prime_{j1} )} + \Tr^{2}{(\mathbf{O}_{j1} )} - 2 \Tr{(\mathbf{O}^\prime_{j1} )} \Tr{(\mathbf{O}_{j1} )}\bigr)  \\
    p_3 =& l_{1}^{2} \sigma_{1}^{2} \label{eq:rational_fun_4D_num}
\end{split}
\raisetag{2\baselineskip} 
\end{equation}
\end{widetext}

\section{Recursive solutions}
\label{sec:recursive-sol}
We derive an alternative approach to computing the polynomial coefficients of the auto- and cross-spectrum based on a recursive algorithm. By this method, we can obtain coefficients for all entries of the power spectral matrix simultaneously in $O(n^4)$ operations (\textit{Supplemental Material, Fig.~S1}). To make our solution compact, we denote the noise covariance matrix $\mathbf{C} \equiv \mathbf{L} \, \mathbf{D} \,\mathbf{L}^{\top}$. We specify the algorithm in the following theorem and provide a proof in \textit{Supplemental Material}. Note that the algorithm can alternatively be derived by leveraging the Barnett-Leverrier recursive formula \cite{barnett1989leverrier}.

\begin{theorem}
\label{theorem:0}
The noise power spectral density matrix of an LTI system is a complex-valued rational function of the form
\begin{equation*}
\begin{split}
    \mathbfcal{S}(\omega) &=(\dot{\iota}\omega\I+\J)^{-1} \, \mathbf{C} \, (-\dot{\iota}\omega\I+\J)^{-\top} \\
      &= \frac{\sum\limits_{\alpha=0}^{n-1}\Pb_\alpha \omega^{2\alpha} + \dot{\iota}\omega \sum\limits_{\alpha=0}^{n-2}\Pb_\alpha^\prime \omega^{2\alpha}}{\sum\limits_{\alpha=0}^{n}q_\alpha \omega^{2\alpha}} \label{eq:recursive3}
\end{split}
\end{equation*}
The numerator's matrix coefficients are given by the recursive equations,
\begin{equation}
\begin{split}
    \Pb_{\alpha-1}^\prime &= \J \Pb_{\alpha} - \Pb_{\alpha} \J^\top - \J \Pb_{\alpha}^\prime \J^\top \\
    \Pb_{\alpha-1} &= q_{\alpha}\mathbf{C} + \Pb_{\alpha-1}^\prime \J^\top - \J \Pb_{\alpha-1}^\prime - \J \Pb_{\alpha} \J^\top
\end{split}
\label{eq:num-rec-coefficients}
\end{equation}
for $\alpha \in \{1,\dots, n\}$, starting from $\alpha =n$ with $\Pb_{n}=\Pb_{n}^\prime=\mathbf{0}$. The scalar coefficients of the denominator can be recursively calculated using,
\begin{equation}
\begin{split}
    q_{\alpha} = \frac{1}{n-\alpha}\left[\Tr\left(\J \Q_{\alpha-1}^\prime \right) + \Tr\left(\J \Q_\alpha \J^\top \right)\right]
\end{split}
\end{equation}
if $\alpha<n$, and $q_n = 1$. The coefficients $\Q_{\alpha-1}^\prime$ and $\Q_{\alpha}$ are given in turn by the recursive equations,
\begin{equation}
\begin{split}
    \Q_{\alpha-1}^\prime &= \J \Q_{\alpha} - \Q_{\alpha} \J^\top - \J \Q_{\alpha}^\prime \J^\top \\
    \Q_{\alpha-1} &= q_{\alpha} \mathbf{I}+ \Q_{\alpha-1}^\prime \J^\top - \J \Q_{\alpha-1}^\prime - \J \Q_{\alpha} \J^\top
\end{split}
\label{eq:suppl-rec-coefficients}
\end{equation}
for $\alpha \in \{0,1, \dots, n\}$, starting from $\alpha=n$ with $\Q_n = \Q_n^\prime =\mathbf{0}$ and $ \Q_{-1} = \Q_{-1}^\prime =\mathbf{0}$.
\end{theorem}

Note that Eq.~\ref{eq:suppl-rec-coefficients} is the same as Eq.~\ref{eq:num-rec-coefficients} for $\mathbf{C} = \mathbf{I}$, and that the $\Pb_{\alpha}^\prime$ ($\Q_{\alpha}^\prime$) matrices are anti-symmetric, whereas the $\Pb_{\alpha}$ ($\Q_{\alpha}$) matrices are symmetric.

The approach used to derive the recursive solution for $\mathbfcal{S}(\omega)$ is broadly applicable to matrices of similar forms. In \textit{Appendix~\ref{sec:appendix_recursive}}, we provide an example of one such recursive algorithm for the similarly defined matrix
\begin{equation}
    \mathbfcal{T}(\omega) = (\omega\I-\G)^{-1} \, \Y \, (\omega\I-\G^\top)^{-1}
\end{equation}
We use this result to prove the following Theorem (proof in \textit{Supplemental Material}),

\begin{theorem}
\label{theorem:without_omega}
Let $\mathbf{G}$ be a convergent matrix, i.e., with a spectral radius less than 1, $\mathbf{Y}$ be a positive definite matrix, and $\mathbf{R}$ be a matrix with the following series expansion:
\begin{equation}
\begin{split}
    \mathbf{R} &= \left(\I - \mathbf{G} \right)^{-1} \mathbf{Y} \left(\I - \mathbf{G}^\top\right)^{-1} \\
    &= \sum_{m = 0}^\infty \sum_{\substack{j+k = m \\ j\geq0, \, k\geq 0}}\G^j \, \Y \, \left(\G^\top\right)^{k}  \\
\end{split}
\end{equation}
$\Ss_m \coloneqq \sum_{\substack{j+k = m \\ j\geq0, \, k\geq 0}}\G^j \, \Y \, \left(\G^\top\right)^{k}$ can be found by the following recursive algorithm,
\begin{equation}
\begin{split}
    t_{m} &= \frac{2}{m}\left[\Tr\left(\Y^{-1} \G \Ub_{m-2} \G^\top \right) - \Tr\left(\Y^{-1}\G \Ub_{m-1}  \right)\right] \\
    \Ub_{m} &= t_{m} \Y + \G \Ub_{m-1} + \Ub_{m-1} \G^\top -\G \Ub_{m-2} \G^\top \\
    \Ss_{m} &= \Ub_{m} - \sum_{k=0}^{m-1} t_{m-k} \, \Ss_k     
\end{split}
\raisetag{3\baselineskip} 
\end{equation}
starting with $m=0$ until $m=2n$ and the initial values $t_0=1$ and $\Ub_{-2} = \Ub_{-1} = \mathbf{0}$.
\end{theorem}

We will use this result in Sec.~\ref{sec:hawkes_model} to recursively find the terms of various orders in the power series expansion of the integrated covariance matrix of a network of spiking neurons modeled by Hawkes processes \cite{hawkes1971point}.

\section{Applications}
\label{sec:application}
We now apply the rational function solutions of sections  \ref{sec:element-wise} and \ref{sec:recursive-sol} to compute the power spectra of a range of nonlinear models exhibiting fixed point solutions and, where appropriate, we provide analytical closed-form expressions. We consider four models: (i) The stochastic Hindmarsh-Rose model simulating the subthreshold activity of a neuron subject to an input current \cite{hindmarsh1984model, storace2008hindmarsh}; (ii) a 4D Wilson-Cowan model simulating the activity of excitatory and inhibitory neuronal populations in the cortex \cite{wilson1972excitatory, keeley2019firing}; (iii) the stochastic Stabilized Supralinear Network (SSN) model \cite{rubin2015stabilized, ahmadian2013analysis}, which is a neural circuit model for the sensory cortex approximating divisive normalization \cite{heeger1992normalization, carandini2012normalization}; (iv) an evolutionary game-theoretic model which involves a population of agents playing a generalized version of the rock-paper-scissors game in $n$ dimensions \cite{toupo2015nonlinear}. Additionally, we apply the recursive solution described in Theorem~\ref{theorem:without_omega} to analyze the correlation structure of a network of spiking neurons, in random and ring networks, modeled by the Hawkes process \cite{hawkes1971point}.

\subsection{Hindmarsh-Rose}

Hindmarsh-Rose \cite{hindmarsh1984model, storace2008hindmarsh} is a phenomenological 3D neuron model that can be considered a simplification of the Hodgkin-Huxley equations \cite{hodgkin1952quantitative} or a generalization of the Fitzhugh-Nagumo model \cite{fitzhugh1961impulses, yamakou2019stochastic}. We consider a stochastic version of the model based on the implementation by Storace et al. \cite{storace2008hindmarsh}. The system of SDEs defining the model is given by
\begin{equation}
\begin{split}
    \frac{\mathrm{d}x}{\mathrm{d}t} &= y - x^3 + bx^2 + I - z + \sigma \eta_1 \\
    \frac{\mathrm{d}y}{\mathrm{d}t} &= 1 - 5x^2 - y\\
    \frac{\mathrm{d}z}{\mathrm{d}t} &= \mu \bigl( s(x - x_{rest}) - z\bigr) \label{eq:hr}
\end{split}
\end{equation}
Here, $x$ represents the membrane potential of the neuron, $y$ is a recovery variable, and $z$ is a slow adaptation variable. $I$ represents the synaptic current input to the neuron; $\mu \ll 1$ controls the relative timescales between the fast subsystem $(x, y)$ and the slow $z$ variable; $x_{rest}$ is the resting state of the membrane potential of the neuron; $s$ governs adaptation, and $b$ controls the transition of the activity to different dynamical states (quiescent, spiking, regular bursting, irregular bursting and so on) \cite{storace2008hindmarsh}. We add Gaussian white noise, $\sigma \eta_1$, with standard deviation $\sigma$, to the membrane potential and choose the model parameters so that the system is in a quiescent (subthreshold) state with a single stable equilibrium, namely $I=5.5$, $b=0.5$, $\mu=0.01$, $x_{rest}=-1.6$, $s=4$ and $\sigma=0.001$ \footnote{We choose parameters in the pale cyan region of the bifurcation diagram in Fig.~1 of \cite{storace2008hindmarsh}}.

Since the system is in a quiescent state, it exhibits a single fixed-point solution. Assuming that the deterministic steady-state is given by the triplet $(x_e, y_e, z_e)$, we can linearize the system about the fixed point and write Eq.~\ref{eq:hr} as the LTI system
\begin{equation}
\begin{bmatrix}
    \dot{x}\\
    \dot{y}\\
    \dot{z}
    \end{bmatrix} = 
    \begin{bmatrix}
    2 b x_e-3 x_e^2 & 1 & -1 \\
    -10 x_e & -1 & 0 \\
    \mu  s & 0 & -\mu  \\
    \end{bmatrix}
    \begin{bmatrix}
     x - x_e\\
     y - y_e\\
     z - z_e
    \end{bmatrix} + 
    \begin{bmatrix}
    \sigma\eta_1\\
    0 \\
    0 
    \end{bmatrix}
\end{equation}

\begin{figure}[bt!]
    \centering
    \includegraphics[width=\linewidth]{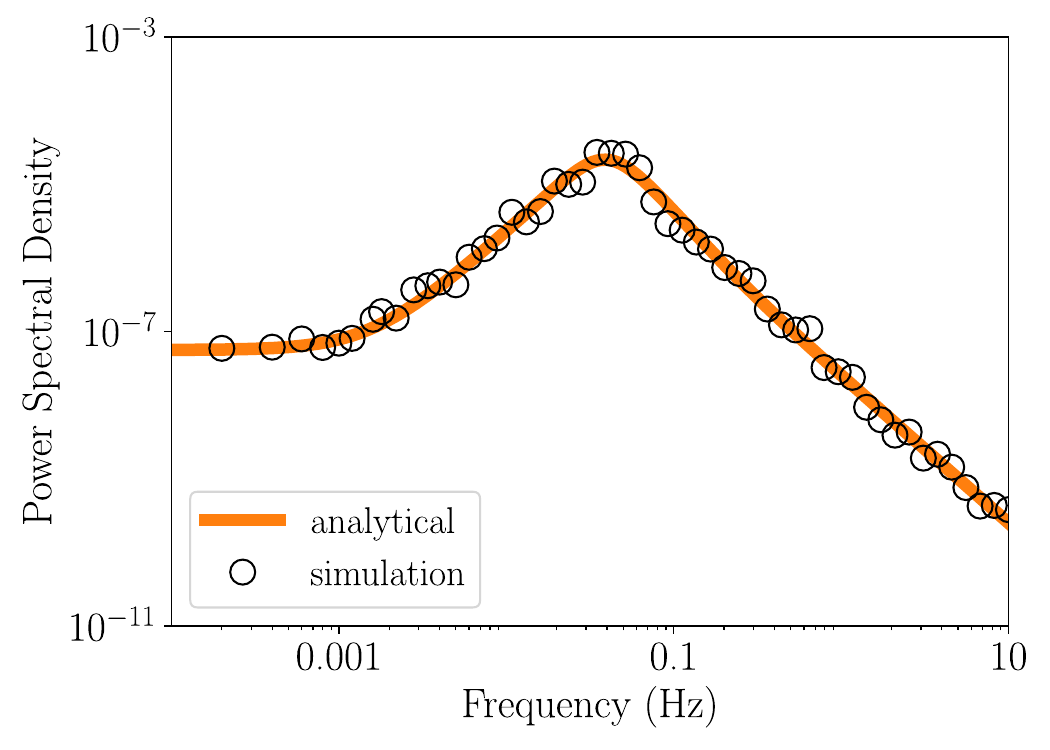}
    \caption{\textbf{Power spectrum of the Hindmarsh-Rose model.} The power spectral density of the membrane potential $x$, is calculated via simulation and our analytical (element-wise) solution. We use the following parameters to simulate the model: $I=5.5$, $b=0.5$, $\mu=0.01$, $x_{rest}=-1.6$, and $s=4$. We consider the case of additive Gaussian white noise with standard deviation $\sigma=0.001$.}
    \label{fig:psd_hr}
\end{figure}
To show how our element-wise solution leads to a closed-form expression for the power spectrum of the membrane potential, we use the general analytical solution for a 3D system given by Eq.~\ref{eq:rational_fun_3D},
\begin{equation*}
\mathcal{S}_{x}(\omega) = \frac{Z(\omega)}{Q(\omega)} =  \frac{p_0 + p_1 \omega^2 + p_2 \omega^4}{q_0 + q_1 \omega^2 + q_2 \omega^4 + q_3 \omega^6}.
\end{equation*}
From the application of Eq.~\ref{eq:ration_denominator_compiled}, we obtain the coefficients of the denominator, $Q(\omega)$, as
\begin{equation}
\begin{split}
    q_0 =& \mu ^2 (x_e (3 x_e -2 b +10)+s)^2\\
    q_1 =& \mu ^2 \left((x_e (3 x_e-2 b)+s)^2-20 x_e+1\right)\\
    &+x_e^2 (3 x_e - 2 b +10)^2-2 \mu  s+20 \mu  s x_e \\
    q_2 =& x_e \left(x_e (3 x_e - 2 b)^2-20\right)+\mu ^2-2 \mu  s+1 \\
    q_3 =& 1
\end{split}
\end{equation}
Similarly, the coefficients of the numerator, $Z(\omega)$, can be found from Eq.~\ref{eq:rational_fun_3D_num}: using the fact that $\sigma_1=\sigma_2=\sigma_3=\sigma$, $l_1=1$, and $ l_2 = l_3 = 0$, we obtain
\begin{equation}
\begin{split}
    p_0 =& \frac{1}{2} \sigma^{2} \bigl(\Tr^{2}{(\mathbf{O}_{11}^{2} )} - \Tr{(\mathbf{O}_{11}^{4} )})\\
    p_1 =& \sigma^{2} \Tr{(\mathbf{O}_{11}^{2} )} \\
    p_2 =& \sigma^{2}
\end{split}
\end{equation}
Now, plugging $\mathbf{O}_{11}$ given by
\begin{equation}
    \mathbf{O}_{11} = 
    \begin{bmatrix}
    -1 & 0 \\
    0 & -\mu 
    \end{bmatrix}
\end{equation}
into the equation above, yields the polynomial expansion of the numerator as,
\begin{equation}
    Z(\omega) = \mu ^2 \sigma ^2 + \left(\mu ^2+1\right) \sigma ^2 \omega^2 + \sigma^{2} \omega^4
\end{equation}

We simulate the 3D system and verify that the simulated spectrum matches with the analytical element-wise solution we just derived, shown in Fig.~\ref{fig:psd_hr}. Furthermore,  we calculate and plot the power spectral density for the membrane potential $x$, and the absolute cross-power spectral density and coherence between the fast variables $x$ and $y$, using Welch\textsc{\char13}s method \cite{welch1967use} (simulation), the recursive and element-wise rational function solutions, and the ``matrix solution'' given by Eq.~\ref{eq:main_mat_sol} (\textit{n.b.}, this relies on a numerical inverse). We confirm that all the solutions match across the entire range of frequencies, see \textit{Supplemental Material, Fig.~S3}.

\subsection{Wilson-Cowan model}
\begin{figure}[b!]%[ht!]
    \centering
    \includegraphics[width=0.5\linewidth]{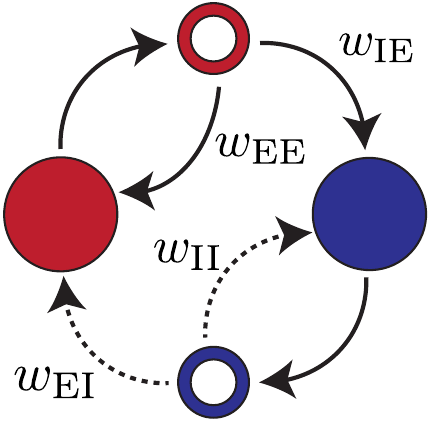}
    \caption{\textbf{Wilson-Cowan 4D model.} The mean-field circuit consists of 4 variables modeling excitatory activity (solid red), inhibitory activity (solid blue), excitatory synaptic activity (hollow red), and inhibitory synaptic activity (hollow blue). The solid and dashed arrows indicate excitatory and inhibitory connections, respectively. $w_{\mathrm{ij}}$, with $i,j\in\{\mathrm{E},\mathrm{I}\}$, represents the weight associated with the connections from $\mathrm{j}$ to $\mathrm{i}$.} 
    \label{fig:wc4d_model}
\end{figure}

The Wilson-Cowan model is a rate-based model that simulates the activity of Excitatory (E) and Inhibitory (I) populations of neurons and has been shown to exhibit gamma oscillations (30-120 Hz) for an appropriate choice of time constants \cite{wilson1972excitatory}. Specifically, we consider a stochastic 4-dimensional version of the model introduced by Keeley et al. \cite{keeley2019firing}, which includes dynamical variables describing the synaptic activation and the firing rate of the E and I populations, see Fig.~\ref{fig:wc4d_model} for an illustration. The stochastic dynamical system is defined by the following set of equations,
\begin{equation}
\begin{split}
    \tau_\mathrm{E} \frac{\mathrm{d}r_\mathrm{E}}{\mathrm{d}t} &= -r_\mathrm{E} + f((I_\mathrm{E} + w_{\mathrm{EE}}s_\mathrm{E} - w_{\mathrm{EI}}s_\mathrm{I} - \theta_\mathrm{E})/\kappa_\mathrm{E}) + \sigma_r \eta_1\\
    \tau_\mathrm{I} \frac{\mathrm{d}r_\mathrm{I}}{\mathrm{d}t} &= -r_\mathrm{I} + f((I_\mathrm{I} + w_{\mathrm{IE}}s_\mathrm{E} -w_{\mathrm{II}}s_\mathrm{I} - \theta_\mathrm{I})/\kappa_\mathrm{I}) + \sigma_r \eta_2\\
    \tau_{s_\mathrm{E}} \frac{\mathrm{d}s_\mathrm{E}}{\mathrm{d}t} &= -s_\mathrm{E} + \gamma_\mathrm{E} r_\mathrm{E} (1-s_\mathrm{E}) + s_\mathrm{E}^0 + \sigma_s \eta_3\\
    \tau_{s_\mathrm{I}} \frac{\mathrm{d}s_\mathrm{I}}{\mathrm{d}t} &= -s_\mathrm{I} + \gamma_\mathrm{I} r_\mathrm{I} (1-s_\mathrm{I}) + s_\mathrm{I}^0 +  \sigma_s \eta_4
\end{split}
\raisetag{3\baselineskip} 
\end{equation}
\begin{figure*}[tb!]
    \centering
    \includegraphics[width=\linewidth]{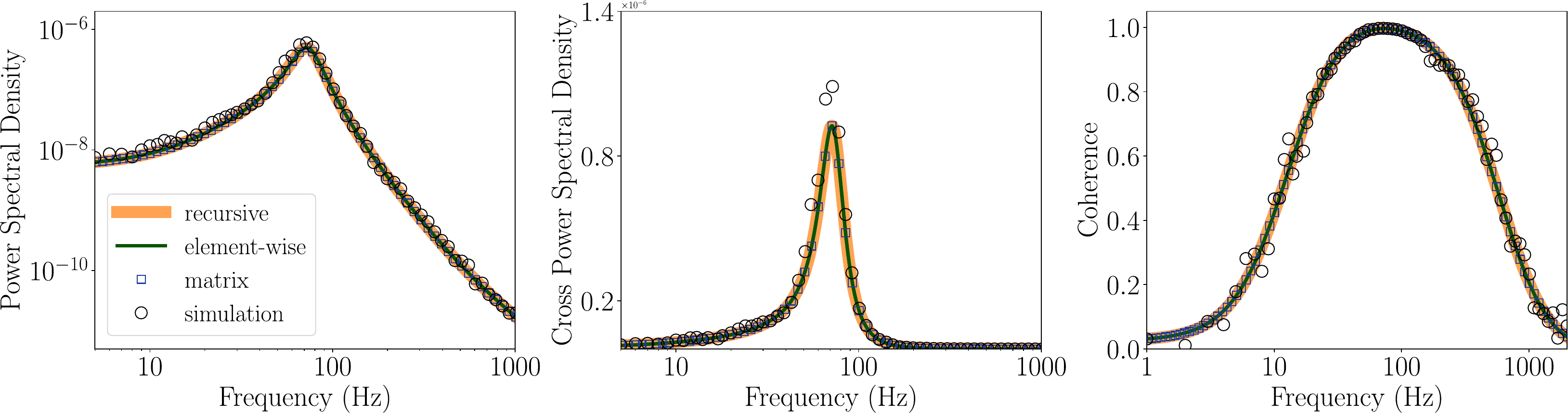}
    \caption{\textbf{Spectral density (auto and cross) and coherence for the Wilson-Cowan model.} We plot the power spectral density (left), the absolute cross-power spectral density (middle), and the coherence (right) for a given stimulus, each calculated via simulation, the rational function solutions (recursive and element-wise), and the matrix solution (Eq.~\ref{eq:main_mat_sol}). The PSD is calculated for the firing rate of the excitatory population. The cross-PSD and coherence are calculated between the firing rates of the excitatory and inhibitory populations. We use the following parameters for simulations and analytical calculations. Time constants: $\tau_\mathrm{E} = 2$ms, $\tau_\mathrm{I} = 8$ms, $\tau_{s_\mathrm{E}} = 10$ms and $\tau_{s_\mathrm{I}} = 10$ms. Strength of the weights: $w_{\mathrm{EE}}=5$, $w_{\mathrm{EI}}=5$, $w_{\mathrm{IE}}=3.5$ and $w_{\mathrm{II}}=3$. Offset of the input: $\theta_\mathrm{E}=0.4$ and $\theta_\mathrm{I}=0.4$. Scaling of the input: $\kappa_\mathrm{E}=0.2$ and $\kappa_\mathrm{I}=0.02$. Scaling of the population synaptic activity: $\gamma_\mathrm{E}=1$ and $\gamma_\mathrm{I}=2$. External inputs to the system: $I_\mathrm{E}=1$, $I_\mathrm{I}=0.5$, $s_\mathrm{E}^0=0.2$ and $s_\mathrm{I}^0=0.05$. We consider additive Gaussian white noise with standard deviations $\sigma_r=0.001, \text{ and } \sigma_s=0.002$. }
    \label{fig:psd_wc}
\end{figure*}
These equations describe the activity of E and I neurons at a population level, with $r_\mathrm{E}$ and $r_\mathrm{I}$ representing their firing rates, and $s_\mathrm{E}$ and $s_\mathrm{I}$ representing the corresponding synaptic activations. $w_{\alpha \beta}$, where $\alpha$ and $\beta \in \{\mathrm{E}, \mathrm{I}\}$, are the synaptic weights, and $f(x) = 1/(1+e^{-x})$ is a sigmoid nonlinearity.

The input to the sigmoid is a weighted sum of the synaptic responses plus the external drives $I_\alpha$, which are offset by $\theta_\alpha$ and scaled by $\kappa_\alpha$. The synaptic activity ($s_\alpha$) for E and I populations evolves with time and depends on the firing rate of the corresponding population ($r_\alpha$), scaled by $\gamma_\alpha$, as well as on the corresponding background synaptic inputs $s_{\alpha}^0$. The time constants controlling the rise/decay times are $\tau_\alpha$. The variables evolve stochastically on account of the uncorrelated Gaussian additive white noise, $\sigma\eta_i$, with standard deviations $\sigma_r$ and $\sigma_s$, added to the firing rates and the synaptic activity, respectively.

We perform a simulation of the 4-dimensional system and generate power spectral density plots for the firing rate of the E population, $r_\mathrm{E}$, the absolute cross-power spectral density between the firing rate of the E and I populations, and the coherence 
\begin{equation}
    K_{ij} = \frac{|\mathcal{S}_{ij}|^2}{|\mathcal{S}_{ii}||\mathcal{S}_{jj}|} 
\end{equation}
between the firing rates of the E and I populations using Welch\textsc{\char13}s method \cite{welch1967use} for simulation, the recursive and element-wise rational function solutions, and the matrix solution (Eq.~\ref{eq:main_mat_sol}). Our results demonstrate that all the solutions match across the entire range of frequencies, as confirmed in Fig.~\ref{fig:psd_wc}.

\subsection{Stabilized Supralinear Network (SSN)}
\begin{figure}[b!]
    \centering
    \includegraphics[width=0.9\linewidth]{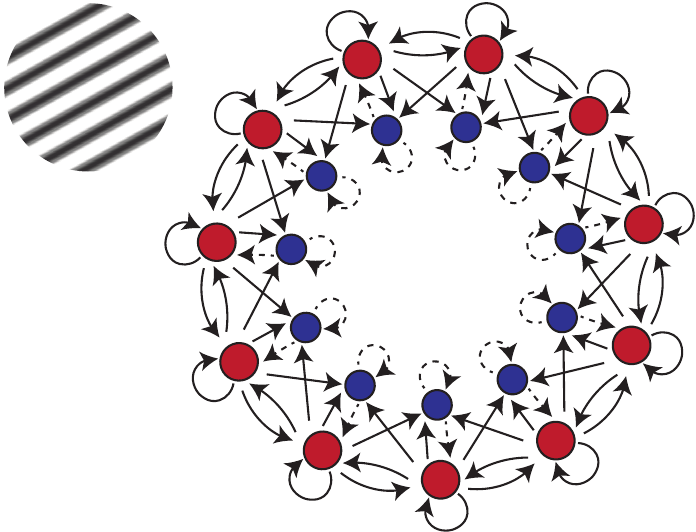}
    \caption{\textbf{SSN circuit.} The circuit consists of 11 excitatory (red) and 11 inhibitory neurons (blue) and receives a sinusoidal grating input shown on the left. The arrows schematize the connections between the different neurons. E-E and I-E connections are excitatory (solid arrows) fully connected, whereas E-I and I-I connections are self-inhibitory (dashed arrows).} 
    \label{fig:ssn_model}
\end{figure}
\begin{figure*}[tb!]
    \centering
    \includegraphics[width=\linewidth]{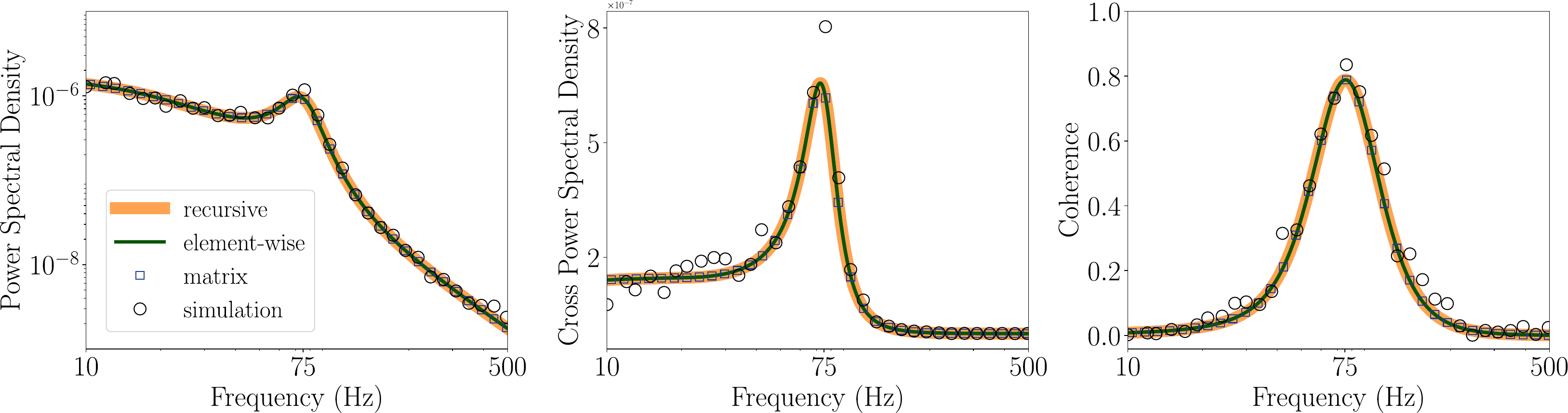}
    \caption{\textbf{Spectral density (auto and cross) and coherence for the SSN circuit.} We plot the power spectral density (left), the absolute cross-power spectral density (middle), and the coherence (right) for a given stimulus each calculated via simulation, the rational function solutions (recursive and element-wise), and the matrix solution (Eq.~\ref{eq:main_mat_sol}). The PSD is calculated for the excitatory neuron with the highest preference for the stimulus (i.e., exhibiting maximum firing rate). The cross-PSD and coherence are calculated between the two excitatory neurons with the highest firing rates. We use the following parameters for simulating the model: Contrast $c=50$, stimulus length $l=9$, $\Delta x = 3.0$, $\sigma_{RF}=0.125\Delta x$. Time constants of $\tau_\mathrm{E}=6$ms and $\tau_\mathrm{I}=4$ms. Supralinear activation parameters $n=2.2$ and $k=0.01$. See \textit{Supplemental Material} for the parameters associated with the weight matrices. We add zero-mean Gaussian noise to the SDEs describing the activity of the E and I neurons with standard deviation $\sigma = 0.01$.} 
    \label{fig:psd_ssn}
\end{figure*}
%, $J_{ \mathrm{ \mathrm{EE}}} = 2.0$, $J_{ \mathrm{IE}} = 2.25$, $J_{ \mathrm{EI}} = 0.9$, $J_{ \mathrm{II}} = 0.5$, $\sigma_{ \mathrm{EE}}=4$ and $\sigma_{ \mathrm{IE}}=8$.
The Stabilized Supralinear Network (SSN) \cite{rubin2015stabilized, ahmadian2013analysis} is a neural circuit model reproducing divisive normalization approximately \cite{heeger1992normalization, carandini2012normalization} that has been used to simulate neuronal activity in the sensory cortex \cite{ahmadian2013analysis, rubin2015stabilized, hennequin2018dynamical, kraynyukova2018stabilized, holt2023stabilized}. We implement an SSN for simulating the activity of the primary visual cortex (V1) as described in \cite{rubin2015stabilized}. We consider a 1-dimensional spatial input which is a sinusoidal grating of a fixed orientation and a given contrast (\textit{n.b.}, the response of V1 neurons is tuned to specific positions, orientations, size, etc., of the visual stimulus \cite{kandel2000principles}) and varies with Michelson contrast, $c=(I_{\max}-I_{\min})/(I_{\max}+I_{\min})$, where $I_{\min, \max}$ are the minimum and maximum luminance in the stimulus.  The input evaluated at a set of discrete points $\mathbf{x}$, is taken to be proportional to the contrast, $\mathbf{h}(\mathbf{x})\propto c$, and decays with distance from the origin (see \textit{Supplemental Material} for the precise functional form). We operate in a regime where the circuit exhibits a fixed point solution with damped oscillations (for parameters see Fig.~\ref{fig:psd_ssn} and \textit{Supplemental Material}).

The dynamical equations for the firing rates of the excitatory neurons, $\mathbf{r}_\mathrm{E}(\mathbf{x})$, and the inhibitory neurons, $\mathbf{r}_\mathrm{I}(\mathbf{x})$, are given by
\begin{equation}
\begin{split}
    \tau_\mathrm{E} \frac{\mathrm{d}\mathbf{r}_\mathrm{E}(\mathbf{x})}{\mathrm{d}t} &= -\mathbf{r}_\mathrm{E}(\mathbf{x}) + \mathbf{r}_{\mathrm{E}}^{ss}(\mathbf{x}) + \sigma \mathbf{\eta}_\mathrm{E}(\mathbf{x})\\
    \tau_\mathrm{I} \frac{\mathrm{d}\mathbf{r}_\mathrm{I}(\mathbf{x})}{\mathrm{d}t} &= -\mathbf{r}_\mathrm{I}(\mathbf{x}) + \mathbf{r}_{\mathrm{I}}^{ss}(\mathbf{x}) + \sigma \mathbf{\eta}_\mathrm{I}(\mathbf{x})
\end{split}
\end{equation}
where $\tau_\mathrm{E}$ and $\tau_\mathrm{I}$ are the time constants of the excitatory and inhibitory neurons, $\sigma\mathbf{\eta}_\mathrm{E}(\mathbf{x})$ and $\sigma \mathbf{\eta}_\mathrm{I}(\mathbf{x})$ are zero-mean uncorrelated Gaussian noise vectors with standard deviation $\sigma$, and $\mathbf{r}_{\mathrm{E}}^{ss}$ and $\mathbf{r}_{\mathrm{I}}^{ss}$ are the steady-state responses of the excitatory and inhibitory neurons, given by the equations
\begin{equation}
\begin{split}
    \mathbf{r}_{\mathrm{E}}^{ss}(\mathbf{x}) &= k([\mathbf{h}(\mathbf{x}) + \mathbf{W}_{\mathrm{EE}} \mathbf{r}_\mathrm{E} (\mathbf{x}) -\mathbf{W}_{\mathrm{EI}} \mathbf{r}_\mathrm{I} (\mathbf{x})]_+)^n \\
    \mathbf{r}_{\mathrm{I}}^{ss}(\mathbf{x}) &= k([\mathbf{h}(\mathbf{x}) + \mathbf{W}_{\mathrm{IE}} \mathbf{r}_\mathrm{E} (\mathbf{x}) -\mathbf{W}_{\mathrm{II}} \mathbf{r}_\mathrm{I} (\mathbf{x})]_+)^n
\end{split}
\end{equation}
Here, $[z]_+= \max(z,0)$ denotes a ReLU activation. In the SSN, the steady-state responses depend on a supralinear activation ($n=2.2$ and $k=0.01$) of the input to the neurons, which is a sum of the stimulus input and feedback from the neurons in the population. The strength of the connections between E-E (excitatory to excitatory) and I-E (inhibitory to excitatory) neurons are encoded in the weight matrices $\mathbf{W}_{\mathrm{EE}}$ and $\mathbf{W}_{\mathrm{IE}}$, which decay like Gaussians as a function of distance from the preferred stimulus. The E-I and I-I connections are not fully-connected but self-inhibitory, with diagonal weight matrices $\mathbf{W}_{\mathrm{EI}}$ and $\mathbf{W}_{\mathrm{II}}$, respectively.

We simulate a 22-dimensional system, see Fig.~\ref{fig:ssn_model} for an illustration. In Fig.~\ref{fig:psd_ssn}, we show the power spectral density for the excitatory neuron with the highest preference for the stimulus (viz., that is maximally firing); the absolute cross-power spectral density and the coherence between the two excitatory neurons with the highest firing rates, each using Welch\textsc{\char13}s method \cite{welch1967use} (simulation), our recursive and element-wise rational function solutions, and the matrix solution (Eq.~\ref{eq:main_mat_sol}). We confirm that all the solutions match across the entire range of frequencies.

\subsection{Rock-Paper-Scissors}
The game of Rock-Paper-Scissors, where rock smashes scissors, scissors cut paper, and paper wraps rock has been employed by evolutionary biologists to investigate the interactions between competing populations of species, where one species has an advantage over only one of its opponents. The game dynamics that describe the evolution of the populations following a particular strategy can be examined with the use of evolutionary game theory \cite{toupo2015nonlinear}. We explore the evolution of this dynamical system in accordance with a sociological model \cite{bomze1983lotka, hofbauer2003evolutionary}, where each population, following a unique strategy, seeks to minimize the difference between the average payoff of the population and the individual payoff.

Mathematically, we can represent the system as follows: Let $\mathbf{x} = [x_1, x_2, ..., x_n]$ be the $n$-dimensional system comprising the fractional populations corresponding to different strategies. The sociological model dictates that these fractional populations evolve according to a metric given by $\dot{x_i} \propto (f_{x_i}-\phi)$, where $f_{x_i}$ is the individual payoff for choosing a strategy corresponding to the population $x_i$, and $\phi$ is the average payoff of the population given by $\phi = \sum_{i =1}^{n} x_i f_{x_i}$. It is important to note that the dynamical system is effectively $n-1$ dimensional due to the mass conservation equation $\sum_{i=1}^{n} x_i=1$. 

\begin{figure}[t]
    \centering
    \includegraphics[width=0.75\linewidth]{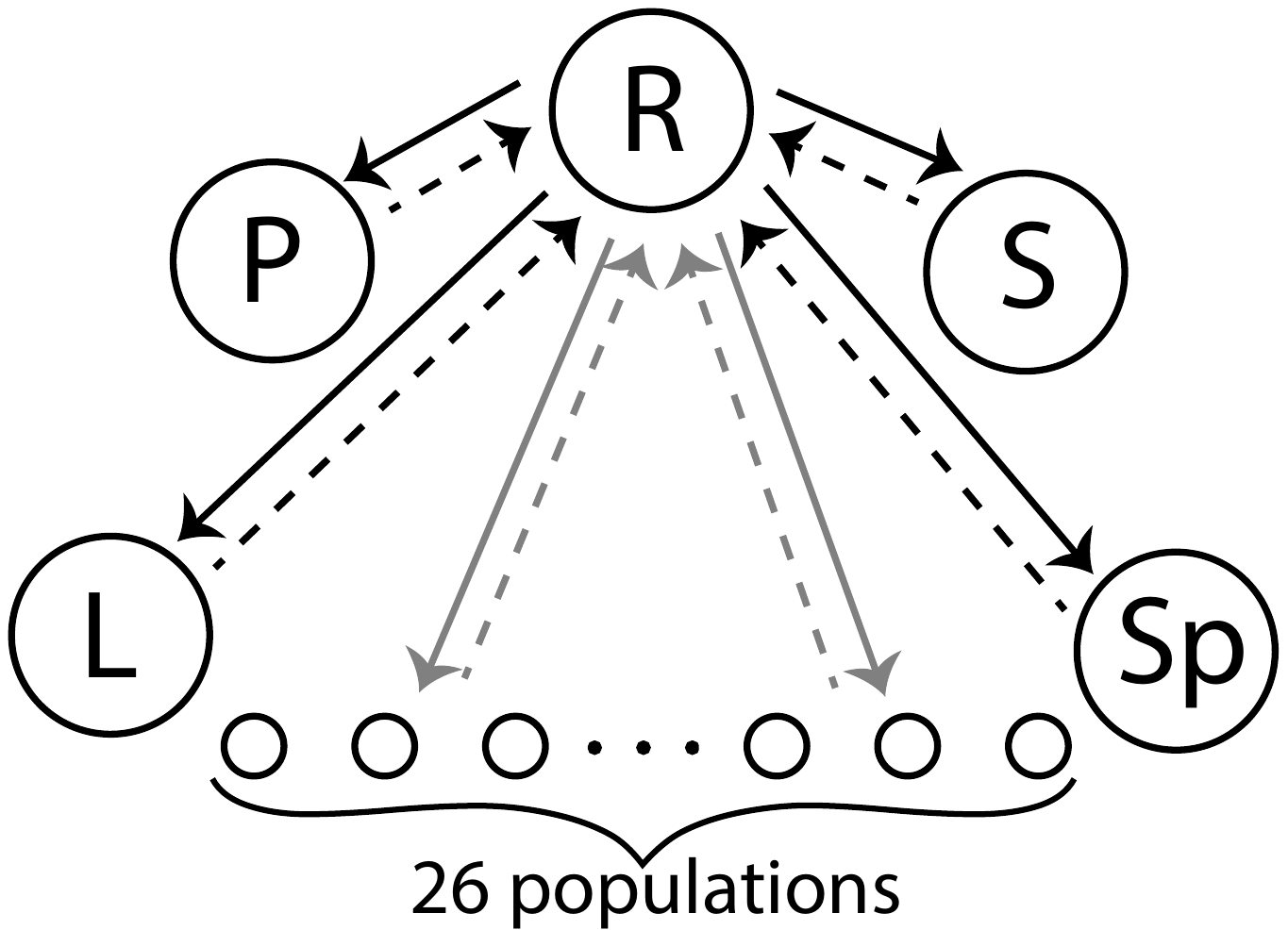}
    \caption{\textbf{31-D Rock-Paper-Scissors model.} Circles represent the different strategies ($n=31$) in the evolutionary game theoretic model. Shown in the figure are 5 of the 31 strategies: Rock (R), paper (P), Scissors (S), Spock (Sp), and Lizard (L). We consider a model with global mutations, i.e., each population associated with a strategy mutates into every other population with the same rate $\mu$. Depicted in the figure are mutations to (dashed arrows) and from (solid arrows) the Rock population.} 
    \label{fig:rps_model}
\end{figure}

We consider a 31-dimensional ($n=31$) version of the Rock-Paper-Scissors game, consisting of 31 populations each following a unique strategy, described by the pay-off matrix, $\mathbf{P}$. For an $n$-dimensional version of the game, the payoff matrix ($P_{ij}$) can be summarized as,
\begin{equation}
    P_{ij} = 
    \begin{cases}
        (-1)^{i+j+1}  & i > j\\ 
        0 & i = j\\
        (-1)^{i+j}  & i < j.
    \end{cases} \label{eq:payoff}
\end{equation}
Here, 1 indicates a win against the opponent, $-1$ indicates a loss and 0 indicates neither a win nor a loss (a draw). Therefore, the individual payoff for a given strategy is given in terms of $\mathbf{P}$ as, $f_{x_i} = \sum_{j=1}^n P_{ij} x_j$.

\begin{figure}[t!]
    \centering
    \includegraphics[width=\linewidth]{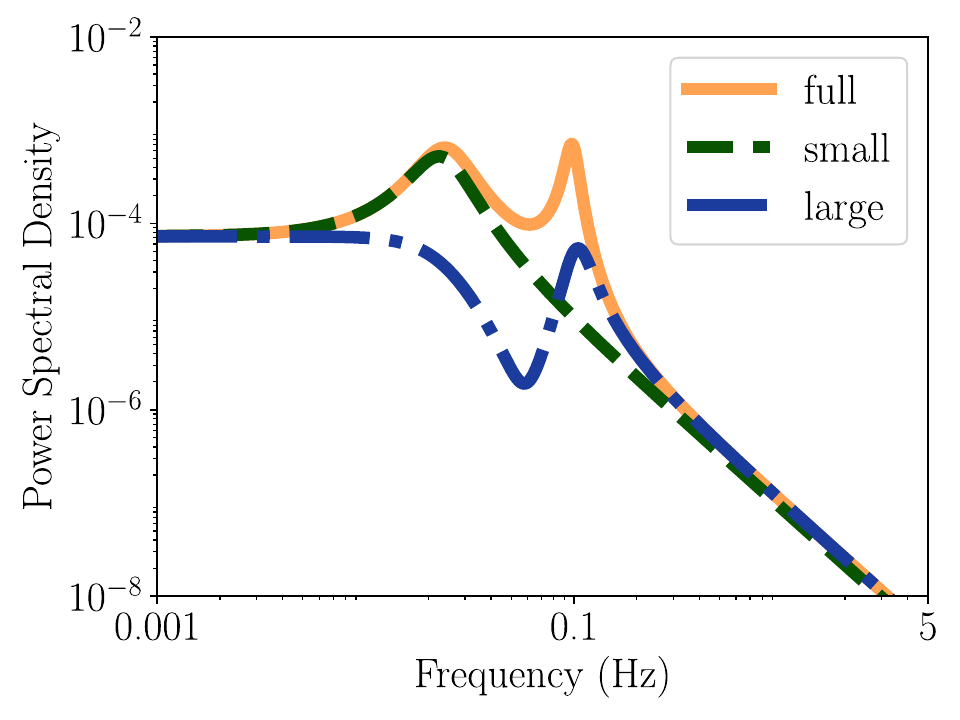}
    \caption{\textbf{Rational function solution for the spectral density in the Rock-Paper-Scissors-Lizard-Spock ($\bf n=5$) model.} We plot the full rational function for the power spectral density of the ``paper'' population, and its analytical approximations isolating the low-frequency peak, $\mathcal{S}^{22}_{\mathrm{small}}(\omega) = (p_0 + p_1 \omega^2)/(q_0 + q_1 \omega^2 + q_2 \omega^4)$ and the high-frequency peak, $\mathcal{S}^{22}_{\mathrm{large}}(\omega) = (p_0 + p_2 \omega^4 + p_3 \omega^6)/(q_0 + q_2 \omega^4 + q_3 \omega^6 + q_4 \omega^8)$.} 
    \label{fig:rps_decomposition}
\end{figure}

\begin{figure*}[th!]
    \centering
    \includegraphics[width=\linewidth]{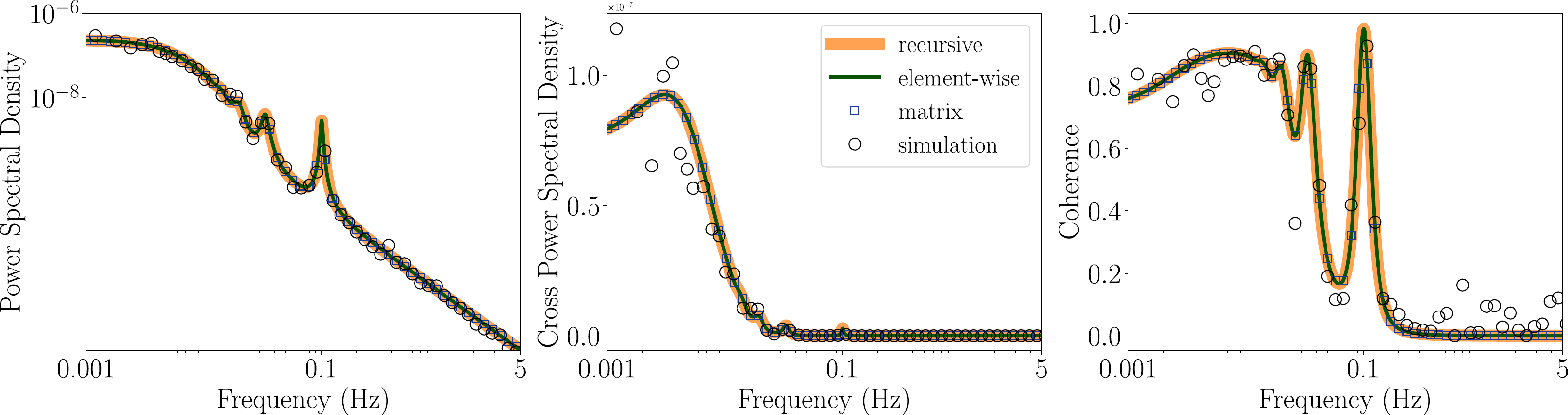}
    \caption{\textbf{Spectral density (auto and cross) and coherence using different methods for a 31-dimensional Rock-Paper-Scissors type model.} We plot the power spectral density (left), the absolute cross-power spectral density (middle), and the coherence (right) for a fixed mutation parameter ($\mu=0.0005$) each calculated via simulation, the rational function solutions, and the matrix solution. The PSD is calculated for the Rock population. The cross-PSD and coherence are calculated between the Rock and Paper populations. We consider the case of multiplicative noise, $\sigma x_j$, added to each population, with standard deviation $\sigma=10^{-4}$.}
    \label{fig:psd_rps}
\end{figure*}
In addition to the populations evolving according to the sociological strategy, we also consider the effects of global mutations and fluctuations in population. For the global mutations, we introduce a parameter $\mu$ such that each population mutates into every other population with constant rate $\mu$, as illustrated in Fig.~\ref{fig:rps_model}. This results in a stochastic dynamical system where the deterministic part evolves according to replicator-mutator dynamics \cite{ mobilia2010oscillatory, toupo2015nonlinear}. Furthermore, we model the noise as multiplicative Gaussian white noise, $\sigma x_i \eta_i$, with standard deviation $\sigma$ small enough that the population stays close to the fixed point. Thus, for the $n$-dimensional system, we can represent the set of differential equations for all $i\leq n$ as follows:
\begin{equation}
    \frac{\mathrm{d}x_i}{\mathrm{d}t} = x_i (f_{x_i} - \phi) + \mu\biggl(-(n-1)x_i + \sum_{\substack{j=1 \\ j\neq i}}^n x_j\biggr) +\sigma x_i \eta_i 
\label{eq:replicator}
\end{equation}

To ensure that mass conservation is followed, we define the system according to the replicator-mutator differential equation (Eq.~\ref{eq:replicator}) for $i\leq n-1$, and compute the $n$-th population from the constraint $\sum_{i=1}^{n} x_i = 1$. Recall that the individual payoff is $f_{x_i}=\sum_{i=1}^{n} P_{ij}x_j$, and the average payoff of the population is $\phi = \sum_{i=1}^{n} x_i f_{x_i} = 0$. The dynamical system has a fixed point at $x_i^{ss} = 1/n$, for all $i\leq n$. Additionally, the elements of the Jacobian, $J_{ij}$, for the dynamical system evaluated at the fixed point can be defined as a function of the mutation rate parameter $\mu$, 
\begin{equation}
    J_{ij} = 
    \begin{cases}
        \begin{cases}
            0 & j \text{ is even}\\
            (-1)^{i}\times 2/n & j \text{ is odd}\\
        \end{cases}  & i > j\\ 
        (-1)^{i}/n - \mu n & i = j\\
        \begin{cases}
            (-1)^{i}\times 2/n & j \text{ is even}\\
            0 & j \text{ is odd}\\
        \end{cases}  & i < j.
    \end{cases}
\end{equation}

We simulate a 31D system with a fixed mutation parameter, $\mu=0.0005$, and the scaling factor of the noise, $\sigma=10^{-4}$, that is the same for all the populations. In Fig.~\ref{fig:psd_rps} (\textit{Fig.~S4} for the 5-D version), we show the power spectral density for the Rock population (viz., the first strategy), the absolute cross-power spectral density and the coherence between the Rock and Paper populations (viz., the first and second strategy), each using Welch\textsc{\char13}s method \cite{welch1967use} (simulation), our recursive and element-wise rational function solutions, and the matrix solution (Eq.~\ref{eq:main_mat_sol}). Our solutions match across the entire range of frequencies, verifying their correctness. See the \textit{Supplemental Material} for the analysis of a 5-dimensional Rock-Paper-Scissors-Lizard-Spock version of the game, alongside the analytical solution for its auto-spectrum. For this system, it is possible to leverage the analytical coefficients obtained by our rational function solution to derive analytical approximations that isolate the different features of the model (Fig.~\ref{fig:rps_decomposition} and \textit{Supplemental Material}).

\subsection{Hawkes model}
\label{sec:hawkes_model}
\begin{figure}[htb]
    \centering
    \includegraphics[width=\linewidth]{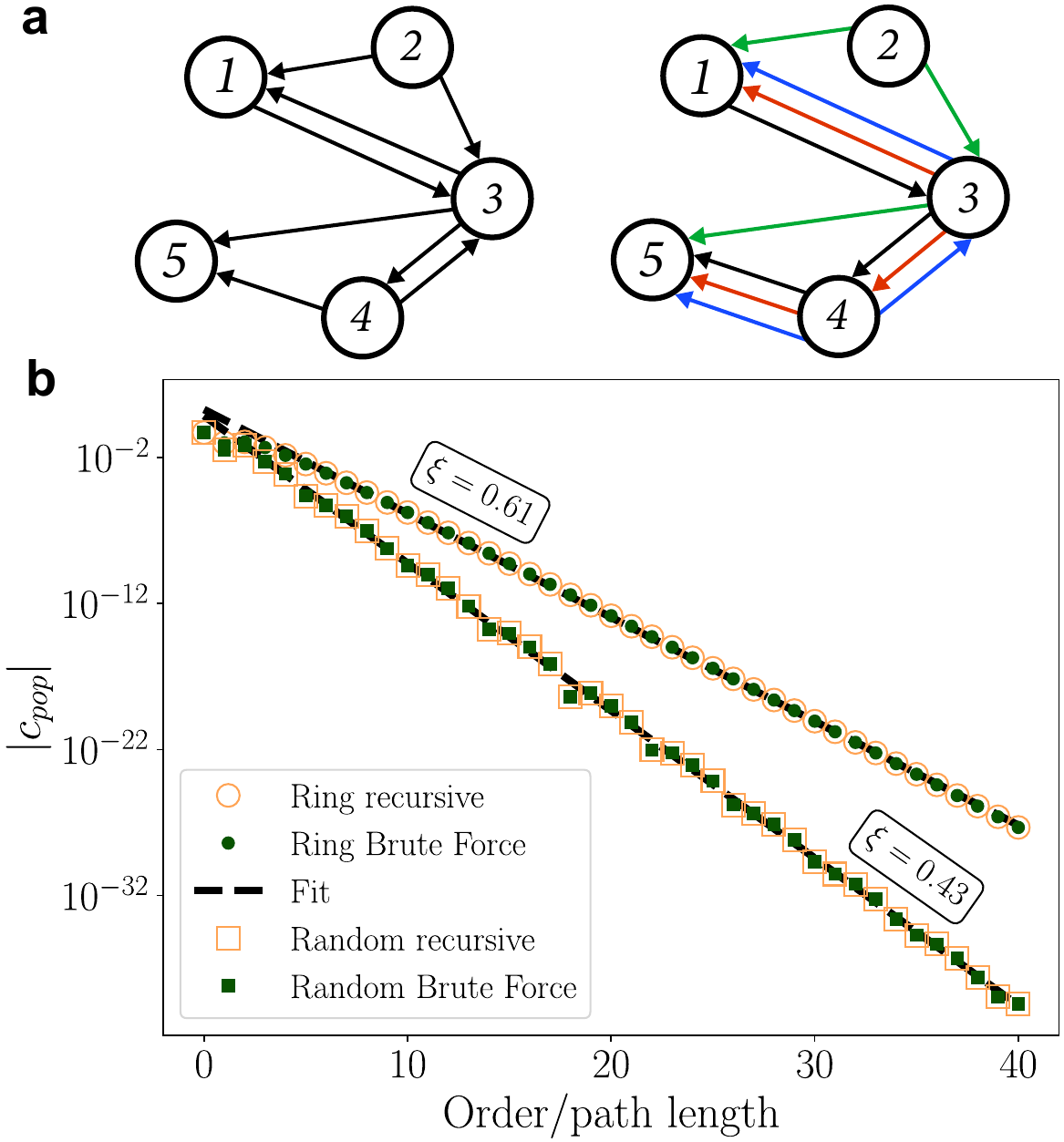}
    \caption{\textbf{Population covariance of spiking neural networks.} \textbf{a}: (Left) Directed graph representing the connectivity structure of a spiking network model. (Right) Third-order paths contributing to the $(a=1,b=5)$ entry of $\mathbf{S}_3$ in the integrated covariance, Eq.~\ref{eq:hawkes_covariance}. $c=2,3,4$ are the intermediary nodes in the graph. Colored arrows represent the path between nodes $a=1$ and $b=5$ via (green) $c=2$ and (red) $c=3$ due to the term $\mathbf{G} \, \mathbf{Y} \left(\mathbf{G}^{\top }\right)^2$; (blue) via $c=4$ due to the term $\mathbf{G}^2 \, \mathbf{Y}  \left(\mathbf{G}^{\top }\right)$; (black) direct path between nodes $a=1$ and $b=5$ due to the term $\mathbf{Y} \left(\mathbf{G}^{\top }\right)^3$.
    \textbf{b}: Absolute average pairwise population covariance, $|c_{pop}|$, as a function of path length $m$, found by calculating the mean of the matrix elements $\Ss_m$ in Eq.~\ref{eq:hawkes_covariance}, for a directed cycle graph (circles), and a directed random graph (squares). The contributions $\Ss_m$ are computed by the recursive algorithm in Theorem~\ref{theorem:without_omega}, and by a brute force approach. The Hawkes models are defined by the following parameters: $N_E = 16$ (\#E neurons), $N_I = 4$ (\#I neurons), $g_E = 0.015$, $g_I = -0.075$. For the model on a random graph, the probability of drawing a directed edge between any two nodes $(i, j)$ is $0.45$.}
    \label{fig:recursive_g}
\end{figure}

Our final example shows how we can use Theorem~\ref{theorem:without_omega} to find the connectivity-dependent correlation structure of spiking neurons connected in an arbitrary network. We consider the Hawkes model \cite{hawkes1971point} describing recurrently connected linearly interacting spiking neurons, as in Ref.~\cite{pernice2011structure}. The network consists of linearly interacting point processes with recurrent connectivity given by the kernel matrix, $\G(t)$. The firing rates vector, $\mathbf{y}(t)$, evolves according to the Hawkes process, given by the following equation,
\begin{equation}
    \mathbf{y}(t) = \mathbf{y}_0 + \int_{-\infty}^{t} \G(t-u) d \mathbf{N}(u)
\end{equation}
where $\mathbf{N}(t)$ is the vector of point processes and $\mathbf{y}_0$ is a biasing vector encoding a constant external drive. Assuming a stationary state, the covariance matrix is defined by the following self-consistent equation,
\begin{equation}
    \mathbf{R}(\tau) = \G(\tau) \Y + (\G * \mathbf{R})(\tau)
\end{equation}
where $\Y$ is the diagonal matrix containing the entries of $\mathbf{y}_0$ on its diagonal and $*$ denotes convolution. The spectral density matrix of the Hawkes process is then given by \cite{hawkes1971point},
\begin{equation}
    \mathbfcal{S}(\omega) = \left(\I - \hat{\G}(\omega)\right)^{-1} \, \Y \, \left(\I - \hat{\G}^\top(\omega)\right)^{-1}
\end{equation}
where $\hat{\G}(\omega)$ is the Fourier transform of $\G(\tau)$. Then, the integrated covariance matrix $\mathbf{R} \coloneqq \int_{-\infty}^
\infty\mathbf{R}(\tau)\mathrm{d}\tau = \mathbfcal{S}(0)$ is given by, 
\begin{equation}
    \mathbf{R} = \left(\I - \mathbf{G} \right)^{-1} \mathbf{Y} \left(\I - \mathbf{G}^\top\right)^{-1}
\end{equation}
where $\G \coloneqq \int_{-\infty}^\infty\G(\tau)\mathrm{d}\tau = \hat{\G}(0)$.

We consider the problem of how the connectivity pattern influences the correlation structure in networks of spiking neurons, as studied in \cite{ostojic2009connectivity, pernice2011structure, trousdale2012impact}. In particular, when the spectral radius of $\G$ is less than $1$, we can expand the covariance matrix as the following infinite sum,
\begin{equation}
\begin{split}
    \mathbf{R} &= \left(\sum_{j=0}^{\infty} \mathbf{G}^j\right) \mathbf{Y} \left(\sum_{k=0}^{\infty} \left(\mathbf{G}^{\top }\right)^{k}\right) \\
    &=  \sum_{j, k = 0}^{\infty} \mathbf{G}^j \, \mathbf{Y}  \left(\mathbf{G}^{\top }\right)^{k} = \sum_{m = 0}^\infty \Ss_m
\end{split}
\label{eq:hawkes_covariance}
\end{equation}
where $\Ss_m \coloneqq \sum_{\substack{j+k = m \\ j\geq0, \, k\geq 0}}\G^j \, \Y \, \left(\G^\top\right)^{k}$ accounts for the contributions of \textit{paths} of order (length) $m$ to the integrated covariance matrix. Note that for the $(a, b)^{th}$ entry of $\mathbf{G}^j \, \mathbf{Y}  \left(\mathbf{G}^{\top }\right)^{k}$, $\G^j$ represents the sum over all paths towards $a$ from driving nodes $c$ that are distance $j$ from $a$; while $\left(\mathbf{G}^{\top }\right)^{k}$ encodes the sum over all paths towards $b$ from nodes $c$ that are distance $k$ from $b$. Then the total length of the paths between $a$ and $b$ through driving nodes $c$ is $m = j + k$ (see Fig.~\ref{fig:recursive_g}a for an example). Therefore, this quantity provides a way of quantifying how graph connectivity influences correlations in a network of spiking neurons modeled as a Hawkes process. 

We demonstrate how the recursive algorithm in Theorem~\ref{theorem:without_omega} allows us to compute $\Ss_m$ in $\mathcal{O}(n^4)$ compared to $\mathcal{O}(n^5)$ by a brute force solution. We consider two different types of networks, both described in detail in \cite{pernice2011structure}. The first network has neurons (both E \& I) in a ring structure with only connections between nearest neighbors allowed (i.e., in a bidirectional directed cycle graph). The second is a randomly connected network of E and I neurons (i.e., a random directed graph). The total number of neurons is 20, with 16 E and 8 I neurons, each receiving a constant input of $10$Hz, i.e., $y_0 = 10$Hz. The kernels in matrix $\G(\tau)$ are exponentially decaying with a time constant of $10$ms and an amplitude of $1.5$Hz/-$7.5$Hz for E/I spikes. This corresponds to the entries of the matrix $\G$ being $0.015$ for E presynaptic neurons and $-0.075$ for I presynaptic neurons. For the randomly connected network, the oriented edge between neurons $i$ and $j$ is added with a probability $0.45$, chosen so that the total number of connections for the ring model and the random model are equal.

Fig.~\ref{fig:recursive_g} shows the absolute average pairwise population covariance, $|c_{pop}|$, as a function of path length $m$, found by calculating the mean of the matrix elements of $\Ss_m$, and confirms that the recursive solution (Theorem~\ref{theorem:without_omega}) matches the brute force solution. Moreover, since $|c_{pop}|$ decays exponentially as a function of the order $m$, we estimate the typical order/path (correlation) length ($\xi$) by assuming the relation, $|c_{pop}| \propto e^{-m/\xi}$. Clearly, the network topology influences the characteristic path length, so that $\xi_{ring} > \xi_{random}$, despite both networks having the same number of connections.

% Note that the sum is calculated over all combinations of $(j,k)$ in the range $0 \leq j<\infty$ and $0 \leq k<\infty$. 
% $\mathbf{G}^{(j,k)}$ can be interpreted as how the full correlation between neurons $a$ and $b$ results from the contributions of all neurons, weighted by their rate, via all possible paths of length $j$ to node $a$ and length $k$ to node $b$, for all $j$ and $k$. Further, the collective influence on the correlations of paths of a given length can be found by adding all the terms of a given length/order $(j+k)$. 

\section{Discussion}
\label{sec:discussion}
In this work, we presented two approaches for computing explicitly the rational function solution of the power spectral density matrix for stochastic linear time-invariant systems driven by Gaussian white noise. In one approach, we derived element-wise solutions (valid for arbitrary $n$-dimensional LTI systems) for the rational function coefficients of any given entry of the PSD matrix, computable in roughly $\mathcal{O}(n^4\log n)$ operations. Each coefficient is expressed in terms of Bell polynomials \cite{bell1934exponential} evaluated on submatrices of the Jacobian, defined according to the algorithm summarized in Fig.~\ref{fig:O_algo}. Our element-wise solution is particularly useful to gain insights into the dependence of the PSD on the model's functional form and its parameters, especially when the model is fairly low dimensional. Therefore, we provide explicit solutions for the auto-spectra of a general $n=2,3$ and $4$ dimensional system driven by uncorrelated Gaussian noise increments. A similar, albeit longer, symbolic expression can be obtained for any dimensional system and noise correlation structure via our software package (or working through the general solution in \textit{Supplemental Material}), see section \ref{sec:code}. We also showed that the polynomial matrix coefficients in the numerator and the polynomial coefficients in the denominator can be obtained recursively by means of a Faddeev-Leverrier-type algorithm in $\mathcal{O}(n^4)$ operations, see Theorem~\ref{theorem:0}.\footnote{Note that since we perform our numerical calculations with arbitrary precision our implementation scales $\mathcal{O}(n^5)$, see \textit{Supplemental Material}.} Since the algorithm involves simple matrix multiplication and trace operations, it should be possible to further optimize it to achieve better time complexity. Crucially, both our solutions are free of matrix inverses, in addition to being more revealing than the simpler matrix solution given by Eq.~\ref{eq:main_mat_sol}.

Finally, we applied our approaches for the derivation of analytical (i.e., closed form) PSD of stochastic nonlinear dynamical models of neural systems, namely Fitzhugh-Nagumo ($n=2$), Hindmarsh-Rose ($n=3$), Wilson-Cowan ($n=4$), and the Stabilized Supralinear Network ($n=22$), as well as of an evolutionary game-theoretic model with mutations ($n=5,31$). 

To showcase the power of the recursive approach and its broad applicability, we used it to find a recursive solution for an auxiliary spectral matrix of similar form, Appendix~\ref{sec:appendix_recursive}. This enabled us to derive a recursive algorithm (Theorem~\ref{theorem:without_omega}) for the terms of varying order in an expansion of the integrated covariance matrix for a network of linearly interacting spiking neurons modeled by Hawkes processes. The recursive nature of the solution not only leads to a faster algorithm ($\mathcal{O}(n^4)$ vs the $\mathcal{O}(n^5)$ of the brute force approach) but also provides deeper insights into the structure of the covariance matrix.

Since $\omega$ appears in different powers in both the numerator and denominator of the rational function of the power spectral density, we can neglect certain terms when they are small within the relevant range of $\omega$. For example, in Fig.~\ref{fig:rps_decomposition}, we demonstrated how setting these small terms to zero in specific frequency regimes, both small and large, enables us to isolate the two peaks observed in the spectra of the paper population within the 5D Rock-Paper-Scissor-Lizard-Spock model. The analytical tractability of these solutions may enable us to rule out certain models, identify simpler ones (e.g., by providing insight into how model parameters contribute to spectral features) or 
%identify specific parameter regimes that best fit the experimental data. 
% For instance, in an experimental system exhibiting both fast and slow dynamics, we could exclude certain dynamical systems by demonstrating that they cannot oscillate at two distinct frequencies. 
%For instance, this approach may allow us 
to establish constraints on model parameters that ensure the presence of a peak in the spectra within a specified frequency range. This direction will be explored in future work. 

Overall, our solutions should find broad applicability in the analysis of the dynamics of stochastic biological systems exhibiting fixed point solutions, as well as in the broader context of dynamical systems and control theory.

\section{Code availability}
\label{sec:code}
The code for the general solutions of the auto- and cross-spectrum along with descriptions of the various dynamical systems (including the Hawkes model) considered in the manuscript can be found at \href{https://github.com/martiniani-lab/spectre}{\url{https://github.com/martiniani-lab/spectre}}.

\section*{Acknowledgements}
The authors are deeply grateful to David Heeger for many insightful discussions that motivated this work. The authors also acknowledge valuable discussions with Mathias Casiulis, Ruben Coen-Cagli, Lyndon Duong, and Flaviano Morone. This work was supported by the National Institute of Health under award number R01EY035242. S.M. acknowledges the Simons Center for Computational Physical Chemistry for financial support. This work was supported in part through the NYU IT High Performance Computing resources, services, and staff expertise.

\bibliographystyle{apsrev4-1}
\bibliography{bibliography}

\newpage

\appendix

\section{Rational function decomposition of the PSD} \label{sec:appendix_rational}
The power spectral density of the LTI SDE described in Eq.~\ref{eq:main2} can alternatively be written as the inverse of the Hermitian positive definite matrix (see \textit{Supplemental Material}), assuming that $\mathbf{C} \equiv \mathbf{L} \, \mathbf{D} \, \mathbf{L}^\top$ is positive definite,
\begin{equation}
\begin{split}
    \mathbfcal{S}(\omega) =& \bigl[\mathbf{J}^{\top} {\mathbf{C}}^{-1} \mathbf{J} + \omega^2 {\mathbf{C}}^{-1}
    \\& + \dot{\iota}\omega (\mathbf{J}^{\top}{\mathbf{C}}^{-1} - {\mathbf{C}}^{-1}\mathbf{J}) \bigr]^{-1}
\end{split}
\label{eq:matrix_sol_cholesky}
\end{equation}
This enables us to efficiently compute the matrix inverse by Cholesky decomposition. We can recast $\mathbfcal{S}(\omega)$ as a rational function and upon expanding Eq.~\ref{eq:main_mat_sol} in terms of the adjugate and determinant matrices, we find,
\begin{equation}
\begin{split}
    \mathbfcal{S}(\omega) &= (\dot{\iota}\omega \mathbf{I} + \mathbf{J})^{-1} \, \mathbf{C} \, (-\dot{\iota}\omega \mathbf{I} + \mathbf{J})^{-\top} \\ &= \frac{\mathrm{adj}(\J+\dot{\iota} \omega\I) \, \mathbf{C} \, \mathrm{adj}^\top(\mathbf{J-\dot{\iota} \omega\I)}}{\mathrm{det}(\J+\dot{\iota} \omega\I)\, \mathrm{det}(\mathbf{J-\dot{\iota} \omega\I})} \\& = \frac{\mathbf{Z}(\omega)}{Q(\omega)}  
\end{split}
\label{eq:gen1} 
\end{equation}
Here, $\mathbf{Z}(\omega)$ is a complex polynomial matrix representing the numerator of the solution, and $Q(\omega)$ is an even-powered polynomial of degree $2n$ representing the denominator of the solution. When we express the adjugate of the matrices in terms of the corresponding cofactor matrices, it is apparent that $\mathbf{Z}(\omega)$ contains complex polynomials of degree $2n-2$ with even powers of $\omega$ being real and odd powers of $\omega$ being imaginary.

\section{Recursive algorithm for auxiliary spectral matrix} \label{sec:appendix_recursive}
\begin{theorem} 
\label{theorem:0_g}
Let $\G$ be a real square matrix, $\Y$ be a positive definite matrix, and $\mathbfcal{T}(\omega)$ be defined as follows,
\begin{equation}
\begin{split}
    \mathbfcal{T}(\omega) &= (\omega\I-\G)^{-1} \, \Y \, (\omega\I-\G^\top)^{-1} \\
    &= \frac{\sum\limits_{\alpha=0}^{2n-2}\Pb_\alpha \omega^{\alpha} }{\sum\limits_{\alpha=0}^{2n}q_\alpha \omega^{\alpha}} \label{eq:rational_function_g}
\end{split}
\end{equation}
The numerator's matrix coefficients are given by the recursive equations,
\begin{equation}
    \Pb_{\alpha-1} = q_{\alpha+1} \Y + \G \Pb_{\alpha} + \Pb_{\alpha} \G^\top -\G \Pb_{\alpha+1} \G^\top 
\end{equation}
for $\alpha \in \{1,\dots, 2n-1\}$, starting from $\alpha = 2n-1$ with $\Pb_{2n}=\Pb_{2n-1}=\mathbf{0}$. The scalar coefficients of the denominator can be recursively calculated using,
\begin{equation}
\begin{split}
    q_\alpha = \frac{2}{2n - \alpha}\left[\Tr\left(\Y^{-1} \G \Pb_\alpha \G^\top \right) - \Tr\left(\Y^{-1}\G \Pb_{\alpha-1} \right)\right]
\end{split}
\end{equation}
for $\alpha \in \{1,\dots, 2n-1\}$, and $q_{2n} = 1$.
\end{theorem}

\end{document}

% --------------------------------------------------------------
%                         Start here
% --------------------------------------------------------------
 
\title{Supplementary Material for Manuscript:\\ ``Element-wise and Recursive Solutions for the Power Spectral Density of Biological Stochastic Dynamical Systems at Fixed Points"}

\author[1,2]{Shivang Rawat}
\author[1,2,3]{Stefano Martiniani}
\affil[1]{Courant Institute of Mathematical Sciences, New York University, New York 10003, USA}
\affil[2]{Center for Soft Matter Research, Department of Physics, New York University, New York 10003, USA}
\affil[3]{Simons Center for Computational Physical Chemistry, Department of Chemistry, New York University, New York 10003, USA}

\date{}

\date{\vspace{-5ex}}

\small\maketitle
\pagenumbering{gobble}
\newpage
\pagenumbering{gobble}
\tableofcontents
\addtocontents{toc}{\protect\setcounter{tocdepth}{3}}
\newpage
\pagenumbering{arabic}

\section{Spectral density matrix for a linear SDE / non-linear SDE near stable fixed points}

We begin with describing the system. A general high-dimensional stochastic differential equation (SDE) describing the time evolution of quantities associated with the system can be written as:
\begin{equation}
    \mathrm{d}\mathbf{x} = \mathbf{f}(\mathbf{x})\mathrm{d}t + \mathbf{L} \mathrm{d}\mathbf{W}    \label{eq:main1}
\end{equation}
Where $\mathbf{x} \in \R^{n}$ is a vector of variables that evolve with time according to the nonlinear function $ \mathbf{f}(\mathbf{x}) \in \R^{n} $ and an additive noisy drive given by $\mathbf{L} \mathrm{d}\mathbf{W} \in \R^{n}$. The matrix  $\mathbf{L} \in \R^{n\times m}$ is composed of the weights of the different noise sources appearing in the SDE,
\begin{equation}
    \mathbf{L} = \begin{bmatrix}
        l_{11} & l_{12} & \dots & l_{1m}\\
        l_{21} & l_{22} & \dots & l_{2m}\\
        \vdots &        & \ddots & \vdots \\
        l_{n1} & l_{n2} & \dots & l_{nm}
    \end{bmatrix}
\end{equation}
The differential noise vector (Wiener increments) ($\mathrm{d}\mathbf{W} = [\begin{matrix} \mathrm{d}\mathrm{W}_1 & \mathrm{d}\mathrm{W}_2 & \dots & \mathrm{d}\mathrm{W}_m \end{matrix}]^\top$) contains infinitesimally small independent Wiener increments with the following properties:
\begin{equation}
    \mathbb{E}[\mathrm{d}\mathrm{W}_i \mathrm{d}\mathrm{W}_j] = 
    \begin{cases}
        0 & i \neq j\\
        \sigma_i^2 \,dt & i=j
    \end{cases} \label{eq:wiener}
\end{equation}
with $\sigma^i \in \R^+$. Therefore, there is a diagonal diffusion matrix, $\mathbf{D}$, associated with the noise vector $\mathrm{d}\mathbf{W}$, such that its diagonal entries contain the variance of the corresponding noise terms. We can express this matrix as the product of two matrices that contain the corresponding standard deviations, $\sigma_i$. We call this matrix $\mathbf{S} \in \R^{m\times m}$ defined as follows,

\begin{equation}
    \mathbf{S} = \begin{bmatrix}
    \sigma_{1} & & & &\\
     & \sigma_2 & & & \\
    & & \ddots & & \\
    & &  & \sigma_{m-1} & \\
    & &  &  & \sigma_{m}\\
  \end{bmatrix}
\end{equation}
Such that the diffusion matrix $\mathbf{D}$ can be written as,
\begin{equation}
    \mathbf{D} = \mathbf{S} \, \mathbf{S}^\top \label{eq:diff_mat}
\end{equation}
We want to calculate the Power Spectral Density (PSD) matrix using the responses $\mathbf{x}(t)$, near the stable fixed points (stable node/spiral), when simulated using the dynamical system, Eq.~\ref{eq:main1}. To accomplish this, we consider the deterministic part of the SDE and linearize $\mathbf{f}(\mathbf{x})$ around its stable fixed point $\mathbf{x}_{ss}$, such that $\mathbf{f}(\mathbf{x_{ss}})=\mathbf{0}$.
Now, let us consider the linearized system about the stable fixed point $\mathbf{x}_{ss}$,
\begin{equation}
    \mathrm{d}\mathbf{x} = \left. \mathbf{J} \right|_{\mathbf{x}_{ss}} (\mathbf{x}-\mathbf{x}_{ss})\mathrm{d}t + \mathbf{L} \mathrm{d}\mathbf{W}     \label{eq:2}
\end{equation}
We make the variable transformation $\mathbf{\Delta}=\mathbf{x}-\mathbf{x}_{ss}$ and use a simpler notation for the Jacobian $\left. \mathbf{J} \right|_{\mathbf{x}_{ss}} \rightarrow \mathbf{J} \in \R^{n\times n}$. The SDE then becomes, 
\begin{equation}
    \mathrm{d}\mathbf{\Delta} = \mathbf{J}\mathbf{\Delta}\mathrm{d}t + \mathbf{L} \mathrm{d}\mathbf{W}      \label{eq:3}
\end{equation}
We can solve this LTI SDE analytically by Ito's solution, given by the following sum of a function of the initial condition (we assume that we start at the steady-state such that $\boldsymbol\Delta(0)=\mathbf{0}$) and a stochastic integral,
\begin{equation}
\begin{split}
    \boldsymbol\Delta(t)&=e^{\mathbf{J}t}\boldsymbol\Delta(0)+\int_{0}^{t} e^{\mathbf{J}(t-s)} \,\mathbf{L}\mathrm{d}\mathbf{W}(s) = \int_{0}^{t} e^{\mathbf{J}(t-s)} \,\mathbf{L}\mathrm{d}\mathbf{W}(s) 
\end{split}
\end{equation}
Similarly, we have,
\begin{equation}
    \boldsymbol\Delta(t+\tau)= \int_{0}^{t+\tau} e^{\mathbf{J}(t+\tau-s)} \,\mathbf{L}\mathrm{d}\mathbf{W}(s) 
\end{equation}
The cross-correlation matrix $\mathbfcal{R}(t,\tau)$ is defined as follows,
\begin{equation}
    \mathbfcal{R}(t,\tau) = \mathop{\mathbb{E}}\left[\mathbf{\Delta}(t)\mathbf{\Delta}^\top(t+\tau)\right]
\end{equation}
We assume stationary conditions which can be implemented by taking the limit $t \rightarrow \infty$ to make the statistics of the random variable $\mathbf{\Delta}(t \to \infty)$ stationary, such that the correlation matrix is only a function of $\tau$,  $\mathbfcal{R}(\tau)$. Substituting the solution for $\boldsymbol\Delta(t)$ and applying the limit, we get,
\begin{equation}
    \mathbfcal{R}(\tau)= \lim_{t\to\infty} \mathop{\mathbb{E}}\left[\mathbf{\Delta}(t)\mathbf{\Delta}^\top(t+\tau)\right]  = \lim_{t\to\infty} \mathop{\mathbb{E}}\left[\int_{0}^{t} e^{\mathbf{J}(t-s)} \,\mathbf{L}\mathrm{d}\mathbf{W}(s) \left(\int_{0}^{t+\tau} e^{\mathbf{J}(t+\tau-s)} \,\mathbf{L}\mathrm{d}\mathbf{W}(s) \right)^\top \right]   \label{eq:cor}
\end{equation}
Now, we calculate the correlation matrix for positive and negative values of $\tau$ separately. First, we consider the case when $\tau \geq 0$. Using the Ito isometry, the properties of matrix exponentials, and, the independence of Wiener increments through time, we can write Eq.~\ref{eq:cor} as,
\begin{equation}
\begin{split}
    \mathbfcal{R}(\tau) & = \lim_{t\to\infty}\int_{0}^{t} e^{\mathbf{J}(t-s)} \,\mathbf{L} \, \mathbb{E} \bigl[\mathrm{d}\mathbf{W}(s) \mathrm{d}\mathbf{W}(s)^\top \bigr]   \,\mathbf{L}^\top  \,e^{\mathbf{J}^\top(t+\tau-s)} \\
    & = \lim_{t\to\infty}\int_{0}^{t} e^{\mathbf{J}(t-s)} \,\mathbf{L} \, \mathbf{D}   \,\mathbf{L}^\top  \,e^{\mathbf{J}^\top(t+\tau-s)} \, \mathrm{d}s
    \label{eq:cor1}
\end{split}
\end{equation}
We use the property described in Eq.~\ref{eq:wiener} to make the substitution $\mathbb{E} \bigl[\mathrm{d}\mathbf{W}(s) \mathrm{d}\mathbf{W}(s)^\top \bigr] = \mathbf{D} \mathrm{d} s$. We simplify Eq.~\ref{eq:cor1} further by changing the variable of integration as $k=t-s$. 
\begin{equation}
\begin{split}
    \mathbfcal{R}(\tau) &= \lim_{t\to\infty}\int_{0}^{t} e^{\mathbf{J}k} \,\mathbf{L}  \, \mathbf{D}  \,\mathbf{L}^\top  \,e^{\mathbf{J}^\top(\tau+k)} \,\mathrm{d}k \\
    &= \int_{0}^{\infty} e^{\mathbf{J}k} \,\mathbf{L} \, \mathbf{D}   \,\mathbf{L}^\top  \,e^{\mathbf{J}^\top(\tau+k)} \,\mathrm{d}k
\end{split}
\end{equation}
Next, we consider the case when $\tau<0$. Again, using the Ito isometry, the properties of matrix exponentials, and, the independence of Wiener increments through time, we can write Eq.~\ref{eq:cor} as,
\begin{equation}
    \mathbfcal{R}(\tau) = \lim_{t\to\infty}\int_{0}^{t+\tau} e^{\mathbf{J}(t-s)} \,\mathbf{L}  \, \mathbf{D}  \,\mathbf{L}^\top  \,e^{\mathbf{J}^\top(t+\tau-s)} \,\mathrm{d}s   
\end{equation}
We change the variable of integration to $k=t-s$. The integral above becomes,
\begin{equation}
\begin{split}
    \mathbfcal{R}(\tau) &= \lim_{t\to\infty}\int_{-\tau}^{t} e^{\mathbf{J}k} \,\mathbf{L}  \, \mathbf{D}  \,\mathbf{L}^\top  \,e^{\mathbf{J}^\top(\tau+k)} \,\mathrm{d}k \\
    &= \int_{-\tau}^{\infty} e^{\mathbf{J}k} \,\mathbf{L}  \, \mathbf{D}  \,\mathbf{L}^\top  \,e^{\mathbf{J}^\top(\tau+k)} \,\mathrm{d}k
\end{split}
\end{equation}
Therefore the correlation matrix can be written as,
\begin{equation}
    \mathbfcal{R}(\tau) = 
    \begin{cases}
    \int_{0}^{\infty} e^{\mathbf{J}k} \,\mathbf{L}  \, \mathbf{D}  \,\mathbf{L}^\top  \,e^{\mathbf{J}^\top(\tau+k)} \,\mathrm{d}k, & \text{for } \tau \geq 0\\
    \int_{-\tau}^{\infty} e^{\mathbf{J}k} \,\mathbf{L}  \, \mathbf{D}  \,\mathbf{L}^\top  \,e^{\mathbf{J}^\top(\tau+k)} \,\mathrm{d}k, & \text{for } \tau < 0
    \end{cases}
\end{equation}
Since the correlation matrix is only a function of $\tau$ ($\Delta(t)$ is a wide-sense stationary process in the limit $t \rightarrow \infty$), we can apply the cross-correlation theorem (a general form of Wiener-Khinchin theorem) to calculate the analytical PSD matrix. The PSD matrix is defined as the Fourier transform of the correlation matrix,
\begin{equation}
\begin{split}
     \mathbfcal{S}(\omega) &= \mathcal{F}(\mathbfcal{R}(\tau)) = \int_{-\infty}^{\infty} \mathbfcal{R}(\tau) e^{-\dot{\iota} \omega \tau} \, \mathrm{d}\tau 
\end{split}
\end{equation}
Substituting the analytical integrals for the correlation matrix and simplifying we get,
\begin{equation}
\begin{split}
    \mathbfcal{S}(\omega) &= \int_{-\infty}^{0} \mathbfcal{R}(\tau) e^{-\dot{\iota} \omega \tau} \, \mathrm{d}\tau + \int_{0}^{\infty} \mathbfcal{R}(\tau) e^{-\dot{\iota} \omega \tau} \, \mathrm{d}\tau \\ 
    &= \underset{\Romannum{1}}{\int_{-\infty}^{0} \int_{-\tau}^{\infty} e^{\mathbf{J}k} \,\mathbf{L}  \, \mathbf{D}  \,\mathbf{L}^\top  \,e^{\mathbf{J}^\top(\tau+k)} \,\mathrm{d}k e^{-\dot{\iota} \omega \tau} \, \mathrm{d}\tau} \\
    & \, + \underset{\Romannum{2}}{\int_{0}^{\infty} \int_{0}^{\infty} e^{\mathbf{J}k} \,\mathbf{L}  \, \mathbf{D}  \,\mathbf{L}^\top  \,e^{\mathbf{J}^\top(\tau+k)} \,\mathrm{d}k e^{-\dot{\iota} \omega \tau} \, \mathrm{d}\tau}
\end{split}
\end{equation}
We first consider $\Romannum{1}$ and try to simplify the integral. On changing the order of integration, we get,
\begin{equation}
\begin{split}
    \Romannum{1} &= \int_{-\infty}^{0} \int_{-\tau}^{\infty} e^{\mathbf{J}k} \,\mathbf{L} \, \mathbf{D} \,\mathbf{L}^\top  \,e^{\mathbf{J}^\top(\tau+k)} \, e^{-\dot{\iota} \omega \tau} \,\mathrm{d}k\, \mathrm{d}\tau\\
    &= \int_{0}^{\infty} \int_{-k}^{0} e^{\mathbf{J}k} \,\mathbf{L} \, \mathbf{D} \,\mathbf{L}^\top  \,e^{\mathbf{J}^\top(\tau+k)} \, e^{-\dot{\iota} \omega \tau} \, \mathrm{d}\tau \, \mathrm{d}k
\end{split}
\end{equation}
% \begin{figure}[htpb]
%     \centering
%     \includegraphics[width=0.55\linewidth]{figure 1.pdf}
%     \caption{\small \textbf{Changing the order of integration.} The shaded region in blue is the area that we integrate over. This region is bounded by the limits of the integrals as shown in the red lines.}
%     \label{fig:orderint}
% \end{figure}
Next, we simplify $\Romannum{2}$ by changing the order of integration,
\begin{equation}
\begin{split}
    \Romannum{2} &= \int_{0}^{\infty} \int_{0}^{\infty} e^{\mathbf{J}k} \,\mathbf{L} \, \mathbf{D} \,\mathbf{L}^\top  \,e^{\mathbf{J}^\top(\tau+k)} \, e^{-\dot{\iota} \omega \tau} \,\mathrm{d}k\, \mathrm{d}\tau\\
    &= \int_{0}^{\infty} \int_{0}^{\infty} e^{\mathbf{J}k} \,\mathbf{L} \, \mathbf{D} \,\mathbf{L}^\top  \,e^{\mathbf{J}^\top(\tau+k)} \, e^{-\dot{\iota} \omega \tau} \, \mathrm{d}\tau \, \mathrm{d}k\\
    &= \int_{0}^{\infty} e^{\mathbf{J}k} \,\mathbf{L} \, \mathbf{D}  \,\mathbf{L}^\top \, \int_{0}^{\infty} e^{-\dot{\iota} \omega \tau} \, e^{\mathbf{J}^\top(\tau+k)} \,  \mathrm{d}\tau \, \mathrm{d}k
\end{split}
\end{equation}
Upon shifting the bounds of the inner integral by $k$, we get,
\begin{equation}
\begin{split}
    \Romannum{2} &= \int_{0}^{\infty} e^{\mathbf{J}k} \,\mathbf{L} \, \mathbf{D} \,\mathbf{L}^\top \, \int_{k}^{\infty} e^{-\dot{\iota} \omega (\tau-k)} \, e^{\mathbf{J}^\top\tau} \,  \mathrm{d}\tau \, \mathrm{d}k\\
    &= \int_{0}^{\infty} e^{\dot{\iota} \omega k} e^{\mathbf{J}k} \,\mathbf{L} \, \mathbf{D} \,\mathbf{L}^\top \, \int_{k}^{\infty} e^{-\dot{\iota} \omega \tau} \, e^{\mathbf{J}^\top\tau} \,  \mathrm{d}\tau \, \mathrm{d}k\\
    &= \int_{0}^{\infty} e^{\dot{\iota} \omega k} e^{\mathbf{J}k} \,\mathbf{L} \, \mathbf{D} \,\mathbf{L}^\top \,\left( \int_{0}^{\infty} e^{-\dot{\iota} \omega \tau} \, e^{\mathbf{J}^\top\tau} \,  \mathrm{d}\tau - \int_{0}^{k} e^{-\dot{\iota} \omega \tau} \, e^{\mathbf{J}^\top\tau} \,  \mathrm{d}\tau\right) \, \mathrm{d}k\\
    &= \underset{\Romannum{3}}{\int_{0}^{\infty} e^{\dot{\iota} \omega k} e^{\mathbf{J}k} \, \mathbf{L} \, \mathbf{D} \, \mathbf{L}^\top \int_{0}^{\infty} e^{-\dot{\iota} \omega \tau}  e^{\mathbf{J}^\top\tau}   \mathrm{d}\tau \, \mathrm{d}k} - \underset{\Romannum{4}}{\int_{0}^{\infty} e^{\dot{\iota} \omega k} e^{\mathbf{J}k} \, \mathbf{L}\,  \mathbf{D} \, \mathbf{L}^\top  \int_{0}^{k} e^{-\dot{\iota} \omega \tau}  e^{\mathbf{J}^\top\tau}   \mathrm{d}\tau \, \mathrm{d}k}
\end{split}
\end{equation}

We now introduce a general property of the Laplace transform of matrix exponentials (also known as the resolvent of a matrix) which we will use in the derivation.
\begin{equation}
    \int_{0}^{\infty} e^{-\tau s} \, e^{\tau \mathbf{J}} \mathrm{d}\tau = (s\mathbf{I}-\mathbf{J})^{-1}   
\end{equation}
Here $s$ can be any complex number. We calculate this integral for 2 different imaginary values, which we will then use directly. First, we make the substitution $s\rightarrow\dot{\iota}\omega$, this gives us,
\begin{equation}
    \int_{0}^{\infty} e^{-\dot{\iota}\omega\tau} \, e^{\tau\mathbf{J}} \mathrm{d}\tau = (\dot{\iota}\omega\mathbf{I}-\mathbf{J})^{-1}   \label{eq:aside1}
\end{equation}
Next, we make the substitution $s\rightarrow-\dot{\iota}\omega$, this gives us,
\begin{equation}
    \int_{0}^{\infty} e^{\dot{\iota}\omega\tau} \, e^{\tau\mathbf{J}} \mathrm{d}\tau = -(\dot{\iota}\omega\mathbf{I}+\mathbf{J})^{-1}   \label{eq:aside2}
\end{equation}
Using Eq.~\ref{eq:aside1} and Eq.~\ref{eq:aside2} to simplify $\Romannum{3}$, we get,
\begin{equation}
\begin{split}
    \Romannum{3} &= \int_{0}^{\infty} e^{\dot{\iota} \omega k} e^{\mathbf{J}k} \,\mathbf{L} \, \mathbf{D} \,\mathbf{L}^\top \,\int_{0}^{\infty} e^{-\dot{\iota} \omega \tau} \, e^{\mathbf{J}^\top\tau} \,  \mathrm{d}\tau \, \mathrm{d}k \\
    &= \int_{0}^{\infty} e^{\dot{\iota} \omega k} e^{\mathbf{J}k} \,\mathbf{L} \, \mathbf{D} \,\mathbf{L}^\top \, \mathrm{d}k \, (\dot{\iota}\omega\mathbf{I}-\mathbf{J}^\top)^{-1} \\
    &= \int_{0}^{\infty} e^{\dot{\iota} \omega k} e^{\mathbf{J}k} \, \mathrm{d}k \,\mathbf{L} \, \mathbf{D} \,\mathbf{L}^\top \, (\dot{\iota}\omega\mathbf{I}-\mathbf{J}^\top)^{-1} \\
    &=  -(\dot{\iota}\omega\mathbf{I}+\mathbf{J})^{-1} \,\mathbf{L}  \, \mathbf{D} \,\mathbf{L}^\top \, (\dot{\iota}\omega\mathbf{I}-\mathbf{J}^\top)^{-1} \\
    &=  (\dot{\iota}\omega\mathbf{I}+\mathbf{J})^{-1} \,\mathbf{L} \, \mathbf{D} \,\mathbf{L}^\top \, (-\dot{\iota}\omega\mathbf{I}+\mathbf{J})^{-\top} 
\end{split}
\end{equation}
Now, we simplify the integral $\Romannum{4}$, we first make the variable change $\tau=-\alpha$,
\begin{equation}
\begin{split}
    \Romannum{4} &= -\int_{0}^{\infty} e^{\dot{\iota} \omega k} e^{\mathbf{J}k} \,\mathbf{L}  \, \mathbf{D}  \,\mathbf{L}^\top \, \int_{0}^{k} e^{-\dot{\iota} \omega \tau} \, e^{\mathbf{J}^\top\tau} \,  \mathrm{d}\tau \, \mathrm{d}k\\
    &= -\int_{0}^{\infty} e^{\dot{\iota} \omega k} e^{\mathbf{J}k} \,\mathbf{L}  \, \mathbf{D}  \,\mathbf{L}^\top \, \int_{-k}^{0} e^{\dot{\iota} \omega \alpha} \, e^{-\mathbf{J}^\top\alpha} \,  \mathrm{d}\alpha \, \mathrm{d}k\\
    &= -\int_{0}^{\infty} \int_{-k}^{0} e^{\mathbf{J}k} \,\mathbf{L}  \, \mathbf{D}  \,\mathbf{L}^\top \,  e^{\dot{\iota} \omega (\alpha + k)} \, e^{-\mathbf{J}^\top\alpha} \,  \mathrm{d}\alpha \, \mathrm{d}k
\end{split}
\end{equation}
Next, we make the following variable transformation in the inner integral, $\alpha+k=-m$, this gives,
\begin{equation}
    \Romannum{4} = -\int_{0}^{\infty} \int_{-k}^{0} e^{\mathbf{J}k} \,\mathbf{L}  \, \mathbf{D}  \,\mathbf{L}^\top \,  e^{- \dot{\iota} \omega m} \, e^{\mathbf{J}^\top(m+k)} \,  \mathrm{d}m \, \mathrm{d}k
\end{equation}
Finally, we change the variable of integration $\tau=m$ in the inner integral,
\begin{equation}
\begin{split}
    \Romannum{4} &= -\int_{0}^{\infty} \int_{-k}^{0} e^{\mathbf{J}k} \,\mathbf{L}  \, \mathbf{D}  \,\mathbf{L}^\top \,  e^{- \dot{\iota} \omega \tau} \, e^{\mathbf{J}^\top(\tau+k)} \,  \mathrm{d}\tau \, \mathrm{d}k\\
    &= -\int_{0}^{\infty} \int_{-k}^{0} e^{\mathbf{J}k} \,\mathbf{L}  \, \mathbf{D}  \,\mathbf{L}^\top \, e^{\mathbf{J}^\top(\tau+k)} \,  e^{- \dot{\iota} \omega \tau} \,  \mathrm{d}\tau \, \mathrm{d}k
\end{split}
\end{equation}
The final solution is $\Romannum{1}+\Romannum{3}+\Romannum{4}$,
\begin{equation}
\begin{split}
    \mathbfcal{S}(\omega) =& \int_{0}^{\infty} \int_{-k}^{0} e^{\mathbf{J}k} \,\mathbf{L}  \, \mathbf{D}  \,\mathbf{L}^\top  \,e^{\mathbf{J}^\top(\tau+k)} \, e^{-\dot{\iota} \omega \tau} \, \mathrm{d}\tau \, \mathrm{d}k \nonumber \\&+ (\dot{\iota}\omega\mathbf{I}+\mathbf{J})^{-1} \,\mathbf{L}  \, \mathbf{D}  \,\mathbf{L}^\top \, (-\dot{\iota}\omega\mathbf{I}+\mathbf{J})^{-\top} \\&
    -\int_{0}^{\infty} \int_{-k}^{0} e^{\mathbf{J}k} \,\mathbf{L}  \, \mathbf{D}  \,\mathbf{L}^\top \, e^{\mathbf{J}^\top(\tau+k)} \,  e^{- \dot{\iota} \omega \tau} \,  \mathrm{d}\tau \, \mathrm{d}k \nonumber
\end{split}
\end{equation}

The first and last integrals cancel each other and we are left with the final solution,
\begin{equation}
    \mathbfcal{S}(\omega) = (\dot{\iota}\omega\mathbf{I}+\mathbf{J})^{-1} \,\mathbf{L}  \, \mathbf{D}  \,\mathbf{L}^\top \, (-\dot{\iota}\omega\mathbf{I}+\mathbf{J})^{-\top} \label{eq:matrix_sol}
\end{equation}
We also show that the solution is Hermitian. To do this, take the conjugate transpose of the above expression on both sides, denoted by $*$. Using the properties of conjugate transpose and the fact that $\mathbf{I},\mathbf{J}, \mathbf{L}, \mathbf{D}$ contain purely real elements, we simplify the expression below to find that the conjugate transpose of the matrix is equal to the matrix itself,
\begin{equation}
\begin{split}
    \mathbfcal{S}^{*}(\omega) &= \left[(\dot{\iota}\omega\mathbf{I}+\mathbf{J})^{-1} \,\mathbf{L}  \, \mathbf{D}\,\mathbf{L}^\top \, (-\dot{\iota}\omega\mathbf{I}+\mathbf{J})^{-\top}\right]^{*}\\
    &= ((-\dot{\iota}\omega\mathbf{I}+\mathbf{J})^{-\top})^{*} \, (\mathbf{L}  \, \mathbf{D}\,\mathbf{L}^\top)^{*} \, ((\dot{\iota}\omega\mathbf{I}+\mathbf{J})^{-1})^{*}\\
    &= ((-\dot{\iota}\omega\mathbf{I}+\mathbf{J})^{*})^{-\top} \, (\mathbf{L}  \, \mathbf{D}\,\mathbf{L}^\top)^{*} \, ((\dot{\iota}\omega\mathbf{I}+\mathbf{J})^{*})^{-1}\\
    &= ((-\dot{\iota}\omega\mathbf{I})^{*}+\mathbf{J}^{*})^{-\top} \, (\mathbf{L}^{\top})^{*}  \, \mathbf{D}^{*} \,\mathbf{L}^{*} \, ((\dot{\iota}\omega\mathbf{I})^{*}+\mathbf{J}^{*})^{-1}\\
    &= (\dot{\iota}\omega\mathbf{I}^{*}+\mathbf{J}^*)^{-\top} \, (\mathbf{L}^*)^{\top} \, \mathbf{D} \,\mathbf{L}^{\top} \, (-\dot{\iota}\omega\mathbf{I}^{*}+\mathbf{J}^{*})^{-1}\\
    &= (\dot{\iota}\omega\mathbf{I}+\mathbf{J}^\top)^{-\top} \, (\mathbf{L}^\top)^{\top} \, \mathbf{D} \, \mathbf{L}^{\top} \, (-\dot{\iota}\omega\mathbf{I}+\mathbf{J}^{\top})^{-1}\\
    &= (\dot{\iota}\omega\mathbf{I}+\mathbf{J})^{-1} \, \mathbf{L} \, \mathbf{D} \,\mathbf{L}^{\top} \, (-\dot{\iota}\omega\mathbf{I}+\mathbf{J})^{-\top}\\
    &= \mathbfcal{S}(\omega)
\end{split}
\label{eq:matrix_sol_hermitian}
\end{equation}
Additionally, if the inverse of the input covariance matrix $\mathbf{C} = \mathbf{L} \, \mathbf{D} \, \mathbf{L}^\top$ exists, i.e., it is positive definite ($\mathbf{C}\succ 0$), we can write the solution as,
\begin{equation}
\begin{split}
    \mathbfcal{S}(\omega) &= (\dot{\iota}\omega \mathbf{I} + \mathbf{J})^{-1} \, \mathbf{C} \, (-\dot{\iota}\omega \mathbf{I} + \mathbf{J})^{-\top} \\ 
    &= (\dot{\iota}\omega \mathbf{I} + \mathbf{J})^{-1} \, {\mathbf{C}} \, (-\dot{\iota}\omega \mathbf{I} + \mathbf{J}^{\top})^{-1}\\
    &= \left[(-\dot{\iota}\omega \mathbf{I} + \mathbf{J}^{\top}) \, {\mathbf{C}}^{-1} \, (\dot{\iota}\omega \mathbf{I} + \mathbf{J})\right]^{-1} \\
    &= \left[\mathbf{J}^{\top} {\mathbf{C}}^{-1} \mathbf{J} + \omega^2 {\mathbf{C}}^{-1} + \dot{\iota}\omega (\mathbf{J}^{\top}{\mathbf{C}}^{-1} - {\mathbf{C}}^{-1}\mathbf{J}) \right]^{-1}
\end{split}
\label{eq:matrix_sol_cholesky}
\end{equation}
where the matrix $\left[\mathbf{J}^{\top} {\mathbf{C}}^{-1} \mathbf{J} + \omega^2 {\mathbf{C}}^{-1} + \dot{\iota}\omega (\mathbf{J}^{\top}{\mathbf{C}}^{-1} - {\mathbf{C}}^{-1}\mathbf{J}) \right]$ is Hermitian and positive definite which allows us to use algorithms optimized for the inversion of Hermitian positive definite matrices such as the Cholesky decomposition.

To show that the matrix $\left[\mathbf{J}^{\top} {\mathbf{C}}^{-1} \mathbf{J} + \omega^2 {\mathbf{C}}^{-1} + \dot{\iota}\omega (\mathbf{J}^{\top}{\mathbf{C}}^{-1} - {\mathbf{C}}^{-1}\mathbf{J}) \right]$ is positive definite, we have to show that the matrix is Hermitian and prove the following,
\begin{equation}
\begin{split}
    \mathbf{x}^*\left[\mathbf{J}^{\top} {\mathbf{C}}^{-1} \mathbf{J} + \omega^2 {\mathbf{C}}^{-1} + \dot{\iota}\omega (\mathbf{J}^{\top}{\mathbf{C}}^{-1} - {\mathbf{C}}^{-1}\mathbf{J}) \right]\mathbf{x} > 0; \quad \forall \;\mathbf{x} \in \C^n \symbol{92} \{\mathbf{0}\}
\end{split}
\label{eq:matrix_sol_cholesky_pd}
\end{equation}
The fact that the matrix is Hermitian is easily seen from Eq.~\ref{eq:matrix_sol_hermitian}. To prove the second part, consider any non-zero complex vector $\mathbf{x} \in \C^n \symbol{92} \{\mathbf{0}\}$. Since, $\mathbf{C}$ is positive definite, $\mathbf{C}^{-1}$ is also positive definite, therefore, $\mathbf{x}^*\mathbf{C}^{-1}\mathbf{x}>0$. Since, $(\dot{\iota}\omega \mathbf{I} + \mathbf{J})$ is invertible for any $\omega \in \R$, we can write a transformation $\mathbf{y} = (\dot{\iota}\omega \mathbf{I} + \mathbf{J}) \, \mathbf{x}$. The complex vector $\mathbf{y}$ must satisfy $\mathbf{y}^*\mathbf{C}^{-1}\mathbf{y}>0$. Upon substitution, we get,
\begin{equation}
\begin{split}
    \mathbf{y}^*\mathbf{C}^{-1}\mathbf{y} &> 0 \\
    [(\dot{\iota}\omega \mathbf{I} + \mathbf{J}) \, \mathbf{x}]^* \, \mathbf{C}^{-1} (\dot{\iota}\omega \mathbf{I} + \mathbf{J}) \,\mathbf{x} &> 0 \\
    \mathbf{x}^*\left[(\dot{\iota}\omega \mathbf{I} + \mathbf{J})^* \, \mathbf{C}^{-1} (\dot{\iota}\omega \mathbf{I} + \mathbf{J})\right] \mathbf{x} &> 0 \\
    \mathbf{x}^*\left[(-\dot{\iota}\omega \mathbf{I} + \mathbf{J}^\top) \, \mathbf{C}^{-1} (\dot{\iota}\omega \mathbf{I} + \mathbf{J})\right] \mathbf{x} &> 0 \\
    \mathbf{x}^*\left[\mathbf{J}^{\top} {\mathbf{C}}^{-1} \mathbf{J} + \omega^2 {\mathbf{C}}^{-1} + \dot{\iota}\omega (\mathbf{J}^{\top}{\mathbf{C}}^{-1} - {\mathbf{C}}^{-1}\mathbf{J})\right] \mathbf{x} &> 0 \\
\end{split}
\label{eq:matrix_sol_cholesky_pd2}
\end{equation}
Therefore, $\left[\mathbf{J}^{\top} {\mathbf{C}}^{-1} \mathbf{J} + \omega^2 {\mathbf{C}}^{-1} + \dot{\iota}\omega (\mathbf{J}^{\top}{\mathbf{C}}^{-1} - {\mathbf{C}}^{-1}\mathbf{J}) \right]$ is positive definite.

\newpage

Now, we proceed with the determination of rational functions for the auto-spectrum (diagonal elements) and the cross-spectrum (off-diagonal elements) using the general matrix solution, as described by Eq.~\ref{eq:matrix_sol}. In the following sections, we present two distinct approaches to obtaining the rational function solution for $\mathbfcal{S}(\omega)$. The first approach utilizes a recursive algorithm that calculates the coefficient matrices for the numerator and denominator polynomial coefficients. By employing this method, we can obtain the coefficients for both the auto and cross-spectrum of all variables simultaneously. The second approach involves an element-wise solution, which expresses the auto and cross-spectrum in terms of elementary trace operations applied to the submatrices of $\mathbf{J}$. 

\section{Recursive algorithm solutions}
\label{sec:recursive_algorithm}

\subsection{Spectral density solution}
First, we describe the recursive algorithm to calculate the rational function solution for $\mathbfcal{S}(\omega)$. To make our solution compact, we replace the noise covariance matrix $\mathbf{L} \, \mathbf{D} \,\mathbf{L}^{\top}$ with $\mathbf{C}\succcurlyeq 0  \in \R^{n\times n}$. Now, we write the inverse of complex matrices in terms of the corresponding determinants, $\mathrm{det}(\mathbf{X})$, and adjugate matrices, $\mathrm{adj}(\mathbf{X})$,
\begin{equation}
\begin{split}
    \mathbfcal{S}(\omega) &= (\dot{\iota}\omega\I+\J)^{-1} \, \mathbf{C} \, (-\dot{\iota}\omega\I+\J)^{-\top} \\
    &= \frac{\mathrm{adj}(\J+\dot{\iota} \omega\I) \, \mathbf{C} \, \mathrm{adj}^\top(\mathbf{J}-\dot{\iota} \omega\I)}{\mathrm{det}(\J+\dot{\iota} \omega\I)\, \mathrm{det}(\mathbf{J}-\dot{\iota} \omega\I)} \\
    &= \frac{\mathbf{Z}(\omega)}{Q(\omega)}  \label{eq:gen1} 
\end{split}
\end{equation}
Here, $\mathbf{Z}(\omega)$ is a complex polynomial matrix representing the numerator of the solution, and $Q(\omega)$ is an even-powered polynomial of degree $2n$ representing the denominator of the solution. When we express the adjugate of the matrices in terms of the corresponding cofactor matrices, it is evident that $\mathbf{Z}(\omega)$ contains complex polynomials of degree $2n-2$ with even powers of $\omega$ being real and odd powers of $\omega$ being imaginary. For the denominator, we have, 
\begin{equation}
    Q(\omega) = |\mathbf{J}+\dot{\iota} \omega\mathbf{I}| |\mathbf{J}-\dot{\iota} \omega\mathbf{I}| 
\end{equation}
If $\lambda_i$ are the eigenvalues of the matrix $\mathbf{J}$, using Eq.~\ref{eq:det_prod}, we have,
\begin{equation}
    |\mathbf{J} \pm \dot{\iota} \omega| = \prod_{j=1}^n (\lambda_j \pm \dot{\iota}\omega)
\end{equation}
Therefore,
\begin{equation}
\begin{split}
    Q(\omega) &= \prod_{j=1}^n (\lambda_j + \dot{\iota}\omega)(\lambda_j - \dot{\iota}\omega) \\
    &= \prod_{j=1}^n (\lambda_j^2 + \omega^2) \label{eq:recursive0}
\end{split}
\end{equation}
Therefore, we can write the equation as follows,
\begin{equation}
\begin{split}
    \mathbfcal{S}(\omega) &= \frac{\Pb(\omega) + \dot{\iota}\omega \Pb^\prime(\omega)}{Q(\omega)} \\
    &= \frac{\Pb_0 + \Pb_1 \omega^2 +...+ \Pb_{n-1}\omega^{2n-2} + \dot{\iota}\omega \left(\Pb_0^\prime + \Pb_1^\prime \omega^2 +...+ \Pb_{n-2}^\prime \omega^{2n-4} \right)}{q_0 + q_1 \omega^2 +...+q_n\omega^{2n}}\label{eq:recursive2} 
\end{split}
\end{equation}

\begin{theorem} 
\label{theorem:0}
The noise power spectral density matrix of an LTI system is a complex-valued rational function of the form
\begin{equation}
\begin{split}
    \mathbfcal{S}(\omega) &=(\dot{\iota}\omega\I+\J)^{-1} \, \mathbf{C} \, (-\dot{\iota}\omega\I+\J)^{-\top} \\
      &= \frac{\sum\limits_{\alpha=0}^{n-1}\Pb_\alpha \omega^{2\alpha} + \dot{\iota}\omega \sum\limits_{\alpha=0}^{n-2}\Pb_\alpha^\prime \omega^{2\alpha}}{\sum\limits_{\alpha=0}^{n}q_\alpha \omega^{2\alpha}} \label{eq:recursive3}
\end{split}
\end{equation}
The numerator's matrix coefficients are given by the recursive equations,
\begin{equation}
\begin{split}
    \Pb_{\alpha-1}^\prime &= \J \Pb_{\alpha} - \Pb_{\alpha} \J^\top - \J \Pb_{\alpha}^\prime \J^\top \\
    \Pb_{\alpha-1} &= q_{\alpha}\mathbf{C} + \Pb_{\alpha-1}^\prime \J^\top - \J \Pb_{\alpha-1}^\prime - \J \Pb_{\alpha} \J^\top
\end{split}
\label{eq:num-rec-coefficients}
\end{equation}
for $\alpha \in \{1,\dots, n\}$, starting from $\alpha =n$ with $\Pb_{n}=\Pb_{n}^\prime=\mathbf{0}$. The scalar coefficients of the denominator can be recursively calculated using,
\begin{equation}
\begin{split}
    q_{\alpha} = \frac{1}{n-\alpha}\left[\Tr\left(\J \Q_{\alpha-1}^\prime \right) + \Tr\left(\J \Q_\alpha \J^\top \right)\right]
\end{split}
\end{equation}
if $\alpha<n$, and $q_n = 1$. The coefficients $\Q_{\alpha-1}^\prime$ and $\Q_{\alpha}$ are given in turn by the recursive equations,
\begin{equation}
\begin{split}
    \Q_{\alpha-1}^\prime &= \J \Q_{\alpha} - \Q_{\alpha} \J^\top - \J \Q_{\alpha}^\prime \J^\top \\
    \Q_{\alpha-1} &= q_{\alpha} \mathbf{I}+ \Q_{\alpha-1}^\prime \J^\top - \J \Q_{\alpha-1}^\prime - \J \Q_{\alpha} \J^\top
\end{split}
\label{eq:suppl-rec-coefficients}
\end{equation}
for $\alpha \in \{0,1, \dots, n\}$, starting from $\alpha=n$ with $\Q_n = \Q_n^\prime =\mathbf{0}$ and $ \Q_{-1} = \Q_{-1}^\prime =\mathbf{0}$.
\end{theorem}

\begin{proof}
First, we establish a recursive solution for the coefficient matrices of the numerator. Subsequently, we derive the second part of the theorem, presenting a set of recursive equations to calculate the coefficients of the polynomial in the denominator.

\subsubsection{Coefficient matrices of the numerator}
Consider Eq.~\ref{eq:recursive3},
\begin{equation}
\begin{split}
    (\dot{\iota}\omega\I+\J)^{-1} \, \mathbf{C} \, (-\dot{\iota}\omega\I+\J)^{-\top} 
    = \frac{\sum\limits_{\alpha=0}^{n-1}\Pb_\alpha \omega^{2\alpha} + \dot{\iota}\omega \sum\limits_{\alpha=0}^{n-2}\Pb_\alpha^\prime \omega^{2\alpha}}{\sum\limits_{\alpha=0}^{n}q_\alpha \omega^{2\alpha}}
\end{split}
\end{equation}
Upon rearrangement and left multiplying both sides with $(\J + \dot{\iota}\omega\I)$ and right multiplying with $(\J - \dot{\iota}\omega\I)^\top$, we get,
\begin{equation}
\begin{split}
     \bigl(q_0 + q_1 \omega^2 +...+q_n\omega^{2n}\bigr) \mathbf{C}  =& \bigl(\J + \dot{\iota}\omega\I\bigr) \biggl(\Pb_0 + \Pb_1 \omega^2 +...+ \Pb_{n-1}\omega^{2n-2} \\
     &+ \dot{\iota}\omega \left(\Pb_0^\prime + \Pb_1^\prime \omega^2 +...+ \Pb_{n-2}^\prime \omega^{2n-4} \right)\biggr) \bigl(\J - \dot{\iota}\omega\I\bigr)^\top \\
     =& \J \Pb_0 \J^\top + \J \Pb_1 \J^\top \omega^2 + ...+ \J \Pb_{n-1} \J^\top \omega^{2n-2} \\
     &-\left(\Pb_0^\prime\J^\top \omega^2 + \Pb_1^\prime\J^\top \omega^4 +...+ \Pb_{n-2}^\prime\J^\top \omega^{2n-2} \right) \\
     &+\left(\J\Pb_0^\prime \omega^2 + \J\Pb_1^\prime \omega^4 + ... + \J\Pb_{n-2}^\prime \omega^{2n-2}\right) \\
     &+ \left(\Pb_0 \omega^2 + \Pb_1 \omega^4 + ... + \Pb_{n-1}\omega^{2n}\right)\\
     &+\dot{\iota}\left(\J \Pb_0^\prime \J^\top \omega + \J \Pb_1^\prime \J^\top \omega^3 + ...+ \J \Pb_{n-2}^\prime \J^\top \omega^{2n-3}\right) \\
     & +\dot{\iota} \left(\Pb_0 \J^\top \omega + \Pb_1 \J^\top \omega^3 + ...+ \Pb_{n-1} \J^\top \omega^{2n-1} \right) \\
     & -\dot{\iota} \left( \J \Pb_0 \omega + \J \Pb_1\omega^3 + ...+ \J \Pb_{n-1}\omega^{2n-1} \right) \\
     & +\dot{\iota}  \left(\Pb_0^\prime \omega^3+ \Pb_1^\prime \omega^5 + ... + \Pb_{n-2}^\prime \omega^{2n-1}\right)
     \label{eq:recursive4}
\end{split}
\end{equation}
To arrive at the recursive formula for $\Pb_\alpha$ and $\Pb_\alpha^\prime$, we compare the like powers of $\omega$ on both sides of the equation, starting from the largest power of $\omega$. 
\begin{equation}
\begin{split}
    \Pb_{n-1} &= q_n\mathbf{C} \\
    \Pb_{n-2}^\prime &= \J \Pb_{n-1} - \Pb_{n-1} \J^\top\\
    \Pb_{n-2} &= q_{n-1}\mathbf{C} + \Pb_{n-2}^\prime \J^\top - \J \Pb_{n-2}^\prime - \J \Pb_{n-1} \J^\top \\ 
    \Pb_{n-3}^\prime &= \J \Pb_{n-2} - \Pb_{n-2} \J^\top - \J \Pb_{n-2}^\prime \J^\top\\
    &\ldots,\\
    \Pb_{0}^\prime &= \J \Pb_{1} - \Pb_{1} \J^\top - \J \Pb_{1}^\prime \J^\top\\
    \Pb_{0} &= q_{1}\mathbf{C} + \Pb_{0}^\prime \J^\top - \J \Pb_{0}^\prime - \J \Pb_{1} \J^\top\\
    \mathbf{0} &= \J \Pb_{0} - \Pb_{0} \J^\top - \J \Pb_{0}^\prime \J^\top\\
    \mathbf{0} &=  q_0 \mathbf{C} - \J \Pb_0 \J^\top
    \label{eq:recursive4}
\end{split}
\end{equation}
We also verify from these expressions that $\Pb_{\alpha}^\prime$ matrices are anti-symmetric whereas, $\Pb_{\alpha}$ matrices are symmetric. Here the last two equations are redundant and can be used to check for the accuracy of the numerical computation. In general for even powers of $\omega$, the coefficients of $\omega^{2\alpha}$ for $0 \leq \alpha \leq n$ are given by,
\begin{equation}
\begin{split}
    \Pb_{\alpha-1} &= q_{\alpha}\mathbf{C} + \Pb_{\alpha-1}^\prime \J^\top - \J \Pb_{\alpha-1}^\prime - \J \Pb_{\alpha} \J^\top
    \label{eq:recursive_Peven}
\end{split}
\end{equation}
Note that $\Pb_{n}=\mathbf{0}$ and $\Pb_{-1}=\mathbf{0}$, therefore the above equation is valid for $\alpha \in \{0, n\}$. Similarly, in general for odd powers of $\omega$, the coefficients of $\omega^{2\alpha+1}$ for $0 \leq \alpha \leq n-1$ are given by,
\begin{equation}
\begin{split}
    \Pb_{\alpha-1}^\prime &= \J \Pb_{\alpha} - \Pb_{\alpha} \J^\top - \J \Pb_{\alpha}^\prime \J^\top
    \label{eq:recursive_Podd}
\end{split}
\end{equation}
Note that $\Pb^\prime_{n-1}=\mathbf{0}$ and $\Pb^\prime_{-1}=\mathbf{0}$, therefore the above equation is still valid for $\alpha \in \{0, n\}$.

\subsubsection{Coefficients of the denominator}
Now, onto the second part of the theorem to recursively calculate $q_\alpha$. Note that $q_n=1$ can be easily deduced from Eq.~\ref{eq:recursive0}. Consider the Laplace transform, $\mathcal{L}\{f(t)\}(s)$, of the derivative of the exponential of the matrix $e^{\dot{\iota} t \J}$. For notational convenience, we replace the complex frequency, $s$, in the Laplace transform with $\omega$ and denote $\mathcal{L}\{f(t)\}(\omega)$ as $\mathcal{L}\{f(t)\}$.
\begin{equation}
    \mathcal{L}\left\{\frac{\mathrm{d} }{\mathrm{d} t} e^{\dot{\iota} t\J}\right\} = \mathcal{L}\left\{\dot{\iota} \J e^{\dot{\iota} t\J}\right\} = \dot{\iota} \J (\omega\I - \dot{\iota} \J)^{-1} 
\end{equation}
Note that we have used the property of the resolvent of a matrix (Eq.~\ref{eq:resolvent_appendix}). Similarly, we have,
\begin{equation}
    \mathcal{L}\left\{\frac{\mathrm{d} }{\mathrm{d} t} e^{-\dot{\iota} t\J}\right\} = \mathcal{L}\left\{-\dot{\iota} \J e^{-\dot{\iota} t\J}\right\} = -\dot{\iota} \J (\omega\I + \dot{\iota} \J)^{-1}
\end{equation}
Now using the property of the Laplace transform of a derivative of a function, the $\mathrm{LHS}$ of the equations above becomes,
\begin{equation}
    \omega \mathcal{L}\left\{e^{\dot{\iota} t\J}\right\} - \I = \dot{\iota} \J (\omega\I - \dot{\iota} \J)^{-1} \label{eq:recursive_laplace1}
\end{equation}
and,
\begin{equation}
    \omega \mathcal{L}\left\{e^{-\dot{\iota} t\J}\right\} - \I = -\dot{\iota} \J (\omega\I + \dot{\iota} \J)^{-1}.
    \label{eq:recursive_laplace2}
\end{equation}
Now consider the following matrix product,
\begin{equation}
    \left(\omega \mathcal{L}\left\{e^{\dot{\iota} t\J}\right\} - \I\right) \mathbf{C} \left(\omega \mathcal{L}\left\{e^{-\dot{\iota} t\J}\right\} - \I\right)^\top
\end{equation}
Using Eq.~\ref{eq:recursive_laplace1} \& \ref{eq:recursive_laplace2}, we can also write this product as,
\begin{equation}
    \left(\omega \mathcal{L}\left\{e^{\dot{\iota} t\J}\right\} - \I\right) \mathbf{C} \left(\omega \mathcal{L}\left\{e^{-\dot{\iota} t\J}\right\} - \I\right)^\top = \left(\dot{\iota} \J (\omega\I - \dot{\iota} \J)^{-1} \right) \mathbf{C} \left(-\dot{\iota} \J (\omega\I + \dot{\iota} \J)^{-1} \right)^\top
\end{equation}
We first consider and simplify the $\mathrm{RHS}$ of the equation. 
\begin{equation}
\begin{split}
    \mathrm{RHS} &= \left(\dot{\iota} \J (\omega\I - \dot{\iota} \J)^{-1} \right) \mathbf{C} \left(-\dot{\iota} \J (\omega\I + \dot{\iota} \J)^{-1} \right)^\top \\
    &= \J \left[\, (\dot{\iota} \omega\I + \J)^{-1} \, \mathbf{C} \, (-\dot{\iota} \omega\I + \J)^{-\top}\right]\, \J^\top
\end{split}
\end{equation}
Identifying the matrix product in the rectangular brackets from Eq.~\ref{eq:recursive3}, we can write the $\mathrm{RHS}$ as,
\begin{equation}
\begin{split}
    \mathrm{RHS} =  \frac{1}{Q(\omega)}\J \left[\sum\limits_{\alpha=0}^{n-1}\Pb_\alpha \omega^{2\alpha} + \dot{\iota}\omega \sum\limits_{\alpha=0}^{n-2}\Pb_\alpha^\prime \omega^{2\alpha} \right]\, \J^\top \label{eq:recursive_RHS}
\end{split}
\end{equation}
Now we simplify the $\mathrm{LHS}$ by first expanding the product and then using the properties of the Laplace transform,
\begin{equation}
\begin{split}
    \mathrm{LHS} &= \left(\omega \mathcal{L}\left\{e^{\dot{\iota} t\J}\right\} - \I\right) \mathbf{C} \left(\omega \mathcal{L}\left\{e^{-\dot{\iota} t\J}\right\} - \I\right)^\top \\
    &= \omega^2 \mathcal{L}\left\{e^{\dot{\iota} t\J}\right\} \, \mathbf{C} \, \mathcal{L}\left\{e^{-\dot{\iota} t\J}\right\}^\top - \omega \mathcal{L}\left\{e^{\dot{\iota} t\J}\right\} \mathbf{C} - \omega \, \mathbf{C}\, \mathcal{L}\left\{e^{-\dot{\iota} t\J}\right\}^\top \, +\mathbf{C}\\
    &= \omega^2 (\omega\I - \dot{\iota} \J)^{-1}\, \mathbf{C}\, (\omega\I + \dot{\iota} \J)^{-\top} - \omega \mathcal{L}\left\{e^{\dot{\iota} t\J}\right\} \mathbf{C} - \omega \, \mathbf{C}\, \mathcal{L}\left\{e^{-\dot{\iota} t\J}\right\}^\top \, +\mathbf{C}\\
    &= \omega^2\left[ (\dot{\iota} \omega\I + \J)^{-1}\, \mathbf{C}\, (-\dot{\iota} \omega\I + \J)^{-\top}\right] - \omega \mathcal{L}\left\{e^{\dot{\iota} t\J}\right\} \mathbf{C} - \omega \, \mathbf{C}\, \mathcal{L}\left\{e^{-\dot{\iota} t\J}\right\}^\top \, +\mathbf{C}
\end{split}
\end{equation}
Identifying the matrix product in the rectangular brackets from Eq.~\ref{eq:recursive3}, we can write the $\mathrm{LHS}$ as,
\begin{equation}
\begin{split}
    \mathrm{LHS} = \frac{\omega^2}{Q(\omega)}\left[ \sum\limits_{\alpha=0}^{n-1}\Pb_\alpha \omega^{2\alpha} + \dot{\iota}\omega \sum\limits_{\alpha=0}^{n-2}\Pb_\alpha^\prime \omega^{2\alpha}\right] - \omega \mathcal{L}\left\{e^{\dot{\iota} t\J}\right\} \mathbf{C} - \omega \, \mathbf{C}\, \mathcal{L}\left\{e^{-\dot{\iota} t\J}\right\}^\top \, +\mathbf{C} \label{eq:recursive_LHS}
\end{split}
\end{equation}
Therefore, using Eq.~\ref{eq:recursive_RHS} \& \ref{eq:recursive_LHS}, we arrive at the following equality,
\begin{equation}
    \frac{1}{Q(\omega)}\J \left[\sum\limits_{\alpha=0}^{n-1}\Pb_\alpha \omega^{2\alpha} + \dot{\iota}\omega \sum\limits_{\alpha=0}^{n-2}\Pb_\alpha^\prime \omega^{2\alpha} \right]\, \J^\top = \frac{\omega^2}{Q(\omega)}\left[ \sum\limits_{\alpha=0}^{n-1}\Pb_\alpha \omega^{2\alpha} + \dot{\iota}\omega \sum\limits_{\alpha=0}^{n-2}\Pb_\alpha^\prime \omega^{2\alpha}\right] - \omega \mathcal{L}\left\{e^{\dot{\iota} t\J}\right\} \mathbf{C} - \omega \, \mathbf{C}\, \mathcal{L}\left\{e^{-\dot{\iota} t\J}\right\}^\top \, +\mathbf{C} \label{eq:recursive_both}
\end{equation}
Since, $\mathbf{C}$ is a covariance matrix, it is positive semi-definite. As we know that the coefficients of the denominator depend only on the matrix $\J$ and not on the covariance matrix $\mathbf{C}$, as seen in Eq.~\ref{eq:recursive0}, the recursive solution of $q$ is valid for any semi-positive definite matrix $\mathbf{C}$. We first derive the recursive solution for $q$ by making the assumption that $\mathbf{C}\succ 0$, and therefore $\Ci$ exists. Multiplying both sides by $\Ci$ in Eq.~\ref{eq:recursive_both}, and taking the trace, we get,
\begin{equation}
\begin{split}
    \Tr\left(\frac{\Ci \J}{Q(\omega)}  \left[\sum\limits_{\alpha=0}^{n-1}\Pb_\alpha \omega^{2\alpha} + \dot{\iota}\omega \sum\limits_{\alpha=0}^{n-2}\Pb_\alpha^\prime \omega^{2\alpha} \right]\, \J^\top \right) =& \Tr\left(\frac{\omega^2 \Ci}{Q(\omega)}\left[ \sum\limits_{\alpha=0}^{n-1}\Pb_\alpha \omega^{2\alpha} + \dot{\iota}\omega \sum\limits_{\alpha=0}^{n-2}\Pb_\alpha^\prime \omega^{2\alpha}\right] \right)\\& - \omega\Tr\left( \Ci\mathcal{L}\left\{e^{\dot{\iota} t\J}\right\} \mathbf{C} +  \mathcal{L}\left\{e^{-\dot{\iota} t\J}\right\}^\top\right) +\Tr\left(\mathbf{I} \right)\label{eq:recursive_both1}
\end{split}
\end{equation}
Using the properties of trace, $\Tr(\mathbf{A}^\top)=\Tr(\mathbf{A})$ and $\Tr(\Ci\mathbf{A}\mathbf{C})=\Tr(\mathbf{A})$, we can simplify the equation above to,
\begin{equation}
\begin{split}
    \Tr\left(\Ci \J  \left[\sum\limits_{\alpha=0}^{n-1}\Pb_\alpha \omega^{2\alpha} + \dot{\iota} \sum\limits_{\alpha=0}^{n-2}\Pb_\alpha^\prime \omega^{2\alpha+1} \right] \J^\top \right) =& \Tr\left( \Ci\left[ \sum\limits_{\alpha=0}^{n-1}\Pb_\alpha \omega^{2\alpha+2} + \dot{\iota} \sum\limits_{\alpha=0}^{n-2}\Pb_\alpha^\prime \omega^{2\alpha+3}\right] \right)\\& - \omega Q(\omega)\Tr\left( \mathcal{L}\left\{e^{\dot{\iota} t\J}\right\} +  \mathcal{L}\left\{e^{-\dot{\iota} t\J}\right\}\right) +Q(\omega)n \label{eq:recursive_both2}
\end{split}
\end{equation}
Using Theorem~\ref{theorem:1}, we find,
\begin{equation}
    \Tr\left(\Ci \J  \left[\sum\limits_{\alpha=0}^{n-1}\Pb_\alpha \omega^{2\alpha} + \dot{\iota} \sum\limits_{\alpha=0}^{n-2}\Pb_\alpha^\prime \omega^{2\alpha+1} \right] \J^\top \right) = \Tr\left( \Ci\left[ \sum\limits_{\alpha=0}^{n-1}\Pb_\alpha \omega^{2\alpha+2} + \dot{\iota} \sum\limits_{\alpha=0}^{n-2}\Pb_\alpha^\prime \omega^{2\alpha+3}\right] \right) - \omega Q^\prime(\omega)+Q(\omega)n \label{eq:recursive_both3s}
\end{equation}
Finally, expanding the polynomial corresponding to the denominator, we get,
\begin{equation}
    \Tr\left(\Ci \J  \left[\sum\limits_{\alpha=0}^{n-1}\Pb_\alpha \omega^{2\alpha} + \dot{\iota} \sum\limits_{\alpha=0}^{n-2}\Pb_\alpha^\prime \omega^{2\alpha+1} \right] \J^\top \right) = \Tr\left( \Ci\left[ \sum\limits_{\alpha=0}^{n-1}\Pb_\alpha \omega^{2\alpha+2} + \dot{\iota} \sum\limits_{\alpha=0}^{n-2}\Pb_\alpha^\prime \omega^{2\alpha+3}\right] \right) + \sum\limits_{\alpha=0}^{n} (n-2\alpha) q_\alpha \omega^{2\alpha} \label{eq:recursive_both3s}
\end{equation}
Now, we compare the coefficients of different powers of $\omega$. First, for odd powers of $\omega$, we consider the coefficient of $\omega^{2\alpha+1}$ for $0 \leq \alpha \leq n-1 $. We get the following equations,
\begin{equation}
\begin{split}
    \Tr\left(\Ci \Pb^\prime_{\alpha-1}\right) &= \Tr\left(\Ci \J \Pb^\prime_{\alpha} \J^\top\right)
    \label{eq:recursive_q0}
\end{split}
\end{equation}
Note that $\Pb^\prime_{n-1}=\mathbf{0}$ and $\Pb^\prime_{-1}=\mathbf{0}$, therefore the above equation is still valid for $\alpha \in \{0, n-1\}$. Consider the $\mathrm{LHS}$ of the equation above,
\begin{equation}
\begin{split}
    \mathrm{LHS} &= \Tr\left(\Ci \Pb^\prime_{\alpha-1}\right) \\
    &= \Tr\left(\left(\Ci \Pb^\prime_{\alpha-1}\right)^\top\right) \\
    &= \Tr\left( \left(\Pb^\prime_{\alpha-1}\right)^\top \left(\Ci\right)^{\top} \right) \\
    &= \Tr\left( -\Pb^\prime_{\alpha-1} \Ci \right) \\
    &= -\Tr\left( \Ci \Pb^\prime_{\alpha-1}  \right) \\
    &= - \mathrm{LHS}
\end{split}
\end{equation}
Therefore, $\mathrm{LHS}=0$. Note that, we have used the following facts: $\Pb^\prime_{\alpha-1}$ is antisymmetric, and $\Ci$ is symmetric. Similarly, we can prove that $\mathrm{RHS}=0$. The odd powers of $\omega$ in Eq.~\ref{eq:recursive_both3s} don't give us an insight into the coefficients of the denominator. Therefore, we compare coefficients of even powers of $\omega$. Upon comparing the coefficients of $\omega^{2\alpha}$ for $0 \leq \alpha < n$, we get the following equation,
\begin{equation}
\begin{split}
    0 = \Tr\left(\Ci \Pb_{\alpha-1} \right)  + (n-2\alpha) q_\alpha - \Tr\left(\Ci \J \Pb_\alpha \J^\top \right)
    \label{eq:recursive_q10}
\end{split}
\end{equation}
From Eq.~\ref{eq:recursive_Peven}, we know that $\Pb_{\alpha-1}$ depends on $q_\alpha$, therefore, solving Eq.~\ref{eq:recursive_q10} directly for $q_\alpha$ does not lead to a recursive solution. Thus, we substitute the expression for $\Pb_{\alpha-1}$ from Eq.~\ref{eq:recursive_Peven} to get,
\begin{equation}
\begin{split}
    0 &= \Tr\left(\Ci \left(q_{\alpha}\mathbf{C} + \Pb_{\alpha-1}^\prime \J^\top - \J \Pb_{\alpha-1}^\prime - \J \Pb_{\alpha} \J^\top \right)\right)  + (n-2\alpha) q_\alpha - \Tr\left(\Ci \J \Pb_\alpha \J^\top \right) \\
    0 &= n q_\alpha + \Tr\left( \Ci \Pb_{\alpha-1}^\prime \J^\top \right) - \Tr\left(\Ci \J \Pb_{\alpha-1}^\prime \right) + (n-2\alpha) q_\alpha - 2 \Tr\left(\Ci \J \Pb_\alpha \J^\top \right)
    \label{eq:recursive_q11}
\end{split}
\end{equation}
Solving the equation for $q_\alpha$,
\begin{equation}
\begin{split}
    q_\alpha &= \frac{1}{n-\alpha}\left[\frac{\Tr\left(\Ci \J \Pb_{\alpha-1}^\prime \right) - \Tr\left( \Ci \Pb_{\alpha-1}^\prime \J^\top \right) }{2} + \Tr\left(\Ci \J \Pb_\alpha \J^\top \right)\right] \\
    &= \frac{1}{n-\alpha}\left[\frac{\Tr\left(\Ci \J \Pb_{\alpha-1}^\prime \right) - \Tr\left( \J  \left(\Pb_{\alpha-1}^\prime\right)^\top \left(\Ci\right)^\top \right) }{2} + \Tr\left(\Ci \J \Pb_\alpha \J^\top \right)\right] \\
    &= \frac{1}{n-\alpha}\left[\frac{\Tr\left(\Ci \J \Pb_{\alpha-1}^\prime \right) - \Tr\left(-\Ci \J \Pb_{\alpha-1}^\prime \right) }{2} + \Tr\left(\Ci \J \Pb_\alpha \J^\top \right)\right] \\
    &= \frac{1}{n-\alpha}\left[\Tr\left(\Ci \J \Pb_{\alpha-1}^\prime \right) + \Tr\left(\Ci \J \Pb_\alpha \J^\top \right)\right]
    \label{eq:recursive_q12}
\end{split}
\end{equation}
Since the solution relies on the inverse of the noise covariance matrix $\mathbf{C}$, it might lead to the accumulation of errors in the recursive solution, especially when the dimensionality of the system is large. We also made the assumption that $\mathbf{C}$ was positive definite to be able to calculate the inverse, but that might not be the case for any general stochastic system. Since the solution is valid for any positive semi-definite matrix $\mathbf{C}$, to circumvent the problems above, we choose $\mathbf{C}=\mathbf{I}$, such that we have,
\begin{equation}
\begin{split}
    q_\alpha = \frac{1}{n-\alpha}\left[\Tr\left(\J \Q_{\alpha-1}^\prime \right) + \Tr\left(\J \Q_\alpha \J^\top \right)\right]
    \label{eq:recursive_q13}
\end{split}
\end{equation}
for $0 \leq \alpha < n$ and $\Q_{-1}^\prime = \mathbf{0}$. Here, the matrices $\Q_{\alpha-1}^\prime$ and $\Q_{\alpha}$ can be interpreted as the coefficient matrices of the numerator of the twin problem given by,
\begin{equation}
\begin{split}
    \mathbfcal{S}(\omega) &= (\dot{\iota}\omega\I+\J)^{-1} \, (-\dot{\iota}\omega\I+\J)^{-\top} \\
    &= \frac{\Q(\omega) + \dot{\iota}\omega \Q^\prime(\omega)}{Q(\omega)} \\
    &= \frac{\Q_0 + \Q_1 \omega^2 +...+ \Q_{n-1}\omega^{2n-2} + \dot{\iota}\omega \left(\Q_0^\prime + \Q_1^\prime \omega^2 +...+ \Q_{n-2}^\prime \omega^{2n-4} \right)}{q_0 + q_1 \omega^2 +...+q_n\omega^{2n}}\label{eq:recursive14} 
\end{split}
\end{equation}
Therefore, using Eq.~\ref{eq:recursive_Peven} \& \ref{eq:recursive_Podd}, we can compute the coefficient matrices $\Q_{\alpha-1}^\prime$ and $\Q_{\alpha}$ recursively,
\begin{equation}
\begin{split}
    \Q_{\alpha-1}^\prime &= \J \Q_{\alpha} - \Q_{\alpha} \J^\top - \J \Q_{\alpha}^\prime \J^\top \\
    \Q_{\alpha-1} &= q_{\alpha} \mathbf{I}+ \Q_{\alpha-1}^\prime \J^\top - \J \Q_{\alpha-1}^\prime - \J \Q_{\alpha} \J^\top 
    \label{eq:recursive_Q_num}
\end{split}
\end{equation}
by initiating the recursion at $\alpha=n$ and iteratively working backwards until $\alpha=1$.
\end{proof}

\begin{theorem} 
\label{theorem:1}
Let $\Tr(\mathbf{J})$ denote the trace of the matrix $\mathbf{J}$, $\mathcal{L}\{f(t)\}$ denote the Laplace transform of $f(t)$ with the complex variable $\omega$ and $Q^\prime(\omega)$ be the derivative of the denominator $Q(\omega)$ with respect to $\omega$, then we have,
\begin{equation}
    \Tr\left(\mathcal{L}\left\{e^{\dot{\iota} t \J} + e^{-\dot{\iota} t \J} \right\}\right) = \frac{Q^\prime(\omega)}{Q(\omega)}.
\end{equation}
\end{theorem}
\begin{proof}
Consider $\mathrm{LHS}$,
\begin{equation}
\begin{split}
    \Tr\left(\mathcal{L}\left\{e^{\dot{\iota} t \J} + e^{-\dot{\iota} t \J} \right\}\right) 
    &= \mathcal{L}\left\{\Tr\left(e^{\dot{\iota} t \J} + e^{-\dot{\iota} t \J} \right)\right\} \\
    &= \mathcal{L}\left\{\Tr\left( \sum_{r=0}^\infty \frac{t^r (\dot{\iota}\J)^r}{r!} +  \sum_{r=0}^\infty \frac{t^r (-\dot{\iota}\J)^r}{r!}\right)\right\} \\
    &= \mathcal{L}\left\{ \sum_{r=0}^\infty \Tr\left(\frac{t^r (\dot{\iota}\J)^r}{r!} \right)+  \sum_{r=0}^\infty \Tr\left(\frac{t^r (-\dot{\iota}\J)^r}{r!} \right) \right\}
\end{split}
\end{equation}
If $\lambda_j$ are the eigenvalues of $\J$, using Lemma~\ref{lemma:1} we can write the equation above as,
\begin{equation}
\begin{split}
    \Tr\left(\mathcal{L}\left\{e^{\dot{\iota} t \J} + e^{-\dot{\iota} t \J} \right\}\right) 
    &= \mathcal{L}\left\{ \sum_{r=0}^\infty \sum_{j=1}^n \frac{t^r (\dot{\iota} \lambda_j)^r}{r!} +  \sum_{r=0}^\infty \sum_{j=1}^n \frac{t^r (-\dot{\iota} \lambda_j)^r}{r!} \right\} \\
    &= \mathcal{L}\left\{  \sum_{j=1}^n e^{\dot{\iota} t \lambda_j} +  \sum_{j=1}^n e^{-\dot{\iota} t \lambda_j} \right\} \\
    &= \sum_{j=1}^n \frac{1}{\omega - \dot{\iota} \lambda_j} +  \sum_{j=1}^n \frac{1}{\omega + \dot{\iota} \lambda_j}  \\
    &= \sum_{j=1}^n \frac{2\omega}{\omega^2 + \lambda_j^2}
\end{split}
\end{equation}
Finally using Lemma~\ref{lemma:2}, we have,
\begin{equation}
\begin{split}
    \Tr\left(\mathcal{L}\left\{e^{\dot{\iota} t \J} + e^{-\dot{\iota} t \J} \right\}\right) 
    &= \frac{Q^\prime(\omega)}{Q(\omega)} \\
    &= \mathrm{RHS}
\end{split}
\end{equation}
\end{proof}

\begin{lemma}
\label{lemma:1}
For $\beta \in \mathbb{C}$, $k \in \mathbb{W}$ and $\mathbf{A} \in \R^{n\times n}$ with eigenvalues $\lambda_i$
\begin{equation}
    \Tr\left(\left(\beta \mathbf{A}\right)^k\right) = \sum_{i=1}^n \left(\beta\lambda_i\right)^k
\end{equation}
\end{lemma}
\begin{proof}
Let $\mathbf{A}$ be a square matrix of size $n\times n$ with eigenvalues $\lambda_1, \lambda_2, \ldots, \lambda_n$. The eigenvalues of $\beta\mathbf{A}$ are $\beta\lambda_1, \beta\lambda_2, \ldots, \beta\lambda_n$. Therefore, the eigenvalues of $(\beta\mathbf{A})^k$ are, $(\beta\lambda_1)^k, (\beta\lambda_2)^k, \ldots, (\beta\lambda_n)^k$. Finally, we use the fact that the trace of a given matrix is the sum of its eigenvalues, therefore we have,
\begin{equation}
    \Tr\left(\left(\beta \mathbf{A}\right)^k\right) = \sum_{i=1}^n \left(\beta\lambda_i\right)^k
\end{equation}
\end{proof}

\begin{lemma}
\label{lemma:2}
If $Q(\omega)=\prod_{j=1}^n (\lambda_j^2 + \omega^2)$, then
\begin{equation}
    \sum_{j=1}^n \frac{2\omega}{\omega^2 + \lambda_j^2} = \frac{Q^\prime(\omega)}{Q(\omega)}
\end{equation}
\end{lemma}
\begin{proof}
Consider $\log Q(\omega)$,
\begin{equation}
\begin{split}
    \log Q(\omega) &= \log \prod_{j=1}^n (\lambda_j^2 + \omega^2) \\
    &= \sum_{j=1}^n \log (\lambda_j^2 + \omega^2)
\end{split}
\end{equation}
Taking the derivative w.r.t. $\omega$ on both sides to get,
\begin{equation}
\begin{split}
    \frac{Q^\prime(\omega)}{Q(\omega)} &= \sum_{j=1}^n \frac{2\omega}{\omega^2 + \lambda_j^2}
\end{split}
\end{equation}
\end{proof}

\newpage
\begin{figure*}[h]
    \centering
    \includegraphics[width=0.9\linewidth]{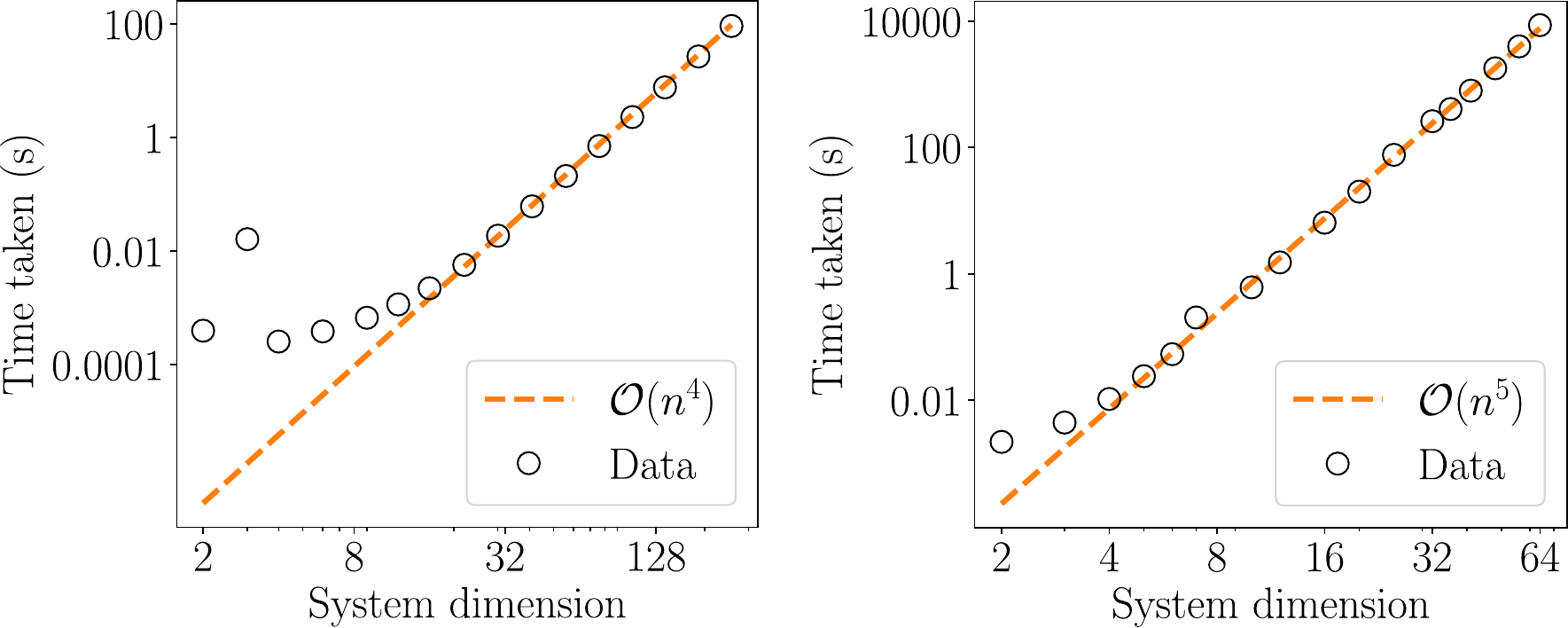}
    \caption{\textbf{Time complexity of the recursive algorithm.} The time complexity of finding the coefficient matrices of the numerator and the coefficients of the polynomials in the denominator using, \textbf{Left:} Fixed precision computations ($\mathcal{O}(n^4)$) and \textbf{Right:} Rational computations ($\mathcal{O}(n^5)$).}
    \label{fig:tc}
\end{figure*}

\newpage

\subsection{Auxiliary spectral matrix}
\begin{theorem} 
\label{theorem:0_g}
Let $\G$ be a real square matrix, $\Y$ be a positive definite matrix, and $\mathbfcal{T}(\omega)$ be defined as follows,
\begin{equation}
\begin{split}
    \mathbfcal{T}(\omega) &= (\omega\I-\G)^{-1} \, \Y \, (\omega\I-\G^\top)^{-1} \\
    &= \frac{\sum\limits_{\alpha=0}^{2n-2}\Pb_\alpha \omega^{\alpha} }{\sum\limits_{\alpha=0}^{2n}q_\alpha \omega^{\alpha}} \label{eq:rational_function_g}
\end{split}
\end{equation}
The numerator's matrix coefficients are given by the recursive equations,
\begin{equation}
\begin{split}
    \Pb_{\alpha-1} = q_{\alpha+1} \Y + \G \Pb_{\alpha} + \Pb_{\alpha} \G^\top -\G \Pb_{\alpha+1} \G^\top \\
\end{split}
\label{eq:num-rec-coefficients}
\end{equation}
for $\alpha \in \{1,\dots, 2n-1\}$, starting from $\alpha = 2n-1$ with $\Pb_{2n}=\Pb_{2n-1}=\mathbf{0}$. The scalar coefficients of the denominator can be recursively calculated using,
\begin{equation}
\begin{split}
    q_\alpha = \frac{2}{2n - \alpha}\left[\Tr\left(\Y^{-1} \G \Pb_\alpha \G^\top \right) - \Tr\left(\Y^{-1}\G \Pb_{\alpha-1} \right)\right]
\end{split}
\end{equation}
for $\alpha \in \{1,\dots, 2n-1\}$, and $q_{2n} = 1$.
\end{theorem}

\begin{proof}
We follow an approach similar to the one presented in Section~\ref{sec:recursive_algorithm}. We can write the spectrum as the following rational function,
\begin{equation}
\begin{split}
    \mathbfcal{T}(\omega) &= (\omega\I-\G)^{-1} \, \Y \, (\omega\I-\G^\top)^{-1} \\
    &= \frac{\mathrm{adj}(\omega\I-\G) \, \Y \, \mathrm{adj}(\omega\I-\G^\top)}{\mathrm{det}(\omega\I-\G)\, \mathrm{det}(\omega\I-\G^\top)} \\
    &= \frac{\sum\limits_{\alpha=0}^{2n-2}\Pb_\alpha \omega^{\alpha} }{\sum\limits_{\alpha=0}^{2n}q_\alpha \omega^{\alpha}} \label{eq:rational_function_g}
\end{split}
\end{equation}
Upon rearrangement and left multiplying both sides with $(\omega\I-\G)$ and right multiplying with $(\omega\I-\G^\top)$, we get,
\begin{equation}
     {\sum\limits_{\alpha=0}^{2n}q_\alpha \omega^{\alpha}} \Y = \left(\omega\I-\G\right) \left(\sum\limits_{\alpha=0}^{2n-2}\Pb_\alpha \omega^{\alpha}\right) \left(\omega\I-\G^\top\right)
\end{equation}
Expanding the series on the left and right gives us,
\begin{equation}
\begin{split}
     \bigl(q_0 + q_1 \omega + q_2 \omega^2 + ...+q_{2n}\omega^{2n}\bigr) \mathbf{Y}  =& \left(\omega\I-\G\right) \biggl(\Pb_0 + \Pb_1 \omega +...+ \Pb_{2n-2}\omega^{2n-2} \biggr) \left(\omega\I-\G^\top\right) \\
     =& \G \Pb_0 \G^\top + \G \Pb_1 \G^\top \omega + ...+ \G \Pb_{2n-2} \G^\top \omega^{2n-2} \\
     &-\left(\Pb_0 \G^\top \omega + \Pb_1 \G^\top \omega^2 +...+ \Pb_{2n-2}\G^\top \omega^{2n-1} \right) \\
     &-\left(\G\Pb_0 \omega + \G\Pb_1 \omega^2 + ... + \G\Pb_{2n-2} \omega^{2n-1}\right) \\
     &+ \Pb_0 \omega^2 + \Pb_1 \omega^3 + ... + \Pb_{2n-2}\omega^{2n}
\end{split}
\end{equation}
Comparing the coefficients of $\omega^{\alpha+1}$ gives us the following recursive solution for $\Pb_\alpha$ (the matrix with the coefficients of the numerator),
\begin{equation}
    \Pb_{\alpha-1} = q_{\alpha+1} \Y + \G \Pb_{\alpha} + \Pb_{\alpha} \G^\top -\G \Pb_{\alpha+1} \G^\top \label{eq:recurrence_g}
\end{equation}
To find a recursive solution for the coefficients of the denominator, we consider the Laplace transform of the following matrix and use the property of the resolvent of a matrix (Eq.~\ref{eq:resolvent_appendix}) to find,
\begin{equation}
\begin{split}
    \mathcal{L}\left\{\frac{\mathrm{d} }{\mathrm{d} t} e^{t\G}\right\} &= \mathcal{L}\left\{\G e^{t\G}\right\} \\
    &= \G (\omega\I - \G)^{-1} 
\end{split}
\end{equation}
Then, using the property of the Laplace transform of the derivative of a function, we can write the $\mathrm{LHS}$ of the equation above as,
\begin{equation}
    \omega \mathcal{L}\left\{e^{ t\G}\right\} - \I = \G (\omega\I - \G)^{-1}
\end{equation}
Now, consider the following matrix product,
\begin{equation}
     \mathrm{LHS} \coloneqq \left(\omega \mathcal{L}\left\{e^{ t\G}\right\} - \I\right) \Y \left(\omega \mathcal{L}\left\{e^{ t\G}\right\} - \I\right)^\top = \G\left[(\omega\I-\G)^{-1} \, \Y \, (\omega\I-\G^\top)^{-1}\right]\G^\top \coloneqq \mathrm{RHS}
\end{equation}
Further simplification of the $\mathrm{LHS}$ gives us,
\begin{equation}
\begin{split}
    \mathrm{LHS} &= \left(\omega \mathcal{L}\left\{e^{ t\G}\right\} - \I\right) \Y \left(\omega \mathcal{L}\left\{e^{ t\G}\right\} - \I\right)^\top \\
    &= \omega^2 \mathcal{L}\left\{e^{t\G}\right\} \, \Y \, \mathcal{L}\left\{e^{t\G}\right\}^\top - \omega \mathcal{L}\left\{e^{ t\G}\right\} \Y - \omega \, \Y\, \mathcal{L}\left\{e^{ t\G}\right\}^\top \, +\Y\\
    &= \omega^2 (\omega\I - \G)^{-1}\, \Y\, (\omega\I -  \G)^{-\top} - \omega \mathcal{L}\left\{e^{ t\G}\right\} \Y - \omega \, \Y\, \mathcal{L}\left\{e^{t\G}\right\}^\top \, +\Y\\
    &= \omega^2\left[ ( \omega\I - \G)^{-1}\, \Y\, (\omega\I - \G^\top)^{-1}\right] - \omega \mathcal{L}\left\{e^{ t\G}\right\} \Y - \omega \, \Y\, \mathcal{L}\left\{e^{ t\G}\right\}^\top \, +\Y
\end{split}
\end{equation}
Since, from Eq.~\ref{eq:rational_function_g} we know the rational function form of $(\omega\I-\G)^{-1} \, \Y \, (\omega\I-\G^\top)^{-1}$, we substitute it in the $\mathrm{LHS}$ and the $\mathrm{RHS}$. We further assume that $\Y$ is positive definite, therefore, its inverse exists and multiply $\mathrm{LHS}$ and $\mathrm{RHS}$ with $\Y^{-1}$ on the left, which gives us the following equation,
\begin{equation}
\begin{split}
    \omega^2\, \Y^{-1}\left[\frac{\sum\limits_{\alpha=0}^{2n-2}\Pb_\alpha \omega^{\alpha} }{\sum\limits_{\alpha=0}^{2n}q_\alpha \omega^{\alpha}}\right] - \omega \left( \Y^{-1} \mathcal{L}\left\{e^{ t\G}\right\} \Y + \mathcal{L}\left\{e^{ t\G}\right\}^\top \right) + \I = \Y^{-1} \G\left[\frac{\sum\limits_{\alpha=0}^{2n-2}\Pb_\alpha \omega^{\alpha} }{\sum\limits_{\alpha=0}^{2n}q_\alpha \omega^{\alpha}}\right]\G^\top
\end{split}
\end{equation}
Multiplying both sides with ${\sum\limits_{\alpha=0}^{2n}q_\alpha \omega^{\alpha}}$ and taking the trace of both sides yields,
\begin{equation}
    \omega^2\Tr\left(\Y^{-1} \sum\limits_{\alpha=0}^{2n-2}\Pb_\alpha \omega^{\alpha} \right) - \omega {\sum\limits_{\alpha=0}^{2n}q_\alpha \omega^{\alpha}} \Tr\left( \Y^{-1} \mathcal{L}\left\{e^{ t\G}\right\} \Y + \mathcal{L}\left\{e^{ t\G}\right\}^\top \right) + n {\sum\limits_{\alpha=0}^{2n}q_\alpha \omega^{\alpha}} = \Tr\left(\Y^{-1} \G\left(\sum\limits_{\alpha=0}^{2n-2}\Pb_\alpha \omega^{\alpha} \right)\G^\top \right)
\end{equation}
Since $\Tr\left(\Y^{-1}  \mathcal{L}\left\{e^{ t\G}\right\}  \Y \right) = \Tr\left( \mathcal{L}\left\{e^{ t\G}\right\} \right)$ and from Lemma~\ref{lemma:1_g} we have that $\Tr\left(\mathcal{L}\left\{e^{ t \G} \right\} + \mathcal{L}\left\{e^{ t \G}\right\}^\top \right) = Q^\prime(\omega)/Q(\omega)$, where $Q(\omega) = {\sum\limits_{\alpha=0}^{2n}q_\alpha \omega^{\alpha}}$ and $Q^\prime(\omega) = {\sum\limits_{\alpha=1}^{2n}\alpha q_\alpha \omega^{\alpha-1}}$, the equation above simplifies to the following equation,
\begin{equation}
    \Tr\left(\Y^{-1} \sum\limits_{\alpha=0}^{2n-2}\Pb_\alpha \omega^{\alpha+2} \right) + {\sum\limits_{\alpha=0}^{2n} \left(n-\alpha\right) q_\alpha \omega^{\alpha}}  = \Tr\left(\Y^{-1} \G\left(\sum\limits_{\alpha=0}^{2n-2}\Pb_\alpha \omega^{\alpha} \right)\G^\top \right)
\end{equation}
Now we compare the coefficients of $\omega^{\alpha}$ and find the following recurrent relation,
\begin{equation}
    \Tr\left(\Y^{-1} \Pb_{\alpha-2}  \right) + \left(n-\alpha\right) q_\alpha = \Tr\left(\Y^{-1} \G \Pb_\alpha \G^\top \right)
\end{equation}
Substituting $\Pb_{\alpha-2} \rightarrow q_{\alpha} \Y + \G \Pb_{\alpha-1} + \Pb_{\alpha-1} \G^\top -\G \Pb_{\alpha} \G^\top$ using Eq.~\ref{eq:recurrence_g}, we get the following equation satisfied by $q_\alpha$,
\begin{equation}
\begin{split}
    \Tr\left(\Y^{-1} \left(q_{\alpha} \Y + \G \Pb_{\alpha-1} + \Pb_{\alpha-1} \G^\top -\G \Pb_{\alpha} \G^\top\right) \right) + \left(n-\alpha\right) q_\alpha &= \Tr\left(\Y^{-1} \G \Pb_\alpha \G^\top \right) \\
    n q_\alpha + 2 \Tr\left(\Y^{-1}\G \Pb_{\alpha-1} \right) - \Tr\left(\Y^{-1} \G \Pb_{\alpha} \G^\top \right)  + \left(n-\alpha\right) q_\alpha &= \Tr\left(\Y^{-1} \G \Pb_\alpha \G^\top \right) \\
\end{split}
\end{equation}
We have used the fact that $\Tr\left(\Y^{-1}\G \Pb_{\alpha-1} \right) = \Tr\left(\Pb_{\alpha-1} \G^\top \Y^{-1}\right) = \Tr\left(\Y^{-1} \Pb_{\alpha-1} \G^\top \right)$ which is true because $\Pb_{\alpha-1}$ and $\Y^{-1}$ are symmetric matrices. Solving for $q_\alpha$ gives us,
\begin{equation}
    q_\alpha = \frac{2}{2n - \alpha}\left[\Tr\left(\Y^{-1} \G \Pb_\alpha \G^\top \right) - \Tr\left(\Y^{-1}\G \Pb_{\alpha-1} \right)\right]
\end{equation}
\end{proof}

\begin{lemma}
\label{lemma:1_g}
Let $\Tr(\mathbf{G})$ denote the trace of the matrix $\mathbf{G}$, $\mathcal{L}\{f(t)\}$ denote the Laplace transform of $f(t)$ with the complex variable $\omega$ and $Q^\prime(\omega)$ be the derivative of the denominator of the rational function in Eq.~\ref{eq:rational_function_g}, $Q(\omega) = {\sum\limits_{\alpha=0}^{2n}q_\alpha \omega^{\alpha}}$ with respect to $\omega$, then we have,
\begin{equation}
    \Tr\left(\mathcal{L}\left\{e^{ t \G} \right\} + \mathcal{L}\left\{e^{ t \G}\right\}^\top \right) = \frac{Q^\prime(\omega)}{Q(\omega)}.
\end{equation}
\end{lemma}
\begin{proof}
Consider the $\mathrm{LHS}$,
\begin{equation}
\begin{split}
    \Tr\left(\mathcal{L}\left\{e^{ t \G} \right\} + \mathcal{L}\left\{e^{ t \G}\right\}^\top \right)
    &= \mathcal{L}\left\{\Tr\left(e^{ t \G} + e^{ t \G^\top} \right)\right\} \\
    &= \mathcal{L}\left\{\Tr\left( \sum_{r=0}^\infty \frac{t^r (\G)^r}{r!} +  \sum_{r=0}^\infty \frac{t^r (\G^\top)^r}{r!}\right)\right\} \\
    &= \mathcal{L}\left\{ \sum_{r=0}^\infty \Tr\left(\frac{t^r (\G)^r}{r!} \right)+  \sum_{r=0}^\infty \Tr\left(\frac{t^r (\G^\top)^r}{r!} \right) \right\}
\end{split}
\end{equation}
If $\lambda_j$ are the eigenvalues of $\G$ and $\G^\top$, we have,
\begin{equation}
\begin{split}
    \Tr\left(\mathcal{L}\left\{e^{ t \G} \right\} + \mathcal{L}\left\{e^{ t \G}\right\}^\top \right)
    &= \mathcal{L}\left\{ \sum_{r=0}^\infty \sum_{j=1}^n \frac{t^r (\lambda_j)^r}{r!} +  \sum_{r=0}^\infty \sum_{j=1}^n \frac{t^r ( \lambda_j)^r}{r!} \right\} \\
    &= \mathcal{L}\left\{  2 \sum_{j=1}^n e^{t \lambda_j} \right\} \\
    &= 2 \sum_{j=1}^n \frac{1}{\omega -  \lambda_j} 
\end{split}
\end{equation}
From Eq.\ref{eq:imp21}, we have,
\begin{equation}
\begin{split}
    {Q(\omega)} &= {\mathrm{det}(\omega\I-\G)\, \mathrm{det}(\omega\I-\G^\top)} \\
    &=  \prod_{j=1}^n (\omega - \lambda_j)^2
\end{split}
\end{equation}
Taking log on both sides and the derivative with respect to $\omega$, we get,
\begin{equation}
\begin{split}
    \log{Q(\omega)} &= 2 \sum_{j=1}^n \log(\omega - \lambda_j) \\
    \frac{Q^\prime(\omega)}{Q(\omega)} &= 2 \sum_{j=1}^n \frac{1}{\omega -  \lambda_j} 
\end{split}
\end{equation}
Therefore,
\begin{equation}
    \Tr\left(\mathcal{L}\left\{e^{ t \G} \right\} + \mathcal{L}\left\{e^{ t \G}\right\}^\top \right) = \frac{Q^\prime(\omega)}{Q(\omega)}.
\end{equation}

\end{proof}

\subsection{Integrated covariance matrix}
\begin{theorem}
\label{theorem:without_omega}
Let $\mathbf{G}$ be a convergent matrix, i.e., with a spectral radius less than 1, $\mathbf{Y}$ be a positive definite matrix, and $\mathbf{R}$ be a matrix with the following series expansion:
\begin{equation}
\begin{split}
    \mathbf{R} &= \left(\I - \mathbf{G} \right)^{-1} \mathbf{Y} \left(\I - \mathbf{G}^\top\right)^{-1} \\
    &= \sum_{m = 0}^\infty \sum_{\substack{j+k = m \\ j\geq0, \, k\geq 0}}\G^j \, \Y \, \left(\G^\top\right)^{k} \\
    &= \sum_{m = 0}^\infty \Ss_m 
\end{split}
\end{equation}
where $\Ss_m \coloneqq \sum_{\substack{j+k = m \\ j\geq0, \, k\geq 0}}\G^j \, \Y \, \left(\G^\top\right)^{k}$. The following algorithm provides a recursive method to compute $\Ss_m$,
\begin{equation}
\begin{split}
    t_{m} &= \frac{2}{m}\left[\Tr\left(\Y^{-1} \G \Ub_{m-2} \G^\top \right) - \Tr\left(\Y^{-1}\G \Ub_{m-1}  \right)\right] \\
    \Ub_{m} &= t_{m} \Y + \G \Ub_{m-1} + \Ub_{m-1} \G^\top -\G \Ub_{m-2} \G^\top \\
    \Ss_{m} &= \Ub_{m} - \sum_{k=0}^{m-1} t_{m-k} \, \Ss_k     
\end{split}
\end{equation}
starting with $m=0$ until $m=2n$ and the initial values $t_0=1$ and $\Ub_{-2} = \Ub_{-1} = \mathbf{0}$.
\end{theorem}

\begin{proof}
Consider a matrix defined as follows,
\begin{equation}
\begin{split}
    \mathbfcal{T}(\omega) &= (\omega\I-\G)^{-1} \, \Y \, (\omega\I-\G^\top)^{-1}  \\
    &= \frac{1}{\omega^2}\left(\I - \frac{1}{\omega}\G \right)^{-1} \, \Y \, \left(\I-\frac{1}{\omega}\G^\top\right)^{-1} 
\end{split}
\end{equation}
Since there always exists a value of $\omega$ such that $(\I - \G/\omega)$ and $(\I - \G^\top/\omega)$ are convergent, we can write the inverses as a convergent power series,
\begin{equation}
\begin{split}
    \mathbfcal{T}(\omega) &= \frac{1}{\omega^2} \left(\sum_{j=0}^{\infty} \frac{1}{\omega^j} \G^j\right) \, \Y \, \left(\sum_{k=0}^{\infty} \frac{1}{\omega^k} \left(\G^\top\right)^{k}\right) \\
    &= \frac{1}{\omega^2} \sum_{j, k = 0}^\infty \frac{1}{\omega^{j+k}} \G^j \, \Y \, \left(\G^\top\right)^{k} \\
    &= \frac{1}{\omega^2} \sum_{m = 0}^\infty \frac{1}{\omega^{m}} \left( \sum_{\substack{j+k = m \\ j\geq0, \, k\geq 0}}\G^j \, \Y \, \left(\G^\top\right)^{k} \right) \\
    &= \frac{1}{\omega^2} \sum_{m = 0}^\infty \frac{1}{\omega^{m}} \Ss_m
\end{split}
\end{equation}
where $\Ss_m$ represents the sum of the contribution of the paths of order (length) $m$. We now find find $\Ss_m$ recursively. Using Theorem~\ref{theorem:0_g}, we get the following equation,
\begin{equation}
    \frac{1}{\omega^2} \sum_{m = 0}^\infty \frac{1}{\omega^{m}} \Ss_m = \frac{\sum\limits_{\alpha=0}^{2n-2}\Pb_\alpha \omega^{\alpha} }{\sum\limits_{\alpha=0}^{2n}q_\alpha \omega^{\alpha}}
\end{equation}
alternatively, we have,
\begin{equation}
    \left( \sum\limits_{\alpha=0}^{2n}q_\alpha \omega^{\alpha} \right) \left(\sum_{m = 0}^\infty \frac{1}{\omega^{m}} \Ss_m \right) = \sum\limits_{\alpha=0}^{2n-2}\Pb_\alpha \omega^{\alpha+2}
\end{equation}
Expanding the product gives us,
\begin{equation}
    \sum\limits_{\alpha=0}^{2n} \sum_{m = 0}^\infty q_\alpha \omega^{\alpha-m}  \Ss_m = \Pb_0 \omega^{2} + \Pb_1 \omega^{3} + ... + \Pb_{2n-2} \omega^{2n}
\end{equation}
Now we compare the coefficients of $\omega^{2n-j}$ on both sides. To find the coefficient of $\omega^{2n-j}$ on the $\mathrm{LHS}$ of the equation above, we first set $\alpha - m = 2n - j$, which gives us $\alpha = 2n - j + m$ and then we find the values of $\alpha$ for $0 \leq m \leq j$ to ensure that $0 \leq \alpha \leq 2n$. This gives us the following relation,
\begin{equation}
    \Ss_j = \Pb_{2n-2-j} - \sum_{k=0}^{j-1} q_{2n-j+k} \, \Ss_k
\end{equation}
Making the following change of index, $j \rightarrow 2n - 1 - \alpha$, gives us
\begin{equation}
    \Ss_{2n-1-\alpha} = \Pb_{\alpha - 1} - \sum_{k=0}^{2n-2-\alpha} q_{\alpha+k+1} \, \Ss_k
\end{equation}
with $\Pb_{\alpha}$ and $q_\alpha$ given by the recursive algorithm in Theorem~\ref{theorem:0_g}. Therefore, the algorithm to recursively find $\Ss_0, \Ss_1, ..., \Ss_{2n}$ is given by,
\begin{equation}
\begin{split}
    \Pb_{\alpha-1} &= q_{\alpha+1} \Y + \G \Pb_{\alpha} + \Pb_{\alpha} \G^\top -\G \Pb_{\alpha+1} \G^\top \\
    \Ss_{2n-1-\alpha} &= \Pb_{\alpha - 1} - \sum_{k=0}^{2n-2-\alpha} q_{\alpha+k+1} \, \Ss_k \\
    q_\alpha &= \frac{2}{2n - \alpha}\left[\Tr\left(\Y^{-1} \G \Pb_\alpha \G^\top \right) - \Tr\left(\Y^{-1}\G \Pb_{\alpha-1} \right)\right]
\end{split}
\end{equation}
starting with $\alpha = 2n-1$ until $\alpha=-1$. Note that $\Pb_{2n}=\Pb_{2n-1}=\Pb_{-1}=\Pb_{-2}=\mathbf{0}$, $q_{2n}=1$ and $q_{-1}=0$. Now, we first simplify this algorithm by first making the following substitution: $q_\alpha = t_{2n-\alpha}$ and $\Pb_\alpha = \Ub_{2n-2-\alpha}$. This gives us the following recurrence relation,
\begin{equation}
\begin{split}
    t_{2n-1-\alpha} &= \frac{2}{2n - 1 - \alpha}\left[\Tr\left(\Y^{-1} \G \Ub_{2n-3-\alpha} \G^\top \right) - \Tr\left(\Y^{-1}\G \Ub_{2n-2-\alpha}  \right)\right] \\
    \Ub_{2n-1-\alpha} &= t_{2n-1-\alpha} \Y + \G \Ub_{2n-2-\alpha} + \Ub_{2n-2-\alpha} \G^\top -\G \Ub_{2n-3-\alpha} \G^\top \\
    \Ss_{2n-1-\alpha} &= \Ub_{2n-1-\alpha} - \sum_{k=0}^{2n-2-\alpha} t_{2n-1-\alpha-k} \, \Ss_k     
\end{split}
\end{equation}
Finally, we make the following index change: $m = 2n-1-\alpha$ which gives us the recursive relation,
\begin{equation}
\begin{split}
    t_{m} &= \frac{2}{m}\left[\Tr\left(\Y^{-1} \G \Ub_{m-2} \G^\top \right) - \Tr\left(\Y^{-1}\G \Ub_{m-1}  \right)\right] \\
    \Ub_{m} &= t_{m} \Y + \G \Ub_{m-1} + \Ub_{m-1} \G^\top -\G \Ub_{m-2} \G^\top \\
    \Ss_{m} &= \Ub_{m} - \sum_{k=0}^{m-1} t_{m-k} \, \Ss_k     
\end{split}
\end{equation}
starting with $m=0$ until $m=2n$ and the following initial values given: $t_0=1$ and $\Ub_{-2} = \Ub_{-1} = \mathbf{0}$.

\end{proof}

\newpage

\section{Element-wise solution}
In this section, we describe an element-wise solution for the rational functions of the auto- and cross-spectrum for a general N-dimensional system. We first assume that the matrix $\mathbf{L}$ is square, i.e., $\mathbf{L} \in \R^{n\times n}$ and therefore, $\mathbf{D} \in \R^{n\times n}$. The result can easily be extended to a rectangular $\mathbf{L} \in \R^{n\times m}$ matrix by defining the positive semidefinite matrix $\mathbf{C}= \mathbf{L} \, \mathbf{D} \,\mathbf{L}^{\top}$ and then computing a new set of $\mathbf{L}$ and $\mathbf{D}$ matrices by LDL decomposition, such that $\mathbf{L}, \mathbf{D} \in \R^{n\times n}$.

We start by rewriting the matrix inverse in Eq.~\ref{eq:gen1} in terms of the corresponding adjoints and determinants of the matrices denoted by $\mathrm{adj}(.)$ and $|.|$, respectively.
From Eq.~\ref{eq:diff_mat}, we know that the diffusion matrix $\mathbf{D}$ can be written as,
\begin{equation}
    \mathbf{D} = \mathbf{S} \, \mathbf{S}^\top
\end{equation}
Substituting in Eq.~\ref{eq:gen1},
\begin{equation}
\begin{split}
    \mathbfcal{S}(\omega) &= \frac{\mathrm{adj}(\mathbf{J}+\dot{\iota} \omega\mathbf{I}) \, \mathbf{L} \, \mathbf{S} \, \mathbf{S}^\top \,\mathbf{L}^{\top} \, \mathrm{adj}^\top(\mathbf{J}-\dot{\iota} \omega\mathbf{I})}{|\mathbf{J}+\dot{\iota} \omega\mathbf{I}|\, |\mathbf{J}-\dot{\iota} \omega\mathbf{I}|} \\
    &= \frac{\mathrm{adj}(\mathbf{J}+\dot{\iota} \omega\mathbf{I}) \, \mathbf{L} \, \mathbf{S} 
    \, \left( \mathrm{adj}(\mathbf{J}+\dot{\iota} \omega\mathbf{I}) \, \mathbf{L} \,  \mathbf{S}  \right)^*}{|\mathbf{J}+\dot{\iota} \omega\mathbf{I}|\, |\mathbf{J}-\dot{\iota} \omega\mathbf{I}|} \\ 
    &= \frac{\mathbf{Z}(\omega)}{Q(\omega)} 
    \label{eq:gen21}
\end{split}
\end{equation}
Where $\mathbf{Z}(\omega)$ is a matrix containing the polynomial expansion for the numerators of both auto and cross spectrums, in terms of $\omega$; and $Q(\omega)$ is the polynomial expansion of the denominator. 

\subsection{Denominator and Newton's identities} \label{sec:denominator_and_newton}
We derive the coefficients of the polynomial expansion of the denominator, $Q(\omega)$. First, we will introduce elementary symmetric polynomials, a common tool used throughout this text, which shows how we can express various terms in the numerator and denominator as polynomials of $\omega$. To demonstrate the analysis, which uses elementary symmetric polynomials, we show how we can expand the denominator into a polynomial in $\omega$, and find its coefficients. First, we show a brute force method of calculating the said coefficients and then demonstrate that we get the same expressions using elementary symmetric polynomials, and finally use the function definitions from Sec.~\ref{sec:function_def} to write a compact expression for the polynomial. We know from Eq.~\ref{eq:gen21} that the denominator $Q(\omega)$ can be written as the product of the determinants,
\begin{equation}
    Q(\omega) = |\mathbf{J}+\dot{\iota} \omega\mathbf{I}| |\mathbf{J}-\dot{\iota} \omega\mathbf{I}| \label{eq:denm_identity}
\end{equation}
where $\mathbf{J}\in \R^{n \times n}$. If $\lambda_i$ are the eigenvalues of the matrix $\mathbf{J}$, using Eq.~\ref{eq:det_prod}, we have,
\begin{equation}
    |\mathbf{J} \pm \dot{\iota} \omega| = \prod_{i=1}^n (\lambda_i \pm \dot{\iota}\omega)
\end{equation}
Therefore we can express the denominator as an even-powered polynomial,
\begin{equation}
\begin{split}
    Q(\omega) &= \prod_{i=1}^n (\lambda_i + \dot{\iota}\omega)(\lambda_i - \dot{\iota}\omega) \\
    &= \prod_{i=1}^n (\lambda_i^2 + \omega^2) \\
    &= q_0 + q_1 \omega^2 + q_2 \omega^4 + \dots + q_n \omega^{2n}
\end{split}
\end{equation}
We also note that the property above is true even if $\lambda$ is complex. We show this by dividing the product into two parts: first when the eigenvalues are purely real (cardinality of $n-2m$) and second when they are complex ($m$ pairs of conjugate complex pairs). Since the matrix $\mathbf{J}$ is real, its complex eigenvalues occur as complex conjugate pairs. Therefore, we can write, 
\begin{equation}
    Q(\omega) = \prod_{i=1}^{n-2m} (\lambda_i + \dot{\iota}\omega)(\lambda_i - \dot{\iota}\omega)  \prod_{j=1}^{m} (\lambda_j + \dot{\iota}\omega) (\lambda_j - \dot{\iota}\omega) (\overline{\lambda}_j + \dot{\iota}\omega) (\overline{\lambda}_j - \dot{\iota}\omega)
\end{equation}
On further simplification, we can write the equation as,
\begin{equation}
\begin{split}
    Q(\omega) & = \prod_{i=1}^{n-2m} (\lambda_i^2 + \omega^2) \prod_{j=1}^{m} (\lambda_j^2 + \omega^2) (\overline{\lambda}_j^2 + \omega^2) \\
    & = \prod_{i=1}^{n} (\lambda_i^2 + \omega^2) \label{eq:demn_wr1}
\end{split}
\end{equation}
Now we demonstrate how we can calculate the coefficients of the polynomials using brute force for a 4-dimensional system. The expansion of Eq.~\ref{eq:demn_wr1} can be written as,
\begin{equation}
    Q(\omega) = q_0 + q_1 \omega^2 + q_2 \omega^4 + q_3 \omega^6 + q^4 \omega^8 \label{eq:denm_expansion}
\end{equation}
We find the different coefficients $q_k$ using the properties of matrices, specifically we use the properties from Eq.~\ref{eq:det_prod} and Eq.~\ref{eq:trace_sum} to find,
\begin{equation}
    q_0 = \lambda_1^2 \lambda_2^2 \lambda_3^2 \lambda_4^2 = \prod_{i=1}^n \lambda_i = \mathrm{det}^2(\mathbf{J}) \label{eq:q_0}
\end{equation}
\begin{equation}
\begin{split}
    q_1 &= \lambda_1^2 \lambda_2^2 \lambda_3^2 + \lambda_1^2 \lambda_2^2 \lambda_4^2 +\lambda_1^2 \lambda_3^2 \lambda_4^2 +\lambda_2^2 \lambda_3^2 \lambda_4^2 \\
    &= \frac{(\sum_{i=1}^n \lambda_i^2)^3 + 2\sum_{i=1}^n \lambda_i^6 - 3\sum_{i=1}^n \lambda_i^4 \sum_{i=1}^n \lambda_i^2}{6}\\
    &= \frac{(\Tr( \mathbf{J}^2))^3 + 2 \Tr( \mathbf{J}^6) - 3 \Tr( \mathbf{J}^4) \Tr( \mathbf{J}^2)}{6}
\end{split}    
\end{equation}
\begin{equation}
    q_2 = \lambda_1^2 \lambda_2^2 + \lambda_1^2 \lambda_3^2 +\lambda_1^2 \lambda_4^2 +\lambda_2^2 \lambda_3^2 +\lambda_2^2 \lambda_4^2 +\lambda_3^2 \lambda_4^2 
    = \frac{(\sum_{i=1}^n \lambda_i^2)^2 - \sum_{i=1}^n \lambda_i^4}{2}
    = \frac{(\Tr( \mathbf{J}^2))^2 - \Tr( \mathbf{J}^4) }{2}
\end{equation}
\begin{equation}
    q_3 = \lambda_1^2 + \lambda_2^2 + \lambda_3^2 + \lambda_4^2   
    = \sum_{i=1}^n \lambda_i^2  
    = \Tr( \mathbf{J}^2) 
\end{equation}
\begin{equation}
    q_4 = 1 \label{eq:q_4}
\end{equation}
Therefore for a 4-D system, we can write the denominator as,
\begin{equation}
  Q(\omega) =  \mathrm{det}^2(\mathbf{J}) + \frac{\Tr( \mathbf{J}^2)^3 + 2 \Tr( \mathbf{J}^6) - 3 \Tr( \mathbf{J}^4) \Tr( \mathbf{J}^2)}{6} \omega^2  + \frac{\Tr( \mathbf{J}^2)^2 - \Tr( \mathbf{J}^4) }{2} \omega^4 + \Tr( \mathbf{J}^2) \omega^6 + \omega^8
\end{equation}

Alternatively, we can find the general solution for any dimensional system using Newton's identities and thereby write the coefficients of different powers of $\omega$ in the denominator. To use this method, we first need to introduce Newton's identities. We give a brief primer on Newton's identities and show how we can express the various orders of elementary symmetric polynomials in terms of Bell polynomials, then continue with our analysis. Let $x_1,x_2,...,x_n$ be the variables, we define $r^k(x_1,x_2,...,x_n)$ to be the $k^{th}$ power sum (for $k\geq1$),
\begin{equation}
    r^k(x_1,x_2,...,x_n) = \sum\nolimits_{i=1}^{n} x_i^k
\end{equation}
and the elementary symmetric polynomial $e^l(x_1,x_2,...,x_n)$ (for $l\geq0$),
\begin{equation}
\begin{split}
    e^0(x_1,x_2,...,x_n) &= 1 \\
    e^1(x_1,x_2,...,x_n) &= x_1 + x_2 +...+x_n \\
    e^2(x_1,x_2,...,x_n) &= \sum\nolimits_{1<i<j<n} x_i x_j\\
    &.\\
    &.\\
    &.\\
    e^k(x_1,x_2,...,x_n) &= \sum\nolimits_{1<i<j<...<t<n}\underbrace{x_i x_j ... x_t}_{\text{k vars}}\\
    e^n(x_1,x_2,...,x_n) &= x_1 x_2 ...x_n
\end{split}
\end{equation}
that obey the so-called Newton's identities, expressed concisely as, 
\begin{equation}
    ke^k(x_1,x_2,...,x_n) = \sum_{i=1}^{k} (-1)^{i-1} e^{k-i}(x_1,x_2,...,x_n) r^i(x_1,x_2,...,x_n) \label{eq:newton1}
\end{equation}

We note that we can write the coefficients of the polynomial in $\omega$, $q_k$, in terms of elementary symmetric polynomials of various orders,  $e^i(\lambda_1^2, \lambda_2^2,...,\lambda_n^2)$,
\begin{equation}
  Q(\omega) = \prod_{i=1}^{n} (\lambda_i^2 + \omega^2) 
    = \sum_{i=0}^{n} e^i(\lambda_1^2, \lambda_2^2,...,\lambda_n^2)\omega^{2(n-i)} \label{eq:denm_lambda}
\end{equation}
Noting that $ e^0(\lambda_1^2, \lambda_2^2,...,\lambda_n^2) = q_4 = 1$ by definition, and $ e^1(\lambda_1^2, \lambda_2^2,...,\lambda_n^2) = q_3 = \sum_{i=1}^n \lambda_i^2 = \Tr( \mathbf{J}^2 )$  and that $r^k$ can be written as follows,
\begin{equation}
    r^k(\lambda_1^2, \lambda_2^2,...,\lambda_n^2) = \sum_{j=1}^{n} \lambda_j^{2k} \label{eq:r_k}
     = \Tr( \mathbf{J}^{2k} )
\end{equation}
we can find $e^i(\lambda_1^2, \lambda_2^2,...,\lambda_n^2)$ recursively for any dimension using Eq.~\ref{eq:newton1}.
\begin{equation}
\begin{aligned}
&\begin{aligned}
    e^2(\lambda_1^2, \lambda_2^2,...,\lambda_n^2) &= q_2 \\
    &= \frac{e^1r^1 - e^0r^2}{2}\\
    &= \frac{\Tr( \mathbf{J}^2 ) \Tr( \mathbf{J}^{2} )-  \Tr( \mathbf{J}^{4} )}{2}\\
    &= \frac{\Tr( \mathbf{J}^2 )^2 - \Tr{ (\mathbf{J}^{4} )}}{2}
\end{aligned}\\
&\begin{aligned}
    e^3(\lambda_1^2, \lambda_2^2,...,\lambda_n^2) &= q_1 \\
    &= \frac{e^2r^1 - e^1r^2 + e^0r^3}{3}\\
    &= \frac{(\Tr( \mathbf{J}^2 )^2 - \Tr( \mathbf{J}^{4} )) \Tr( \mathbf{J}^{2} )/2 -  \Tr( \mathbf{J}^{2} ) \Tr( \mathbf{J}^{4} )   + \Tr( \mathbf{J}^{6} )}{3}\\
    &= \frac{\Tr( \mathbf{J}^2)^3 + 2 \Tr( \mathbf{J}^6) - 3 \Tr( \mathbf{J}^4) \Tr( \mathbf{J}^2)}{6}
\end{aligned}\\
&\begin{aligned}
    e^4(\lambda_1^2, \lambda_2^2,...,\lambda_n^2) &= q_0 \\
    &=  \prod_{i=1}^n \lambda_i^2\\
    &= \mathrm{det}^2(\mathbf{J})
\end{aligned}
\end{aligned}
\end{equation}
We thus confirm that we find the same expressions for the coefficients (Eq.~\ref{eq:q_0}-\ref{eq:q_4}) using the brute force method and using Newton's identities.

We can also express the elementary symmetric polynomial, $e^k$, in terms of Bell polynomial as follows,
\begin{equation}
    e^k = \frac{(-1)^k}{k!} B(-r^1, -1! r^2, ..., -(k-1)! r^k) \label{eq:bell10}
\end{equation}
where $B$ is the complete exponential Bell polynomial and can be calculated using the determinant of the upper Hessenberg matrix, $\mathbb{B}^k$, with determinant evaluation in $\mathbb{O}(n^2)$ \cite{cahill2002fibonacci}. In general, for a Bell polynomial defined on the set of inputs $(x_1,x_2,...,x_k)$, we have,
\begin{equation}
    B(x_1,x_2,...,x_k) = \mathrm{det}(\mathbb{B}^k(x_1,x_2,...,x_k))  \label{eq:bell-1}
\end{equation}
where the matrix $\mathbb{B}^k(x_1,x_2,...,x_k)$ is,
\begin{equation}
    \mathbb{B}^k(x_1,x_2,...,x_k) = \begin{bmatrix} 
    \frac{x_1}{0!} & \frac{x_2}{1!} & \frac{x_3}{2!} & \dots & \frac{x_{k-1}}{(k-2)!} & \frac{x_k}{(k-1)!} \\
    -1 & \frac{x_1}{0!} & \frac{x_2}{1!} & \dots & \frac{x_{k-2}}{(k-3)!} & \frac{x_{k-1}}{(k-2)!}\\
    0 & -2 & \frac{x_1}{0!} & \dots & \frac{x_{k-3}}{(k-4)!} & \frac{x_{k-2}}{(k-3)!} \\
    \vdots &  & \ddots & & \vdots & \vdots \\
    0 & 0 & 0 & \dots & -(k-1) & \frac{x_1}{0!}
    \end{bmatrix}.
\end{equation}
Now, for the input $\mathbf{r}^k = (-r^1, -1! r^2, ..., -(k-1)! r^k)$ to the Bell polynomial, we have the simplified $\mathbb{B}^k(\mathbf{r}^k)$ matrix,
\begin{equation}
    \mathbb{B}^k(\mathbf{r}^k) = \begin{bmatrix} 
    -r^1 & -r^2 & -r^3 & \dots & -r^{k-1} & -r^k \\
    -1 & -r^1 & -r^2 & \dots & -r^{k-2} & -r^{k-1}\\
    0 & -2 & -r^1 & \dots & -r^{k-3} & -r^{k-2} \\
    \vdots &  & \ddots & & \vdots & \vdots \\
    0 & 0 & 0 & \dots & -(k-1) & -r^1
    \end{bmatrix}
\end{equation}
such that,
\begin{equation}
    B(\mathbf{r}^k) = \mathrm{det}(\mathbb{B}^k(\mathbf{r}^k))  \label{eq:bell}
\end{equation}
Recall from Eq.~\ref{eq:r_k} that $r^k(\lambda_1,\lambda_2,\dots,\lambda_n)=\Tr(\mathbf{X}^k)$. Therefore, we define,
\begin{equation}
\begin{aligned}
  \mathbf{r}^k (\mathbf{X}) = (-\Tr( \mathbf{X}), -1!\Tr( \mathbf{X}^2), -2!\Tr( \mathbf{X}^3), ..., -(k-1)! \Tr( \mathbf{X}^k)) \label{eq:r_j} \\ 
\end{aligned}
\end{equation}
and,
\begin{equation}
\begin{aligned}
  e^k = \frac{(-1)^k}{k!} B(\mathbf{r}^k (\mathbf{X})) \\ 
\end{aligned}
\end{equation}
where, 
\begin{equation}
\begin{aligned}
   B(\mathbf{r}^k (\mathbf{X})) = \mathrm{det}
   \begin{bmatrix} 
    -\Tr( \mathbf{X}) & -\Tr( \mathbf{X}^2) & -\Tr( \mathbf{X}^3) & \dots & -\Tr( \mathbf{X}^{k-1}) & -\Tr( \mathbf{X}^k) \\
    -1 & -\Tr( \mathbf{X}) & -\Tr( \mathbf{X}^2) & \dots & -\Tr( \mathbf{X}^{k-2}) & -\Tr( \mathbf{X}^{k-1})\\
    0 & -2 & -\Tr( \mathbf{X}) & \dots & -\Tr( \mathbf{X}^{k-3}) & -\Tr( \mathbf{X}^{k-2}) \\
    \vdots &  & \ddots & & \vdots & \vdots \\
    0 & 0 & 0 & \dots & -(k-1) & -\Tr( \mathbf{X})
    \end{bmatrix} \label{eq:hessen_mat}
\end{aligned}
\end{equation}

We can then express $q_\alpha$, the coefficient of $\omega^{2\alpha}$ in the expansion of Eq.~\ref{eq:denm_expansion}, as
\begin{equation}
    q_\alpha =  \frac{(-1)^{n-\alpha}}{(n-\alpha)!} B(\mathbf{r}^{n-\alpha}(\mathbf{J}^2))
\end{equation}
This equation forms the basis of our ``canonical functions'' definitions in Section~\ref{sec:function_def}. Thus, we express the expansion of the polynomial in the denominator as,
\begin{equation}
\begin{aligned}
    Q(\omega) &= q_0 + q_1 \omega^2 + q_2 \omega^4 + \dots + q_n \omega^{2n} \\
    &= d(\omega; \mathbf{J}) \label{eq:denm_expansion11}
\end{aligned}
\end{equation}
and $q_\alpha$ can be defined in terms of the canonical function $d^\alpha(\mathbf{J})$ as, 
\begin{equation}
    q_\alpha = d^\alpha(\mathbf{J}). \label{eq:denm_coeff_soln}
\end{equation}

\subsection{Numerator}
Now, we move on to finding the rational function solution for the numerator. This corresponds to finding the polynomial expansion (in $\omega$) of the diagonal (off-diagonal) elements of the numerator for auto(cross)-spectrum $\mathbf{Z}(\omega)$, of $\mathbfcal{S}(\omega)$. Following Eq.~\ref{eq:gen21}, we have,
\begin{equation}
\begin{split}
    \mathbf{Z}(\omega) =& \mathrm{adj}(\mathbf{J}+\dot{\iota} \omega\mathbf{I}) \, \mathbf{L} \, \mathbf{S} 
    \, \left( \mathrm{adj}(\mathbf{J}+\dot{\iota} \omega\mathbf{I}) \, \mathbf{L} \,  \mathbf{S}  \right)^* \\
    =& \Pb(\omega) + \dot{\iota}\omega \Pb^\prime(\omega) \label{eq:app-z}
\end{split}
\end{equation}
If $a_{ij}$ are the $(i,j)$ elements of the matrix $\mathbf{J}$, the matrix $\mathbf{J}+\dot{\iota} \omega\mathbf{I}$ is given by,
\begin{equation}
    \mathbf{J}+\dot{\iota} \omega\mathbf{I}=
    \begin{bmatrix} 
    a_{11}+\dot{\iota} \omega &  \dots & a_{1n} \\
    \vdots & \ddots & \vdots\\
    a_{n1} &    \dots    & a_{nn}+\dot{\iota} \omega
    \end{bmatrix}
\end{equation}
% which in 4D is given by,
% \begin{equation}
%     \mathbf{J}+\dot{\iota} \omega\mathbf{I}=
%     \begin{bmatrix}
%         a+\dot{\iota} \omega & b & c & d\\
%         e & f+\dot{\iota} \omega & g & h\\
%         i & j & k+\dot{\iota} \omega & l \\
%         m & n & o & p+\dot{\iota} \omega
%     \end{bmatrix}
% \end{equation}
Put simply, the numerator is obtained by multiplying the matrix $\mathrm{adj}(\mathbf{J}+\dot{\iota} \omega\mathbf{I}) \, \mathbf{L} \, \mathbf{S} $  with its conjugate transpose. Note that from Eq.~\ref{eq:app-z}, we have that the elements of $\mathbf{Z}(\omega)$ take the form
\begin{equation}
\begin{split}
    Z^{ij}(\omega) &= P^{ij}(\omega) + \dot{\iota} \omega P^{ij\prime}(\omega)\\ 
    &= \sum_{\alpha=0}^{n-1}p^{ij}_{\alpha}\omega^{2\alpha} + \dot{\iota}\omega\sum_{\alpha=0}^{n-2}p^{ij\prime}_{\alpha}\omega^{2\alpha}
\end{split}
\end{equation}
We can write the matrix  $\mathrm{adj}(\mathbf{J}+\dot{\iota} \omega\mathbf{I}) \, \mathbf{L} \, \mathbf{S} $ in terms of the corresponding cofactor matrix as,
\begin{equation}
    \mathrm{adj}(\mathbf{J}+\dot{\iota} \omega\mathbf{I}) \, \mathbf{L} \, \mathbf{S} = \begin{bmatrix} 
    M_{11} &  \dots & (-1)^{1+n}M_{1n} \\
    -M_{21} & . &  (-1)^{2+n}M_{2n} \\
    \vdots &   \ddots & \vdots\\
    (-1)^{n+1}M_{n1} &    \dots & (-1)^{n+n}M_{nn}
    \end{bmatrix}^\top 
    \begin{bmatrix}
        \sigma_1 l_{11} & \sigma_2 l_{12} & \dots & \sigma_n l_{1n}\\
        \sigma_1 l_{21} & \sigma_2 l_{22} & \dots & \sigma_n l_{2n}\\
        \vdots &        & \ddots & \vdots \\
        \sigma_1 l_{n1} & \sigma_2 l_{n2} & \dots & \sigma_n l_{nn}
    \end{bmatrix} \label{eq:mat_prod}
\end{equation}
Here the minor $M_{ij}$ is the determinant of the submatrix $\mathbf{M}_{ij} \in \C^{n-1 \times n-1}$, formed by excluding the $i^{th}$ row and $j^{th}$ column of the matrix $\mathbf{J}+\dot{\iota}\omega\mathbf{I}$. Now, the $(a,b)$ element of the resulting matrix product can be written as $\sum_{o=1}^{n} (-1)^{a+o} M_{oa}(\sigma_b l_{ob})$. Similarly, the $(x,y)$ element of the matrix $\overline{\mathrm{adj}(\mathbf{J}+\dot{\iota} \omega\mathbf{I}) \, \mathbf{L} \, \mathbf{S}}$ can be written as $\sum_{p=1}^{n} (-1)^{y+p} \overline{M}_{py}(\sigma_x l_{px})$. To extract the $(i,j)$ element of $\mathbf{Z}(\omega)$, we expand the expression above to find,
\begin{equation}
\begin{split}
    Z^{ij}(\omega) =& \sum_{m=1}^n \biggl( \sum_{o=1}^n (-1)^{i+o} M_{oi} \sigma_m l_{om} \biggr) \biggl( \sum_{p=1}^n (-1)^{j+p} \overline{M}_{pj} \sigma_m l_{pm} \biggr) \\
    =& \sum_{m=1}^n \sigma_m^2 \sum_{o=1}^n \sum_{p=1}^n (-1)^{i+j+o+p} l_{om} l_{pm}  M_{oi}  \overline{M}_{pj}
    \label{eq:minor1}
\end{split}
\end{equation}
In the following sections, we consider $Z^{ij}(\omega)$ in 2 distinct settings: $i=j$ and $i\neq j$ to find the real and complex rational functions for auto-spectrum and cross-spectrum, respectively.

\subsubsection{Auto-spectrum}
We first find the analytical expressions for the numerator of the auto-spectrum, that correspond to the diagonal elements of $\mathbf{Z}(\omega)$. To facilitate our calculations, we split the expression into 4 parts,
\begin{equation}
    Z^{ii}(\omega) = \sum_{m=1}^n \sigma^2_m \biggl( l^2_{im} M_{ii} \overline{M}_{ii} + \sum_{\substack{j=1  \\j\neq i}}^n  l_{jm}^2  M_{ji} \overline{M}_{ji}   + 2 \sum_{\substack{j=1 \\ j\neq i}}^n (-1)^{i+j} l_{im}l_{jm} \mathrm{Re}[M_{ii}  \overline{M}_{ji} ] + 2 \sum_{\substack{j=1 \\ j\neq i}}^n \sum_{\substack{k=1 \\ k \neq i}}^{j-1}  (-1)^{j+k} l_{jm}l_{km} \mathrm{Re}[  M_{ji} \overline{M}_{ki} ] \biggr)  \label{eq:minor2}
\end{equation}
Note that the expression above is purely real, therefore, $\mathbf{Z}(\omega) = \mathbf{P}(\omega)$ and, $Z^{ii}(\omega) = P^{ii}(\omega)$. The minor matrices, $\mathbf{M}_{ij}$, contain some entries with $\dot{\iota}\omega$ added to them. Since exchanging the rows and columns of a matrix only changes the sign of the determinant, we systematically define the transformations on the $\mathbf{M}_{ij}$ such that the entries with $\dot{\iota}\omega$ appear on the diagonals only. But before defining the transformation, we define the matrices $\mathbf{N}_{ij} \in \R^{n-1 \times n-1}$ (with associated minors $N_{ij}\in \R$), which are formed simply by removing the terms $\dot{\iota} \omega$ from the corresponding  $\mathbf{M}_{ij}$ matrix. Then, by applying a transformation, summarized as an algorithm in Fig.~2 of the main text, to $\mathbf{N}_{ij}$, we get the matrices $\mathbf{O}_{ij} \in \R^{n-1 \times n-1}$, in terms of which the solution is defined. Therefore, the $\mathbf{O}_{ij}$ matrices can be calculated directly from the Jacobian matrix $\mathbf{J}$. We show how to construct the matrices for 2, 3, and 4D systems in Fig.~2 of the main text. 

Note that for $i=j$, the transformation is an identity and all the elements of $\mathbf{M}_{ij}$ containing $\iota\omega$ appear on the diagonal of the matrix. Therefore, we can write $M_{ii}=|\mathbf{O}_{ii}+\dot{\iota}\omega\mathbf{I}|$. When $i\neq j$, the exchange of rows and columns in this transformation yields the following relation between the determinants, 
\begin{equation}
O_{ij}=(-1)^{\gamma}N_{ij}
\end{equation}
where $\gamma=|i-j|-\delta_{ij}$ are the number of row/column exchanges performed in the transformation, and $\delta_{ij}$ is the Kronecker delta. Additionally, we introduce the standard unit matrix $\mathbf{E}_\beta = \mathbf{e}_\beta {\mathbf{e}_\beta}^\top$ associated with the matrix $\mathbf{O}_{ij}$, where $\mathbf{e}_\beta \in \{0,1\}^{n-1 \times 1}$ is the standard unit vector equal to 1 only at index $\beta$ and 0 everywhere else. The parameter $\beta$ (for $i\neq j$) is found to be $\beta=\min(i,j)$ , which allows us to express the minors as 
\begin{equation}
M_{ij}=(-1)^{|i-j|-\delta_{ij}}|\mathbf{O}_{ij}+\dot{\iota}\omega\mathbf{I}-\dot{\iota}\omega\mathbf{E}_\beta|   
\end{equation}
The purpose of this decomposition is that it allows us to write 
\begin{equation}
    |\mathbf{O}_{ij}+\dot{\iota}\omega\mathbf{I}-\dot{\iota}\omega\mathbf{E}_\beta|=|\mathbf{O}_{ij}+\dot{\iota}\omega\mathbf{I}| - \dot{\iota}\omega{\mathbf{e}_\beta}^\top \mathrm{adj}(\mathbf{O}_{ij}+\dot{\iota}\omega\mathbf{I}){\mathbf{e}_\beta}
\end{equation}
Since the non-zero element of $\mathbf{e}_\beta$ lies at the index $\beta$, we can simplify the determinant above by defining a submatrix ($\mathbf{O}_{ij}^\prime \in \R^{n-2\times n-2}$) of $\mathbf{O}_{ij}$ formed by removing the row and column at the index $\beta$. The equation then simplifies to,
\begin{equation}
    |\mathbf{O}_{ij}+\dot{\iota}\omega\mathbf{I}-\dot{\iota}\omega\mathbf{E}_\beta|=|\mathbf{O}_{ij}+\dot{\iota}\omega\mathbf{I}| - \dot{\iota}\omega|\mathbf{O}_{ij}^\prime+\dot{\iota}\omega\mathbf{I}|
\end{equation}
This property forms the basis for constructing the canonical functions defined in Section~\ref{sec:function_def}. Upon replacing the matrices in Eq.~\ref{eq:minor2} with matrices $\mathbf{O}_{ij}$, we get,
\begin{equation}
\begin{split}
    P^{ii}(\omega) =&
    \sum_{m=1}^n \sigma^2_m \biggl( l_{im}^2 |\mathbf{O}_{ii} +  \dot{\iota}\omega\mathbf{I}| \overline{|\mathbf{O}_{ii} + \dot{\iota}\omega\mathbf{I}|} + 2 \sum_{\substack{j=1 \\ j\neq i}}^n  (-1)^{i+j} l_{im}l_{jm} \mathrm{Re}\bigl[|\mathbf{O}_{ii} +  \dot{\iota}\omega\mathbf{I}| (-1)^{|i-j|-1} \overline{|\mathbf{O}_{ji} + \dot{\iota}\omega\mathbf{I} - \dot{\iota}\omega\mathbf{E}_{\min(i,j)}|}\bigr]  \\
    & + \sum_{\substack{j=1 \\ j\neq i}}^n  l_{jm}^2 (-1)^{|i-j|-1}|\mathbf{O}_{ji} + \dot{\iota}\omega\mathbf{I} - \dot{\iota}\omega\mathbf{E}_{\min(i,j)} | (-1)^{|i-j|-1}\overline{|\mathbf{O}_{ji} + \dot{\iota}\omega\mathbf{I} - \dot{\iota}\omega\mathbf{E}_{\min(i,j)}|} \\
    & + 2 \sum_{\substack{j=1 \\ j\neq i}}^n \sum_{\substack{k=1 \\ k \neq i}}^{j-1} (-1)^{j+k} l_{jm}l_{km} \mathrm{Re} \left[(-1)^{|j-i|-1}|\mathbf{O}_{ji} + \dot{\iota}\omega\mathbf{I} - \dot{\iota}\omega\mathbf{E}_{\min(i,j)}|  (-1)^{|k-i|-1}\overline{|\mathbf{O}_{ki} + \dot{\iota}\omega\mathbf{I} - \dot{\iota}\omega\mathbf{E}_{\min(i,k)}|} \right] \biggr)
    \label{eq:Os}
\end{split}
\end{equation}
On simplifying this equation, we get,
\begin{equation}
\begin{split}
    P^{ii}(\omega) =
    \sum_{m=1}^n \sigma^2_m \biggl(& l_{im}^2 |\mathbf{O}_{ii} +  \dot{\iota}\omega\mathbf{I}| \overline{|\mathbf{O}_{ii} + \dot{\iota}\omega\mathbf{I}|} - 2 \sum_{\substack{j=1 \\ j\neq i}}^n   l_{im}l_{jm} \mathrm{Re}\bigl[|\mathbf{O}_{ii} +  \dot{\iota}\omega\mathbf{I}|  \overline{|\mathbf{O}_{ji} + \dot{\iota}\omega\mathbf{I} - \dot{\iota}\omega\mathbf{E}_{\min(i,j)}|}\bigr] \\& + \sum_{\substack{j=1 \\ j\neq i}}^n  l_{jm}^2 |\mathbf{O}_{ji} + \dot{\iota}\omega\mathbf{I} - \dot{\iota}\omega\mathbf{E}_{\min(i,j)} | \overline{|\mathbf{O}_{ji} + \dot{\iota}\omega\mathbf{I} - \dot{\iota}\omega\mathbf{E}_{\min(i,j)}|} \\
    & + 2 \sum_{\substack{j=1 \\ j\neq i}}^n \sum_{\substack{k=1 \\ k \neq i}}^{j-1}  l_{jm}l_{km} \mathrm{Re} \biggl[|\mathbf{O}_{ji} + \dot{\iota}\omega\mathbf{I} - \dot{\iota}\omega\mathbf{E}_{\min(i,j)}|  
    \overline{|\mathbf{O}_{ki} + \dot{\iota}\omega\mathbf{I} - \dot{\iota}\omega\mathbf{E}_{\min(i,k)}|} \biggr] \biggr)
    \label{eq:Os2}
\end{split}
\end{equation}

Using the definitions of the functions above, we can write Eq.~\ref{eq:Os2} as,
\begin{equation}
\begin{split}
    P^{ii}(\omega) = 
    \sum_{m=1}^n \sigma^2_m \biggl(& l_{im}^2 d(\omega; \mathbf{O}_{ii}) + \sum_{\substack{j=1 \\ j\neq i}}^n \sum_{\substack{k=1 \\ k \neq i}}^{j-1}  l_{jm}l_{km} t_1(\omega; \mathbf{O}_{ji},\mathbf{O}_{ki},\mathbf{e}_{\min( i,j )},\mathbf{e}_{\min( i,k )}) \\
    & + \sum_{\substack{j=1 \\ j\neq i}}^n  l_{jm}^2 f(\omega; \mathbf{O}_{ji},\mathbf{e}_{\min( i,j )}) - \sum_{\substack{j=1 \\ j\neq i}}^n   l_{im}l_{jm} s_1(\omega; \mathbf{O}_{ii},\mathbf{O}_{ji},\mathbf{e}_{\min( i,j )})  \biggr)
    \label{eq:Os3}
\end{split}
\end{equation}
Eq.~\ref{eq:Os3} defines the closed form polynomial expansion of $P_{ii}(\omega)$ in terms of $\omega$. We can use the corresponding functions to find the coefficient of $\omega^{2\alpha}$. In general, we can write the auto-spectrum for the $i^{th}$ variable as the following polynomial expansion,
\begin{equation}
\begin{aligned}
    {\mathcal{S}}^{ii}(\omega) = \frac{P^{ii}(\omega)}{Q(\omega)} = \frac{p_0^{ii} + p_1^{ii}\omega^2 + p_2^{ii}\omega^4 + ...+p_{\alpha}^{ii}\omega^{2\alpha}+...+p_{n-1}^{ii}\omega^{2n-2}}{q_0 + q_1\omega^2 + q_2\omega^4 + ...+q_{\alpha}\omega^{2\alpha}+...+q_{n}\omega^{2n}} \label{eq:psd_p_q}
\end{aligned}
\end{equation}
Such that,
\begin{equation}
\begin{aligned}
    P^{ii}(\omega) = p^{ii}_0 + p^{ii}_1\omega^2 + p^{ii}_2\omega^4 + ...+p^{ii}_{\alpha}\omega^{2\alpha}+...+p^{ii}_{n-1}\omega^{2n-2}
\end{aligned}
\end{equation}
Where $p_\alpha$ and $q_\alpha$ are coefficients of different even powers of $\omega$ in the numerator and the denominator, respectively. In general for the $i^{th}$ variable, using the Eq.~\ref{eq:Os3}, we can write $p_\alpha$ as, 
\begin{equation}
\begin{split}
   p^{ii}_\alpha  = \sum_{m=1}^n \sigma^2_m \biggl(& l_{im}^2 d^\alpha(\mathbf{O}_{ii}) + \sum_{\substack{j=1 \\ j\neq i}}^n  l_{jm}^2 f^\alpha(\mathbf{O}_{ji},\mathbf{e}_{\min( i,j )}) - \sum_{\substack{j=1 \\ j\neq i}}^n   l_{im}l_{jm} s_1^\alpha(\mathbf{O}_{ii},\mathbf{O}_{ji},\mathbf{e}_{\min( i,j )}) \\&
   + \sum_{\substack{j=1 \\ j\neq i}}^n \sum_{\substack{k=1 \\ k \neq i}}^{j-1}  l_{jm}l_{km} t_1^\alpha(\mathbf{O}_{ji},\mathbf{O}_{ki},\mathbf{e}_{\min( i,j )},\mathbf{e}_{\min( i,k )})\biggr) \label{eq:Os4}
\end{split}
\end{equation}

\textit{\textbf{Special case of independent noise}}:
Now, since the denominator for the auto-spectrum is the same for all the variables and only depends on the Jacobian of the system, we focus on the numerator in the following analysis. We consider special cases of noise correlations and show how the expression for the coefficients of $\omega$ in the numerator simplifies. The dynamical system with uncorrelated noise can be analytically realized by constraining the matrix $\mathbf{L}$ to be diagonal such that $l_{ij}=0$ when $i\neq j$. Using this fact Eq.~\ref{eq:Os4} simplifies to,
\begin{equation}
\begin{split}
   p^{ii}_\alpha  = \sigma^2_i \, l_{ii}^2 \,  d^\alpha(\mathbf{O}_{ii}) + \sum_{\substack{j=1 \\ j\neq i}}^n \sigma^2_j \,  l_{jj}^2 \, f^\alpha(\mathbf{O}_{ji},\mathbf{e}_{\min\{ i,j \}})
    \label{eq:Os5}
\end{split}
\end{equation}

\subsubsection{Cross-spectrum}
We proceed similarly to calculate the polynomial expansion of the cross-power spectral density (cross-spectrum), i.e., the off-diagonal elements of $\mathbf{Z}(\omega)$. Now when $i\neq j$ Eq.\ref{eq:minor1} can be split into four parts as,
\begin{equation}
\begin{split}
    Z^{ij}(\omega) = \sum_{m=1}^n \sigma^2_m \biggl(&  (-1)^{2(i+j)} l_{im} l_{jm} M_{ii} \overline{M}_{jj}  + \sum_{\substack{k=1 \\ k\neq j}}^n (-1)^{i+j+i+k} l_{im} l_{km} M_{ii} \overline{M}_{kj}\\
    & + \sum_{\substack{k=1 \\ k\neq i}}^n (-1)^{i+j+k+j} l_{km} l_{jm} M_{ki} \overline{M}_{jj} + \sum_{\substack{k=1 \\ k\neq i}}^n \sum_{\substack{q=1 \\ q \neq j}}^{n}  (-1)^{i+j+k+q} l_{km} l_{qm} M_{ki} \overline{M}_{qj} \biggr)
\end{split} \label{eq:cross-minor2}
\end{equation}
Since the solution is complex, $Z^{ij}(\omega)=P^{ij}(\omega) + \dot{\iota} \omega   P^{ij\prime}(\omega)$. 
After replacing the minors with the determinants of the matrices $\mathbf{O}_{ij}$, we find,
\begin{equation}
\begin{split}
    Z^{ij}(\omega) = \sum_{m=1}^n  &\sigma^2_m  \biggl( l_{im} l_{jm} |\mathbf{O}_{ii}+\dot{\iota}\omega\mathbf{I}| \overline{|\mathbf{O}_{jj}+\dot{\iota}\omega\mathbf{I}|}   + \sum_{\substack{k=1 \\ k\neq j}}^n (-1)^{j+k} l_{im} l_{km} |\mathbf{O}_{ii}+\dot{\iota}\omega\mathbf{I}| (-1)^{|k-j|-1} \overline{|\mathbf{O}_{kj}+\dot{\iota}\omega\mathbf{I} -\dot{\iota} \omega \mathbf{E}_{\min(k,j)} |} \\
    & + \sum_{\substack{k=1 \\ k\neq i}}^n (-1)^{i+k} l_{km}  l_{jm} (-1)^{|k-i|-1} |\mathbf{O}_{ki}+\dot{\iota}\omega\mathbf{I} -\dot{\iota} \omega \mathbf{E}_{\min(k,i)} | \overline{|\mathbf{O}_{jj}+\dot{\iota}\omega\mathbf{I}|}  \\ 
    & + \sum_{\substack{k=1 \\ k\neq i}}^n \sum_{\substack{q=1 \\ q \neq j}}^{n} (-1)^{i+j+k+q} l_{km} l_{qm} (-1)^{|k-i|-1} |\mathbf{O}_{ki}+\dot{\iota}\omega\mathbf{I} -\dot{\iota} \omega \mathbf{E}_{\min(k,i)}|  (-1)^{|q-j|-1} \overline{|\mathbf{O}_{qj}+\dot{\iota}\omega\mathbf{I} -\dot{\iota} \omega \mathbf{E}_{\min(q,j)} |} \biggr)
\end{split} \label{eq:cross-O1}
\end{equation}
Upon further simplification, we get,
\begin{equation}
\begin{split}
   Z^{ij}(\omega)  =  \sum_{m=1}^n  \sigma^2_m &\biggl( l_{im} l_{jm} |\mathbf{O}_{ii}+\dot{\iota}\omega\mathbf{I}| \overline{|\mathbf{O}_{jj}+\dot{\iota}\omega\mathbf{I}|}  + \sum_{\substack{k=1 \\ k\neq i}}^n \sum_{\substack{q=1 \\ q \neq j}}^{n}  l_{km} l_{qm}  |\mathbf{O}_{ki}+\dot{\iota}\omega\mathbf{I} -\dot{\iota} \omega \mathbf{E}_{\min(k,i)}| \overline{|\mathbf{O}_{qj}+\dot{\iota}\omega\mathbf{I} -\dot{\iota} \omega \mathbf{E}_{\min(q,j)} |}\\ & - \sum_{\substack{k=1 \\ k\neq j}}^n  l_{im} l_{km} |\mathbf{O}_{ii}+\dot{\iota}\omega\mathbf{I}| \overline{|\mathbf{O}_{kj}+\dot{\iota}\omega\mathbf{I} -\dot{\iota} \omega \mathbf{E}_{\min(k,j)} |}  - \sum_{\substack{k=1 \\ k\neq i}}^n  l_{km} l_{jm}|\mathbf{O}_{ki}+\dot{\iota}\omega\mathbf{I} -\dot{\iota} \omega \mathbf{E}_{\min(k,i)} | \overline{|\mathbf{O}_{jj}+\dot{\iota}\omega\mathbf{I}|}  \biggr)
\end{split} \label{eq:cross-O2}
\end{equation}
Since we know that the cross-spectrum is a complex expression, we calculate the real and imaginary parts separately. Using the expansion in Eq.~\ref{eq:recursive2}, we can write the complex rational function for the cross-spectrum as,
\begin{equation}
\begin{aligned}
    \mathcal{S}^{ij}(\omega) &= \frac{P^{ij}(\omega) + \dot{\iota} \omega P^{ij\prime}(\omega)}{Q(\omega)} \\
    &= \frac{(p^{ij}_0 + ...+p^{ij}_{\alpha}\omega^{2\alpha}+...+p^{ij}_{n-1}\omega^{2n-2}) + \dot{\iota}\omega (p^{ij\prime}_0 +  ...+p^{ij\prime}_{\alpha}\omega^{2\alpha}+...+p^{ij\prime}_{n-2}\omega^{2n-4})}{q_0 + q_1\omega^2 + q_2\omega^4 + ...+q_{\alpha}\omega^{2\alpha}+...+q_{n}\omega^{2n}} \label{eq:cpsd_p_q}
\end{aligned}
\end{equation}
Where $p^{ij}_\alpha$ and $p^{ij\prime}_\alpha$ are coefficients of different even (real) and odd (imaginary) powers of $\omega$ in the numerator, respectively.

\textit{\textbf{Real part of the numerator}}:
We first find the real part of Eq.~\ref{eq:cross-O2}. Using the function definitions from Section.~\ref{sec:function_def}, we can write the real part of Eq.~\ref{eq:cross-O2} as,
\begin{equation}
\begin{split}
    P^{ij}(\omega) =  \sum_{m=1}^n  \frac{\sigma^2_m}{2}  \biggl(& l_{im} l_{jm} h_1(\omega; \mathbf{O}_{ii}, \mathbf{O}_{jj})  - \sum_{\substack{k=1 \\ k\neq j}}^n  l_{im} l_{km} s_1(\omega; \mathbf{O}_{ii}, \mathbf{O}_{kj}, \mathbf{e}_{\min( k,j)})  - \sum_{\substack{k=1 \\ k\neq i}}^n  l_{km} l_{jm} s_1(\omega; \mathbf{O}_{jj}, \mathbf{O}_{ki}, \mathbf{e}_{\min( k,i)}) \\
    &+ \sum_{\substack{k=1 \\ k\neq i}}^n \sum_{\substack{q=1 \\ q \neq j}}^{n}  l_{km} l_{qm} t_1(\omega; \mathbf{O}_{ki}, \mathbf{O}_{qj}, \mathbf{e}_{\min(i,k)}, \mathbf{e}_{\min(j,q)}) \biggr) 
\end{split} \label{eq:cross-O3}
\end{equation}
Eq.~\ref{eq:cross-O3} represents a closed-form polynomial expansion (in terms of $\omega$) of the real part of the cross-spectrum between $x_i$ and $x_j$. Using Eq.~\ref{eq:cross-O3}, we can thus write $p_\alpha$ as,
\begin{equation}
\begin{split}
    p^{ij}_\alpha = \sum_{m=1}^n  \frac{\sigma^2_m}{2}  \biggl(& l_{im} l_{jm} h_1^\alpha(\mathbf{O}_{ii}, \mathbf{O}_{jj}) + \sum_{\substack{k=1 \\ k\neq i}}^n \sum_{\substack{q=1 \\ q \neq j}}^{n}  l_{km} l_{qm} t_1^\alpha(\mathbf{O}_{ki}, \mathbf{O}_{qj}, \mathbf{e}_{\min(i,k)}, \mathbf{e}_{\min(j,q)}) \\
    & - \sum_{\substack{k=1 \\ k\neq j}}^n  l_{im} l_{km} s_1^\alpha(\mathbf{O}_{ii}, \mathbf{O}_{kj}, \mathbf{e}_{\min( k,j)})  - \sum_{\substack{k=1 \\ k\neq i}}^n  l_{km} l_{jm} s_1^\alpha(\mathbf{O}_{jj}, \mathbf{O}_{ki}, \mathbf{e}_{\min( k,i)}) \biggr) 
\end{split} \label{eq:cross-O4}
\end{equation}

\textit{\textbf{Imaginary part of the numerator}}:
Next, we find the imaginary part of Eq.~\ref{eq:cross-O2}. Using the function definitions from Section.~\ref{sec:function_def}, we can write the imaginary part of Eq.~\ref{eq:cross-O2} as,
\begin{equation}
\begin{split}
    \omega P^{ij\prime}(\omega) =  \sum_{m=1}^n  \frac{\sigma^2_m}{2} \biggl(&- l_{im} l_{jm} h_2(\omega; \mathbf{O}_{ii}, \mathbf{O}_{jj}) + \sum_{\substack{k=1 \\ k\neq i}}^n \sum_{\substack{q=1 \\ q \neq j}}^{n}  l_{km} l_{qm} t_2(\omega; \mathbf{O}_{ki}, \mathbf{O}_{qj}, \mathbf{e}_{\min(i,k)}, \mathbf{e}_{\min(j,q)}) \\& - \sum_{\substack{k=1 \\ k\neq j}}^n  l_{im} l_{km} s_2(\omega; \mathbf{O}_{ii}, \mathbf{O}_{kj}, \mathbf{e}_{\min( k,j)})  + \sum_{\substack{k=1 \\ k\neq i}}^n  l_{km} l_{jm} s_2(\omega; \mathbf{O}_{jj}, \mathbf{O}_{ki}, \mathbf{e}_{\min( k,i)})  \biggr) 
\end{split} \label{eq:cross-O3-im}
\end{equation}
Eq.~\ref{eq:cross-O3-im} represents a closed-form polynomial expansion (in terms of $\omega$) of the imaginary part of the cross-spectrum between $x_i$ and $x_j$. Using Eq.~\ref{eq:cross-O3-im}, we can thus write $p^{ij\prime}_\alpha$ as,
\begin{equation}
\begin{split}
    p^{ij\prime}_\alpha =  \sum_{m=1}^n  \frac{\sigma^2_m}{2} \biggl(& - l_{im} l_{jm} h_2^\alpha(\mathbf{O}_{ii}, \mathbf{O}_{jj}) + \sum_{\substack{k=1 \\ k\neq i}}^n \sum_{\substack{q=1 \\ q \neq j}}^{n}  l_{km} l_{qm} t_2^\alpha(\mathbf{O}_{ki}, \mathbf{O}_{qj}, \mathbf{e}_{\min(i,k)}, \mathbf{e}_{\min(j,q)}) \\
    & - \sum_{\substack{k=1 \\ k\neq j}}^n  l_{im} l_{km} s_2^\alpha(\mathbf{O}_{ii}, \mathbf{O}_{kj}, \mathbf{e}_{\min( k,j)})  + \sum_{\substack{k=1 \\ k\neq i}}^n  l_{km} l_{jm} s_2^\alpha(\mathbf{O}_{jj}, \mathbf{O}_{ki}, \mathbf{e}_{\min( k,i)}) \biggr) 
\end{split} \label{eq:cross-O4-im}
\end{equation}

\textit{\textbf{Special case of independent noise}}:
Now, since the denominator for the cross-spectrum is the same for all the variables and is equal to the auto-spectrum, we focus on the numerator in the following analysis. Along lines similar to the auto-spectrum, we consider the special case of independent (across variables) Gaussian noise and show how the expression for the coefficients of $\omega$ in the numerator simplifies. The dynamical system with uncorrelated noise can be analytically realized by constraining the matrix $\mathbf{L}$ to be diagonal such that $l_{ij}=0$, when $i\neq j$. Using this fact Eq.~\ref{eq:cross-O4} simplifies to,
\begin{equation}
\begin{split}
    p^{ij\prime}_\alpha = &\sum_{\substack{m=1 \\ m\neq (i,j)}}^n \frac{\sigma^2_m}{2}    l_{mm}^2 t_1^\alpha(\mathbf{O}_{mi}, \mathbf{O}_{mj}, \mathbf{e}_{\min(i,m)}, \mathbf{e}_{\min(j,m)}) - \frac{\sigma^2_i}{2} l_{ii}^2 s_1^\alpha(\mathbf{O}_{ii}, \mathbf{O}_{ij}, \mathbf{e}_{\min( i,j)}) - \frac{\sigma^2_j}{2} l_{jj}^2 s_1^\alpha(\mathbf{O}_{jj}, \mathbf{O}_{ji}, \mathbf{e}_{\min( j,i)}).  \label{eq:cross-indp}
\end{split}
\end{equation}
Similarly, Eq.~\ref{eq:cross-O4-im} simplifies to,
\begin{equation}
\begin{split}
    p^{ij\prime}_\alpha = &\sum_{\substack{m=1 \\ m\neq (i,j)}}^n \frac{\sigma^2_m}{2}    l_{mm}^2 t_2^\alpha(\mathbf{O}_{mi}, \mathbf{O}_{mj}, \mathbf{e}_{\min(i,m)}, \mathbf{e}_{\min(j,m)}) - \frac{\sigma^2_i}{2} l_{ii}^2 s_2^\alpha(\mathbf{O}_{ii}, \mathbf{O}_{ij}, \mathbf{e}_{\min( i,j)}) + \frac{\sigma^2_j}{2} l_{jj}^2 s_2^\alpha(\mathbf{O}_{jj}, \mathbf{O}_{ji}, \mathbf{e}_{\min( j,i)}).  \label{eq:cross-indp-imag}
\end{split}
\end{equation}

%\subsection{Independent noise}

\subsubsection{Coherence as a rational function}
Coherence is a measure of the similarity of two signals as a function of frequency. Mathematically, coherence ($\kappa_{ij}$) between two variables $x_i$ and $x_j$ is defined as the ratio of the squared magnitude of the cross-power spectral density and their individual power spectral densities. Given the representation of the cross-spectrum and auto-spectrum in terms of numerator and denominator, using Eq.~\ref{eq:psd_p_q} and \ref{eq:cpsd_p_q}, we can write coherence as a rational function,
\begin{equation}
\begin{split}
    \kappa^{ij} &= \frac{|\mathcal{S}^{ij}(\omega)|^2}{\mathcal{S}^{ii}(\omega)\mathcal{S}^{jj}(\omega)} = \frac{|P^{ij}(\omega) + \dot{\iota} \omega P^{ij\prime}(\omega)|^2}{P^{ii}(\omega)P^{jj}(\omega)} = \frac{ \left(\sum\limits^{n-1}_{\alpha=0} p_{\alpha}^{ij} \omega^{2\alpha}\right)^2 + \omega^2 \left(\sum\limits^{n-2}_{\alpha=0} p_{\alpha}^{ij\prime} \omega^{2\alpha}\right)^2}{\left(\sum\limits^{n-1}_{\alpha=0}p_{\alpha}^{ii}\omega^{2\alpha} \right)\left(\sum\limits^{n-1}_{\alpha=0}p_{\alpha}^{jj}\omega^{2\alpha} \right)},
\end{split}
\end{equation}
where the superscript of $p$ represents that the coefficient comes from the numerator of the polynomial for cross-spectrum ($ij$) or the auto-spectrum ($ii$ or $jj$).

\newpage
\begin{figure*}[h]
    \centering
    \includegraphics[width=0.9\linewidth]{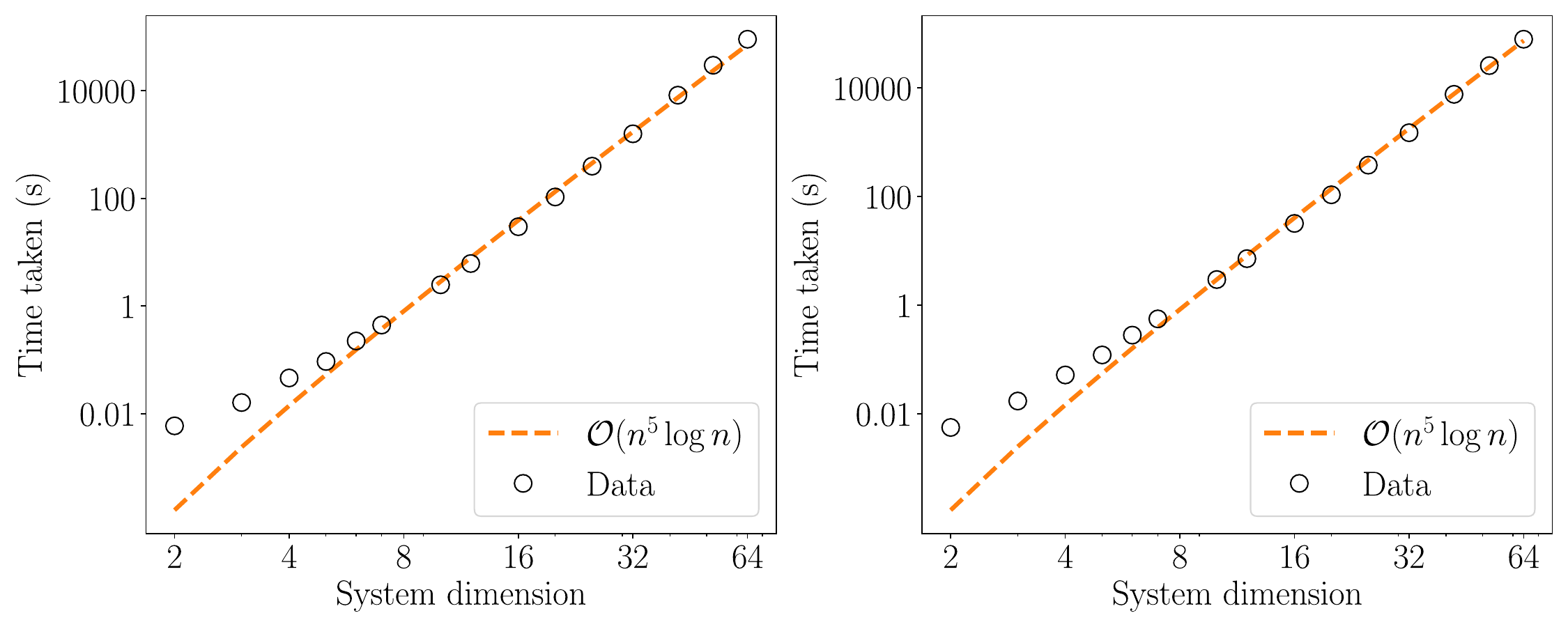}
    \caption{\textbf{Time complexity of the element-wise algorithm.} The time complexity of finding the auto-spectrum (left) of a given variable and cross-spectrum (right) between any two variables for a non-sparse SDE with additive Gaussian white noise with random correlations using \textit{rational computations}. The algorithms run with time complexity of $\mathcal{O}(n^5 \log n )$.}
    \label{fig:tc}
\end{figure*}

\newpage
\section{Models in detail}
In this section, we describe in more detail the various models that are considered for analysis in the main text. 

\subsection{Fitzhugh–Nagumo model}
The SDEs describing the model can be written as,
\begin{equation}\begin{split}
    \frac{\mathrm{d}v}{\mathrm{d}t} &= \biggl(v - \frac{v^3}{3} - w + I \biggr) + h_1(v) \eta_1 \\ \frac{\mathrm{d}w}{\mathrm{d}t} &= \epsilon (v + \alpha -\beta w ) + h_2(w) \eta_2  \label{eq:fzn1}
\end{split}
\end{equation}
The analytical steady-state solution of the deterministic system $(v_e, w_e)$ can be found by setting the RHS of Eq.~\ref{eq:fzn1} equal to $0$. Therefore, $v_e$ is the solution to the cubic equation,
\begin{equation}
    v^3 + 3\left(\frac{1}{\beta}-1\right) v + 3 \left(\frac{\alpha}{\beta} - I\right) = 0
\end{equation}
and $w_e = (v_e + \alpha)/\beta$. The sign of the discriminant, $\Delta = (1/\beta-1)^3 + \frac{9}{4}(\alpha/\beta -I)^2$, of the cubic equation determines if the system admits one or more fixed points. We pick the parameters such that $\Delta>0$, yielding a system with a global attractor. Since the solution to $v_e$ is from a depressed cubic, it can be analytically expressed in terms of $\Delta$ using Cardano's formula, 
\begin{equation}
    v_e = \sqrt[3]{-\frac{3}{2} \left(\frac{\alpha}{\beta} - I\right) - \sqrt{\Delta}} + \sqrt[3]{-\frac{3}{2} \left(\frac{\alpha}{\beta} - I\right) + \sqrt{\Delta}}.
\end{equation}
The parameters used to simulate the system are $I = 0.265$, $\alpha = 0.7$, $\beta = 0.75$ and $\epsilon = 0.08$ that yields a fixed point solution $v_e = -1.00125$, $w_e = -0.401665$. We consider the case of multiplicative noise and set $h_1(v)=0$ and $h_2(w)=\sigma w$ with $\sigma=0.001$.

\subsection{Hindmarsh-Rose}
The parameters used for simulation are as follows: $I=5.5$, $b=0.5$, $\mu=0.01$, $x_{rest}=-1.6$, $s=4$ and $\sigma=0.001$. In addition to the auto-spectrum shown in Fig.~3 of the main text, we confirm that the solutions for cross-spectrum and coherence between the variables $x$ and $y$ match across the entire range of frequencies (Fig.~\ref{fig:psd_hr}).

\subsection{Wilson-Cowan model}
The parameters for simulations and analytical calculations are as follows: The time constants: $\tau_\mathrm{E} = 2$ms, $\tau_\mathrm{I} = 8$ms, $\tau_{s_\mathrm{E}} = 10$ms and $\tau_{s_\mathrm{I}} = 10$ms. The strength of the weights: $w_{\mathrm{EE}}=5$, $w_{\mathrm{EI}}=5$, $w_{\mathrm{IE}}=3.5$ and $w_{\mathrm{II}}=3$. The offset of the input: $\theta_\mathrm{E}=0.4$ and $\theta_\mathrm{I}=0.4$. The scaling of the input: $\kappa_\mathrm{E}=0.2$ and $\kappa_\mathrm{I}=0.02$. The scaling of the population activity: $\gamma_\mathrm{E}=1$ and $\gamma_\mathrm{I}=2$. The standard deviation of the noise: $\eta_1=\eta_2=0.001$ and $\eta_3=\eta_4=0.002$. Finally the external inputs to the system: $I_\mathrm{E}=1$, $I_\mathrm{I}=0.5$, $s_\mathrm{E}^0=0.2$ and $s_\mathrm{I}^0=0.05$.

\subsection{Stabilized Supralinear Network (SSN)}
The input, $\mathbf{h}(\mathbf{x})$, is defined to be discretized into 11 points in space, where the location of each point (in arbitrary units of length) is given by the vector $\mathbf{x}$. The input is centered at 0 with the distance between any two consecutive points $\Delta x = 3$. The input at a given location $x$ is given by,
\begin{equation}
    \mathbf{h}(x) = c\left(\frac{1}{1+e^{-\frac{x+l/2}{\sigma_{RF}}}} \right) \left(1-\frac{1}{1+e^{-\frac{x-l/2}{\sigma_{RF}}}} \right)
\end{equation}
where $l=9$ is the length of the stimulus, and $\sigma_{RF}=0.125\Delta x$ defines the width of the edges of the input.

The E-E (excitatory to excitatory) and I-E (inhibitory to excitatory) connections are given by a Gaussian function. The $i,j$ entry of the weight matrix is defined as follows,
\begin{equation}
    \mathbf{W}_{\alpha \mathrm{E}}(j,k) = J_{\alpha \mathrm{E}}e^{-\frac{(x_j-y_k)^2}{2\sigma_{\alpha \mathrm{E}}^2}}, \quad \alpha \in \{ \mathrm{E}, \mathrm{I}\}
\end{equation}
where $x_i$ and $y_j$ are the locations of the $i^{th}$ and $j^{th}$ neurons, respectively. $J_{ \mathrm{ \mathrm{EE}}} = 2.0$, $J_{ \mathrm{IE}} = 2.25$ define the magnitude of the weights, and $\sigma_{ \mathrm{EE}}=4$ and $\sigma_{ \mathrm{IE}}=8$ are the parameters of the Gaussians determining the spatial extent of the connections between E-E and I-E populations, respectively. The E-I and I-I connections are not fully-connected but self-inhibitory with the weights $J_{ \mathrm{EI}} = 0.9$, $J_{ \mathrm{II}} = 0.5$.

The intrinsic time constants of the excitatory and inhibitory neurons are $\tau_\mathrm{E}=6$ms and $\tau_\mathrm{I}=4$ms. The supralinear activation parameters are $n=2.2$ and $k=0.01$. The standard deviation of the noise added $\mathbf{\eta}_\mathrm{E}(x_i) = \mathbf{\eta}_\mathrm{I}(x_i) = 0.01$.

\subsection{Rock-Paper-Scissors-Lizard-Spock}
Here we show the results for the 5-dimensional version of the Rock-Paper-Scissors game. This game of Rock-Paper-Scissors-Lizard-Spock can be described by the pay-off matrix  ($\mathbf{P}$) given below. Note that the player and the opponent choose their strategies from the rows and the columns of the table, respectively. 1 indicates you win against your opponent, -1 indicates you lose, and 0 indicates you neither win nor lose against your opponent.
\begin{center}
\begin{tabular}{|| c || c | c | c | c | c ||} 
 \hline
  & R & P & S & Sp & L \\ 
 \hline\hline
 R & 0 & -1 & 1 & -1 & 1 \\ 
 \hline
 P & 1 & 0 & -1 & 1 & -1  \\
 \hline
 S & -1 & 1 & 0 & -1 & 1  \\
 \hline
 Sp & 1 & -1 & 1 & 0 & -1 \\
 \hline
 L & -1 & 1 & -1 & 1 & 0 \\ 
 \hline
\end{tabular}
\end{center}
The equations describing the evolution of populations following a certain strategy with global mutations ($\mu=0.01$) and multiplicative Gaussian white noise (with standard deviation $\sigma=0.01$) are defined as follows,
\begin{equation}
\begin{split}
    \dot{x_1} &= x_1 (-x_2 + x_3 - x_4 + x_5 - \phi) + \mu(-4x_1 + x_2 + x_3 + x_4 + x_5) +\sigma x_1 \eta_1 \\
    \dot{x_2} &= x_2 (x_1 - x_3 + x_4 - x_5 - \phi) + \mu(-4x_2 + x_1 + x_3 + x_4 + x_5) +\sigma x_2\eta_2 \\
    \dot{x_3} &= x_3 (-x_1 + x_2 - x_4 + x_5 - \phi) + \mu(-4x_3 + x_1 + x_2 + x_4 + x_5) +\sigma x_3\eta_3 \\
    \dot{x_4} &= x_4 (x_1 - x_2 + x_3 - x_5 - \phi) + \mu(-4x_4 + x_1 + x_2 + x_3 + x_5) +\sigma x_4\eta_4 \\
    x_1 &+ x_2 + x_3 + x_4 + x_5 = 1
\end{split}
\end{equation}
The fixed point solution of the equations above is $x_1^{ss} = x_2^{ss} = x_3^{ss} = x_4^{ss} = 1/5$. Therefore, the effective standard deviation of the noise at steady-state is $\sigma_e=\sigma/5$. Further, the Jacobian matrix is given by,
\begin{equation}
    \mathbf{J} = 
    \begin{bmatrix}
    -1/5 - 5\mu & -2/5 & 0 & -2/5\\
     2/5 & 1/5 - 5\mu & 0 & 2/5\\
    -2/5 & 0 & -1/5 - 5\mu & -2/5\\
    2/5 & 0 & 2/5 & 1/5 - 5\mu\\
    \end{bmatrix}
\end{equation}

Using our analytical solution, we have access to the coefficients of $\omega$ in the numerator and the denominator. Therefore, we can write the rational function for the auto-spectrum of the Paper population as,
\begin{equation}
    \mathcal{S}^{22}(\omega) = \frac{p_0 + p_1 \omega^2 + p_2 \omega^4 + p_3 \omega^6}{q_0 + q_1 \omega^2 + q_2 \omega^4 + q_3 \omega^6 + q_4 \omega^8} \label{eq:rational_full_rps}
\end{equation}
with the coefficients for the numerator given by,
\begin{equation}
\begin{split}
    p_0 &=  \left(15625 \mu ^6+1250 \mu ^5+575 \mu ^4+28 \mu ^3+\frac{79 \mu ^2}{25}-\frac{6 \mu }{125}+\frac{1}{625}\right) \sigma_e ^2\\
    p_1 &=  \left(1875 \mu ^4+100 \mu ^3+18 \mu ^2+\frac{36 \mu }{25}+\frac{83}{625}\right) \sigma_e ^2 \\
    p_2 &= \left(75 \mu ^2+2 \mu -\frac{1}{5}\right) \sigma_e ^2 \\
    p_3 &=  \sigma_e ^2
\end{split}
\end{equation}
and for the denominator given by,
\begin{equation}
\begin{split}
    q_0 &= 390625 \mu ^8+12500 \mu ^6+110 \mu ^4+\frac{4 \mu ^2}{25}+\frac{1}{15625} \\
    q_1 &=  62500 \mu ^6+500 \mu ^4+\frac{28 \mu ^2}{5}-\frac{4}{625} \\
    q_2 &= 3750 \mu ^4-20 \mu ^2+\frac{22}{125} \\
    q_3 &=  100 \mu ^2-\frac{4}{5} \\
    q_4 &= 1
\end{split}
\end{equation}
Similarly, we can write the coefficients for the rational function for the cross-spectrum (between the Rock and Paper populations). These coefficients (for auto/cross-spectrum) define the spectral density over the entire frequency domain. Since the elements of the Jacobian matrix depend on the mutation rate $\mu$ and on the noise standard deviation $\sigma$, the coefficients $p_i$ and $q_i$ are also a function of $\mu$ and $\sigma$. We first confirm that the solutions for the auto-spectrum (for the paper population), cross-spectrum, and coherence (between the Rock and the Paper populations) match across the range of frequencies for the different methods (Fig.~\ref{fig:psd_rps}). We then plot (Fig.~\ref{fig:coeff_rps}) the coefficients $p_i$ and $q_i$ for varying mutation rates and find the values of $\mu$ for which the sign of any of the coefficients changes. We also plot the PSD corresponding to those mutation rates and find that these values of $\mu$ correspond to the emergence of oscillations (marked by the appearance of a peak in the PSD).

We can approximate the spectra to isolate the low frequency peak in the spectrum by discarding the terms with higher powers of $\omega$ which gives us,
\begin{equation}
    \mathcal{S}^{22}_{\mathrm{small}}(\omega) = \frac{p_0 + p_1 \omega^2}{q_0 + q_1 \omega^2 + q_2 \omega^4} \label{eq:rational_small_rps}
\end{equation}
Similarly, to isolate the high-frequency peak, we discard the terms with smaller powers of $\omega$,
\begin{equation}
    \mathcal{S}^{22}_{\mathrm{large}}(\omega) = \frac{p_0 + p_2 \omega^4 + p_3 \omega^6}{q_0 + q_2 \omega^4 + q_3 \omega^6 + q_4 \omega^8} \label{eq:rational_large_rps}
\end{equation}
See Fig.~9 in the main text.

\newpage
\begin{figure*}[h]
    \centering
    \includegraphics[width=\linewidth]{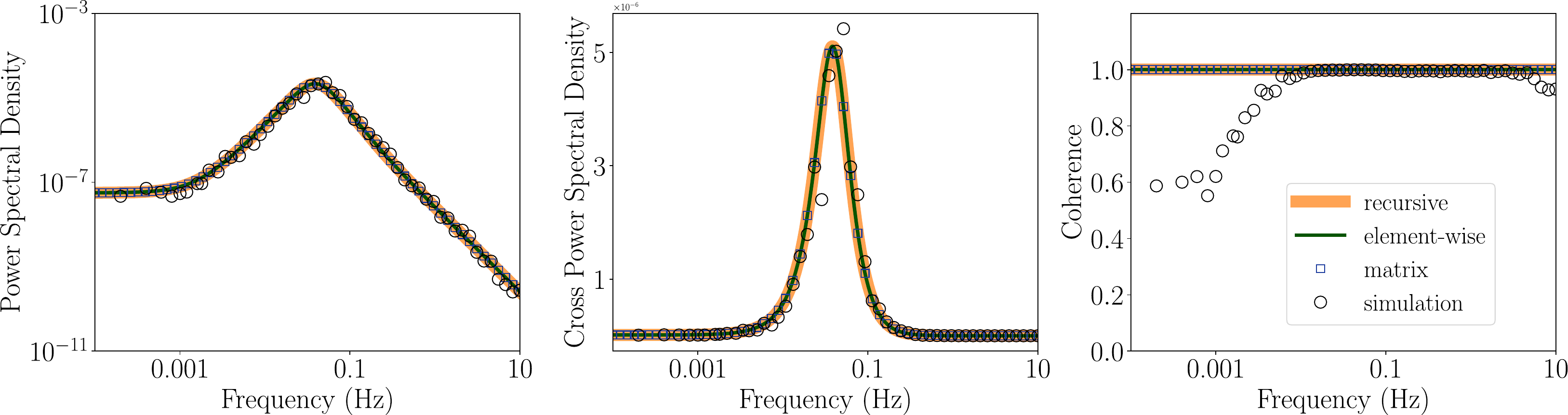}
    \caption{\textbf{Spectral density (auto and cross) and coherence for the Hindmarsh-Rose model.} We plot the power spectral density (left) of the membrane potential $x$ (Fig.~3 of the main text), the absolute cross-power spectral density (middle) between $x$ and $y$, and the coherence (right) between $x$ and $y$, each calculated via simulation, the rational function solutions, and the matrix solution.}
    \label{fig:psd_hr}
\end{figure*}
\newpage
\begin{figure}[h]
    \centering
    \includegraphics[width=\linewidth]{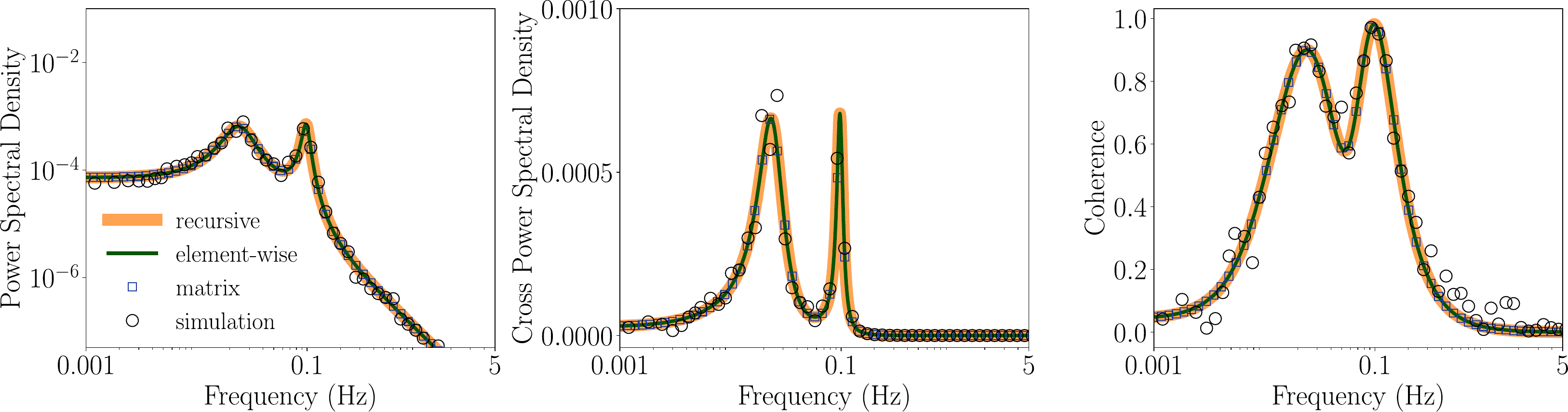}
    \caption{\textbf{Spectral density (auto and cross) and coherence for the Rock-Paper-Scissors-Lizard-Spock model.} We plot the power spectral density (left), the absolute cross power spectral density (middle), and the coherence (right) for a fixed mutation parameter ($\mu=0.01$) each calculated via simulation, the rational function solutions, and the matrix solution. The PSD is calculated for the Paper population. The cross-PSD and coherence are calculated between the Rock and Paper populations.} 
    \label{fig:psd_rps}
\end{figure}
% \newpage
\begin{figure}[h]
    \centering
    \includegraphics[width=\linewidth]{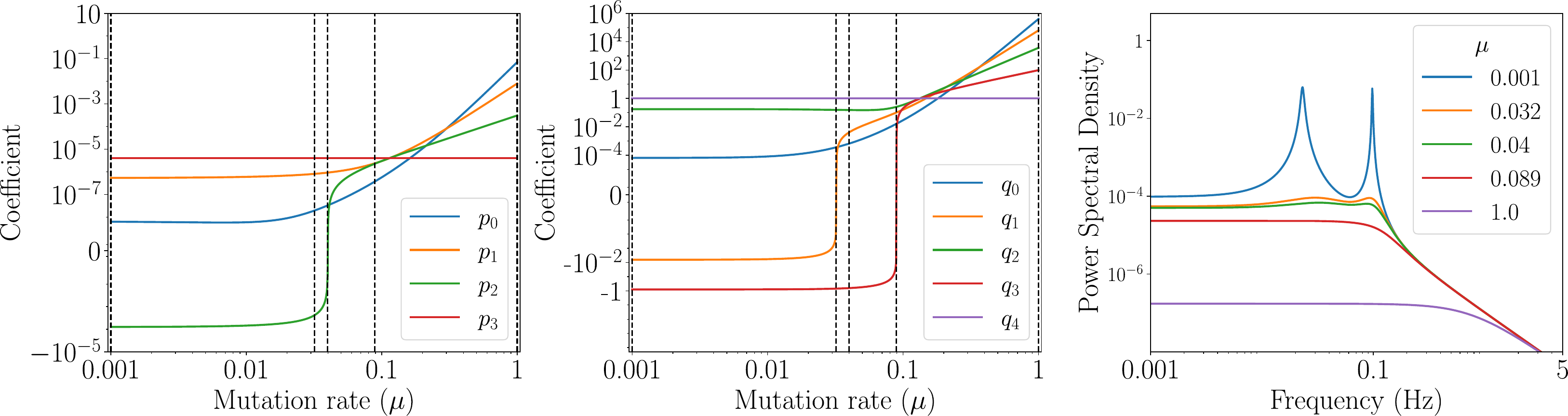}
    \caption{\textbf{Coefficients of the rational function for the auto-spectrum of Paper population in the Rock-Paper-Scissors-Lizard-Spock model.} We plot the coefficients of the numerator (left) and the denominator (middle) of the rational function of the auto-spectrum of the Paper population for a range of mutation rates $\mu$. We identify the mutation rates at which the various coefficients change sign (plotted as vertical dotted lines in the left and middle plots) and plot the corresponding power spectral density. The analysis is done for the Paper population.} 
    \label{fig:coeff_rps}
\end{figure}
% \newpage

\newpage

\section{Detailed function definitions} \label{sec:function_def}
In what follows, we define the set of functions that facilitates the derivation of expressions for auto and cross spectrums.

\subsection{$d(\omega;\mathbf{A})$} 
% \label{function:d}
The function $d(\omega;\mathbf{A} \in \R^{n\times n})$ expands the following determinant product into the corresponding polynomial in $\omega$. Here, $\mathbf{A} \in \R^{n\times n}$.
\begin{equation}
\begin{split}
    d(\omega;\mathbf{A}) &= |\mathbf{A} + \dot{\iota}\omega\mathbf{I}| \, |\mathbf{A} - \dot{\iota}\omega\mathbf{I}| \\
    &= |\mathbf{A} + \dot{\iota}\omega\mathbf{I}| \, \overline{|\mathbf{A} + \dot{\iota}\omega\mathbf{I}|} \\
    &= ||\mathbf{A} + \dot{\iota}\omega\mathbf{I}||^2 \label{eq:d_gen1}
\end{split}
\end{equation}
Denoting $\lambda_i \in \C$ to be the eigenvalues of the matrix $\mathbf{A}$, using Eq.~\ref{eq:imp2}, we can write the modulus of the determinant of the complex matrix as,
\begin{equation}
\begin{split}
    d(\omega;\mathbf{A}) &= \prod_{i=1}^n |\lambda_i + \dot{\iota} \omega|^2 \\
    &= \prod_{i=1}^n (\lambda_i^2 + \omega^2)
\end{split}
\end{equation}
Writing this product as a function of elementary symmetric polynomials (Eq.~\ref{eq:esp}), we find,
\begin{equation}
    d(\omega;\mathbf{A}) = \sum_{j=0}^{n} e^{n-j} \omega^{2j}
\end{equation}
Where, $e^j$ is the $j^{th}$ elementary symmetric polynomial with $x_k={\lambda_k}^2$, and can be calculated recursively using the Newton's identities Eq.~\ref{eq:newton1}. We can also express them as Bell polynomials (Eq.~\ref{eq:bell10}),
\begin{equation}
    d(\omega;\mathbf{A}) = \sum_{j=0}^{n} \frac{(-1)^{n-j}}{{(n-j)}!} B(\mathbf{r}^{n-j}(\mathbf{A}^2)) \omega^{2j}
\end{equation}
Here, $\mathbf{r}^j(\mathbf{X})$ is defined in Eq.~\ref{eq:r_j} and $B(\mathbf{r}^{n-j})$ can be computed using Eq.~\ref{eq:bell}. Now we define $d^\alpha(\mathbf{A})$, where $\alpha \in \{0, \ldots, n \}$, to be the coefficient of $\omega^{2\alpha}$ from the polynomial expansion of $d(\omega;\mathbf{A})$, namely,
\begin{equation}
    d^\alpha(\mathbf{A}) =  \frac{(-1)^{n-\alpha}}{(n-\alpha)!} B(\mathbf{r}^{n-\alpha}(\mathbf{A}^2))
\end{equation}

\subsection{$g(\omega; \mathbf{A}, \mathbf{B})$} \label{function:g}
The function $g(\omega; \mathbf{A} \in \R^{n\times n}, \mathbf{B} \in \R^{n-1\times n-1})$ is defined as the product of the following determinants,
\begin{equation}
    g(\omega; \mathbf{A}, \mathbf{B})= 2 \omega \,\overline{\left|\mathbf{A} + \dot{\iota}\omega\mathbf{I} \right|} \left|\mathbf{B} + \dot{\iota}\omega\mathbf{I} \right| \label{eq:g1}
\end{equation}
 If $\lambda_k \in \C$ are the eigenvalues of matrix $\mathbf{A}$ and $\beta_k \in \C$ are the eigenvalues of matrix $\mathbf{B}$, using Eq.~\ref{eq:imp2}, we can write the equation above as,
 \begin{equation}
    g(\omega; \mathbf{A}, \mathbf{B}) = 2 \omega \prod_{j=1}^n (\lambda_j - \dot{\iota}\omega)  \prod_{k=1}^{n-1} (\beta_k + \dot{\iota}\omega)  
\end{equation}
Expanding the polynomial into a sum containing elementary symmetric polynomials (Eq.~\ref{eq:esp}) we find,
\begin{equation}
    g(\omega; \mathbf{A}, \mathbf{B}) = 2\omega \left(\sum_{j=0}^{n} (-1)^j (\dot{\iota})^{j} e_1^{n-j} \omega^{j}  \right) \left(\sum_{k=0}^{n-1} (\dot{\iota})^{k} e_2^{n-k-1} \omega^{k}  \right)   \label{eq:g11}
\end{equation}
where, $e_1^i$ and $e_2^i$ are the elementary symmetric polynomials with $x_k=\lambda_k$ and $x_k=\beta_k$, respectively. Using the property of product of polynomials in Eq.~\ref{eq:pol_prod}, we expand Eq.~\ref{eq:g11} in $\omega$ and find the coefficients of the powers of $\omega$,
\begin{equation}
\begin{split}
    g(\omega; \mathbf{A}, \mathbf{B}) &= 2\omega  \sum_{p=0}^{2n-1} \sum_{j+k=p} (-1)^j (\dot{\iota})^{j+k} e_1^{n-j} e_2^{n-k-1} \omega^{p}   \\
    & = 2\omega \sum_{p=0}^{2n-1} \sum_{j+k=p} (-1)^j (\dot{\iota})^{p} e_1^{n-j} e_2^{n-k-1} \omega^{p}  \label{eq:iota11}
\end{split}
\end{equation}
subject to the condition,
\begin{equation}
    0 \leq j \leq n , \quad 0 \leq k \leq n-1
\end{equation}
We can divide Eq.~\ref{eq:iota11} into real and complex parts by dividing the iterator ($p$) into even and odd components,
\begin{equation}
    g(\omega; \mathbf{A}, \mathbf{B}) = 2\omega  \sum_{p=0}^{n-1} \sum_{j+k=2p}  (-1)^j (\dot{\iota})^{2p} e_1^{n-j} e_2^{n-k-1} \omega^{2p}   + 2\omega  \sum_{p=0}^{n-1} \sum_{j+k=2p+1} (-1)^j (\dot{\iota})^{2p+1} e_1^{n-j} e_2^{n-k-1} \omega^{2p+1} \label{eq:iota2}
\end{equation}
Upon further simplification, we obtain,
\begin{equation}
    g(\omega; \mathbf{A}, \mathbf{B}) = 2\omega  \sum_{p=0}^{n-1} \sum_{j+k=2p}  (-1)^{j+p}  e_1^{n-j} e_2^{n-k-1} \omega^{2p}   + 2\dot{\iota}\omega  \sum_{p=0}^{n-1} \sum_{j+k=2p+1} (-1)^{j+p}  e_1^{n-j} e_2^{n-k-1} \omega^{2p+1} \label{eq:iota2}
\end{equation}
We split $g(\omega; \mathbf{A}, \mathbf{B})$ into its real and imaginary parts as follows,
\begin{equation}
\begin{split}
    g_1(\omega; \mathbf{A}, \mathbf{B}) &= 2 \omega \,\mathrm{Re}\left[ \overline{\left|\mathbf{A} + \dot{\iota}\omega\mathbf{I} \right|} \left|\mathbf{B} + \dot{\iota}\omega\mathbf{I} \right| \right] \\
    &= 2\omega  \sum_{p=0}^{n-1} \sum_{j+k=2p}  (-1)^{j+p}  e_1^{n-j} e_2^{n-k-1} \omega^{2p} 
\end{split}
\end{equation}
\begin{equation}
\begin{split}
    g_2(\omega; \mathbf{A}, \mathbf{B}) &=  2 \omega \,\mathrm{Im}\left[ \overline{\left|\mathbf{A} + \dot{\iota}\omega\mathbf{I} \right|} \left|\mathbf{B} + \dot{\iota}\omega\mathbf{I} \right| \right] \\ &= 2\omega  \sum_{p=0}^{n-1} \sum_{j+k=2p+1} (-1)^{j+p}  e_1^{n-j} e_2^{n-k-1} \omega^{2p+1} \label{eq:g2_gen}
\end{split}
\end{equation}
such that,
\begin{equation}
    g(\omega; \mathbf{A}, \mathbf{B}) = g_1(\omega; \mathbf{A}, \mathbf{B}) + \dot{\iota} g_2(\omega; \mathbf{A}, \mathbf{B}) \label{eq:sum_g}
\end{equation}
Replacing the elementary symmetric polynomials by Bell polynomials (using Eq.~\ref{eq:bell10}) we obtain,
\begin{equation}
\begin{split}
    &g_1(\omega; \mathbf{A}, \mathbf{B}) = 2\, \sum_{p=0}^{n-1} \sum_{j+k=2p} \frac{(-1)^{p-k-1}}{(n-j)!(n-k-1)!} B(\mathbf{r}^{n-j}(\mathbf{A})) B(\mathbf{r}^{n-k-1}(\mathbf{B})) \omega^{2p+1} \\
    &g_2(\omega; \mathbf{A}, \mathbf{B}) = 2\,  \sum_{p=0}^{n-1} \sum_{j+k=2p+1} \frac{(-1)^{p-k-1}}{(n-j)!(n-k-1)!} B(\mathbf{r}^{n-j}(\mathbf{A})) B(\mathbf{r}^{n-k-1}(\mathbf{B})) \omega^{2p+2} 
\end{split}
\end{equation}
Finally, $g_1^\alpha(\mathbf{A}, \mathbf{B})$, for $\alpha \in \{0, \ldots, n-1 \}$, and $g_2^\alpha(\mathbf{A}, \mathbf{B})$, for $\alpha \in \{0, \ldots, n \}$, can be defined as the coefficients of $\omega^{2\alpha+1}$ and $\omega^{2\alpha}$ for $g_1(\omega; \mathbf{A}, \mathbf{B})$ and $g_2(\omega; \mathbf{A}, \mathbf{B})$, respectively,
\begin{equation}
    g_1^\alpha(\mathbf{A}, \mathbf{B}) = 2 \sum_{j+k=2\alpha} \frac{(-1)^{\alpha-j-1}}{(n-j)!(n-k-1)!} B(\mathbf{r}^{n-j}(\mathbf{A})) B(\mathbf{r}^{n-k-1}(\mathbf{B})) 
\end{equation}
\begin{equation}
    g_2^\alpha(\mathbf{A}, \mathbf{B}) = 2 \sum_{j+k=2\alpha-1} \frac{(-1)^{\alpha-j-1}}{(n-j)!(n-k-1)!} B(\mathbf{r}^{n-j}(\mathbf{A})) B(\mathbf{r}^{n-k-1}(\mathbf{B})) 
\end{equation}

\subsection{$h(\omega; \mathbf{A}, \mathbf{B})$} \label{function:h}
The function $h(\omega; \mathbf{A} \in \R^{n\times n}, \mathbf{B} \in \R^{n\times n})$ is defined as the product of the following determinants,
\begin{equation}
    h(\omega; \mathbf{A}, \mathbf{B})= 2 \,\overline{\left|\mathbf{A} + \dot{\iota}\omega\mathbf{I} \right|} \left|\mathbf{B} + \dot{\iota}\omega\mathbf{I} \right| \label{eq:g1}
\end{equation}
 If $\lambda_k \in \C$ are the eigenvalues of matrix $\mathbf{A}$ and $\beta_k \in \C$ are the eigenvalues of matrix $\mathbf{B}$, using Eq.~\ref{eq:imp2}, we can write the equation above as,
 \begin{equation}
    h(\omega; \mathbf{A}, \mathbf{B}) = 2 \prod_{j=1}^n (\lambda_j - \dot{\iota}\omega)  \prod_{k=1}^{n} (\beta_k + \dot{\iota}\omega)  
\end{equation}
Expanding the polynomial into a sum containing elementary symmetric polynomials (Eq.~\ref{eq:esp}) we find,
\begin{equation}
    h(\omega; \mathbf{A}, \mathbf{B}) = 2 \left(\sum_{j=0}^{n} (-1)^j (\dot{\iota})^{j} e_1^{n-j} \omega^{j}  \right) \left(\sum_{k=0}^{n} (\dot{\iota})^{k} e_2^{n-k} \omega^{k}  \right)  \label{eq:h11} 
\end{equation}
where, $e_1^i$ and $e_2^i$ are the elementary symmetric polynomials with $x_k=\lambda_k$ and $x_k=\beta_k$, respectively. Using the property of product of polynomials in Eq.~\ref{eq:pol_prod}, we expand Eq.~\ref{eq:h11} in $\omega$ and find the coefficients of the powers of $\omega$,
\begin{equation}
    h(\omega; \mathbf{A}, \mathbf{B}) = 2 \sum_{p=0}^{2n} \sum_{j+k=p} (-1)^j (\dot{\iota})^{j+k} e_1^{n-j} e_2^{n-k} \omega^{p}    = 2 \sum_{p=0}^{2n} \sum_{j+k=p} (-1)^j (\dot{\iota})^{p} e_1^{n-j} e_2^{n-k-1} \omega^{p}  \label{eq:iota1}
\end{equation}
subject to the condition,
\begin{equation}
    0 \leq j \leq n , \quad 0 \leq k \leq n
\end{equation}
We can divide Eq.~\ref{eq:iota1} into real and complex parts by dividing the iterator ($p$) into even and odd components,
\begin{equation}
    h(\omega; \mathbf{A}, \mathbf{B}) = 2 \sum_{p=0}^{n} \sum_{j+k=2p}  (-1)^j (\dot{\iota})^{2p} e_1^{n-j} e_2^{n-k} \omega^{2p}  + 2 \sum_{p=0}^{n-1} \sum_{j+k=2p+1} (-1)^j (\dot{\iota})^{2p+1} e_1^{n-j} e_2^{n-k} \omega^{2p+1} \label{eq:iota2}
\end{equation}
Upon further simplification, we obtain,
\begin{equation}
    h(\omega; \mathbf{A}, \mathbf{B}) = 2 \sum_{p=0}^{n} \sum_{j+k=2p}  (-1)^{j+p}  e_1^{n-j} e_2^{n-k} \omega^{2p} + 2\dot{\iota}  \sum_{p=0}^{n-1} \sum_{j+k=2p+1} (-1)^{j+p}  e_1^{n-j} e_2^{n-k} \omega^{2p+1} \label{eq:iota2}
\end{equation}
We split $h(\omega; \mathbf{A}, \mathbf{B})$ into its real and imaginary parts as follows,
\begin{equation}
\begin{split}
    h_1(\omega; \mathbf{A}, \mathbf{B}) &= 2 \,\mathrm{Re}\left[ \overline{\left|\mathbf{A} + \dot{\iota}\omega\mathbf{I} \right|} \left|\mathbf{B} + \dot{\iota}\omega\mathbf{I} \right| \right] \\
    &= 2 \sum_{p=0}^{n} \sum_{j+k=2p}  (-1)^{j+p}  e_1^{n-j} e_2^{n-k} \omega^{2p} 
\end{split}
\end{equation}
\begin{equation}
\begin{split}
    h_2(\omega; \mathbf{A}, \mathbf{B}) &=  2 \,\mathrm{Im}\left[ \overline{\left|\mathbf{A} + \dot{\iota}\omega\mathbf{I} \right|} \left|\mathbf{B} + \dot{\iota}\omega\mathbf{I} \right| \right] \\ 
    &= 2  \sum_{p=0}^{n-1} \sum_{j+k=2p+1} (-1)^{j+p}  e_1^{n-j} e_2^{n-k} \omega^{2p+1}
\end{split}
\end{equation}
such that,
\begin{equation}
    h(\omega; \mathbf{A}, \mathbf{B}) = h_1(\omega; \mathbf{A}, \mathbf{B}) + \dot{\iota} h_2(\omega; \mathbf{A}, \mathbf{B}) \label{eq:sum_h}
\end{equation}
Replacing the elementary symmetric polynomials by Bell polynomials (using Eq.~\ref{eq:bell10}) we find,
\begin{equation}
\begin{split}
    &h_1(\omega; \mathbf{A}, \mathbf{B}) = 2\, \sum_{p=0}^{n} \sum_{j+k=2p} \frac{(-1)^{p-k}}{(n-j)!(n-k)!} B(\mathbf{r}^{n-j}(\mathbf{A})) B(\mathbf{r}^{n-k}(\mathbf{B})) \omega^{2p} \\
    &h_2(\omega; \mathbf{A}, \mathbf{B}) = 2\,  \sum_{p=0}^{n-1} \sum_{j+k=2p+1} \frac{(-1)^{p-k}}{(n-j)!(n-k)!} B(\mathbf{r}^{n-j}(\mathbf{A})) B(\mathbf{r}^{n-k}(\mathbf{B})) \omega^{2p+1} 
\end{split}
\end{equation}
Finally, $h_1^\alpha(\mathbf{A}, \mathbf{B})$, for $\alpha \in \{0, \ldots, n \}$, and $h_2^\alpha(\mathbf{A}, \mathbf{B})$, for $\alpha \in \{0, \ldots, n-1 \}$, can be defined as the coefficients of $\omega^{2\alpha}$ and $\omega^{2\alpha+1}$ for $h_1(\omega; \mathbf{A}, \mathbf{B})$ and $h_2(\omega; \mathbf{A}, \mathbf{B})$, respectively,
\begin{equation}
    h_1^\alpha(\mathbf{A}, \mathbf{B}) = 2 \sum_{j+k=2\alpha} \frac{(-1)^{\alpha-k}}{(n-j)!(n-k)!} B(\mathbf{r}^{n-j}(\mathbf{A})) B(\mathbf{r}^{n-k}(\mathbf{B})) 
\end{equation}
\begin{equation}
    h_2^\alpha(\mathbf{A}, \mathbf{B}) = 2 \sum_{j+k=2\alpha+1} \frac{(-1)^{\alpha-k}}{(n-j)!(n-k)!} B(\mathbf{r}^{n-j}(\mathbf{A})) B(\mathbf{r}^{n-k}(\mathbf{B})) 
\end{equation}

\subsection{$f(\omega; \mathbf{A}, \mathbf{e}_{\beta})$} \label{function:f}
The function $f(\omega; \mathbf{A} \in \R^{n\times n}, \mathbf{e}_{\beta} \in \{0,1\}^{n \times 1})$ is defined as the expansion of the following determinant,
\begin{equation}
    f(\omega; \mathbf{A}, \mathbf{e}_{\beta}) = |\mathbf{A} + \dot{\iota}\omega\mathbf{I} - \dot{\iota}\omega\, \mathbf{e}_{\beta} {\mathbf{e}_{\beta}}^\top| \, |\mathbf{A} - \dot{\iota}\omega\mathbf{I} + \dot{\iota}\omega\, \mathbf{e}_{\beta} {\mathbf{e}_{\beta}}^\top|  \label{eq:f_new1}
\end{equation}
The elements of the vector $\mathbf{e}_{\beta}$ is the standard unit vector equal to $1$ only at the index $\beta$ and zero everywhere else. Eq.~\ref{eq:f_new1} is just the modulus-squared of one of the determinants, therefore,
\begin{equation}
\begin{split}
    f(\omega; \mathbf{A}, \mathbf{e}_{\beta}) &= ||\mathbf{A} + \dot{\iota}\omega\, \mathbf{I} - \dot{\iota}\omega\, \mathbf{e}_{\beta} {\mathbf{e}_{\beta}}^\top||^2  \label{eq:f_new2}
  \end{split}
\end{equation}
We use the property from Eq.~\ref{eq:det_prop}, to write Eq.~\ref{eq:f_new2} as,
\begin{equation}
\begin{split}
    f(\omega; \mathbf{A}, \mathbf{e}_{\beta}) = \left| \left|\mathbf{A} + \dot{\iota}\omega\, \mathbf{I} \right| - \dot{\iota}\omega\,  {\mathbf{e}_{\beta}}^\top
  \mathrm{adj}(\mathbf{A} + \dot{\iota}\omega\mathbf{I})
  \mathbf{e}_{\beta} \right| \label{eq:f_new3}
\end{split}
\end{equation}
Since only one of the entries in the vector $\mathbf{e}_{\beta}$ is 1, we can simplify $f(\omega; \mathbf{A}, \mathbf{e}_{\beta})$ by defining a submatrix ($\mathbf{B} \in \R^{n-1 \times n-1}$) of $\mathbf{A}$, by excluding the row and column of $\mathbf{A}$ at the index $\beta$. Eq.~\ref{eq:f_new3} then becomes,
\begin{equation}
    f(\omega; \mathbf{A}, \mathbf{e}_{\beta}) = \left| \left|\mathbf{A} + \dot{\iota}\omega\mathbf{I} \right| - \dot{\iota}\omega \left| \mathbf{B} + \dot{\iota}\omega \mathbf{I} \right| \right|^2
\end{equation}
Expanding the complex expression, we get,
\begin{equation}
    f(\omega; \mathbf{A}, \mathbf{e}_{\beta}) =  \left||\mathbf{A} + \dot{\iota}\omega\mathbf{I}| \right|^2 
    + \omega^2 \left||\mathbf{B} + \dot{\iota}\omega\mathbf{I} |\right|^2 
    - \dot{\iota}\omega \left[ \overline{\left|\mathbf{A} + \dot{\iota}\omega\mathbf{I} \right|} \left|\mathbf{B} + \dot{\iota}\omega\mathbf{I} \right| - \left|\mathbf{A} + \dot{\iota}\omega\mathbf{I} \right| \overline{\left|\mathbf{B} + \dot{\iota}\omega\mathbf{I} \right|} \right] \label{eq:f_10}
\end{equation}
We can also write Eq.~\ref{eq:f_10} as,
\begin{equation}
    f(\omega; \mathbf{A}, \mathbf{e}_{\beta}) =  \left||\mathbf{A} + \dot{\iota}\omega\mathbf{I}| \right|^2 
    + \omega^2 \left||\mathbf{B} + \dot{\iota}\omega\mathbf{I} |\right|^2 
    + 2\, \omega \,\mathrm{Imag}\left[ \overline{\left|\mathbf{A} + \dot{\iota}\omega\mathbf{I} \right|} \left|\mathbf{B} + \dot{\iota}\omega\mathbf{I} \right| \right]
\end{equation}
Using the definition of the functions, $d(\omega; \mathbf{A})$ (Eq.~\ref{eq:d_gen1}) and $g_2(\omega; \mathbf{A}, \mathbf{B})$ (Eq.~\ref{eq:g2_gen}), we can write,
\begin{equation}
\begin{split}
    f(\omega; \mathbf{A}, \mathbf{e}_{\beta}) =  d(\omega; \mathbf{A}) + \omega^2 d(\omega; \mathbf{B}) + g_2(\omega; \mathbf{A}, \mathbf{B})
\end{split}
\end{equation}
Now we denote $f^\alpha(\mathbf{A}, \mathbf{e}_{\beta})$, for $\alpha \in \{0, \ldots, n \}$, to be the coefficient of $\omega^{2\alpha}$ from the polynomial expansion of $f(\omega; \mathbf{A}, \mathbf{e}_{\beta})$,
\begin{equation}
    f^\alpha(\mathbf{A}, \mathbf{e}_{\beta}) =  d^\alpha(\mathbf{A}) + d^{\alpha-1}(\mathbf{B}) + g_2^\alpha(\mathbf{A}, \mathbf{B})
\end{equation}

\subsection{$s(\omega; \mathbf{A}, \mathbf{B}, \mathbf{e}_{\beta})$} \label{function:s}
The function $s(\omega; \mathbf{A} \in \R^{n\times n}, \mathbf{B}\in \R^{n\times n}, \mathbf{e}_{\beta}\in \{0,1\}^{n\times 1})$ is defined as follows,
\begin{equation}
    s(\omega; \mathbf{A}, \mathbf{B}, \mathbf{e}_{\beta}) = 2\, |\mathbf{A} + \dot{\iota}\omega\mathbf{I}| \, \overline{|\mathbf{B} + \dot{\iota}\omega\mathbf{I} - \dot{\iota}\omega\, \mathbf{e}_{\beta} {\mathbf{e}_{\beta}}^\top|}
\end{equation}
Since all the matrices and vectors are purely real, we can write,
\begin{equation}
    s(\omega; \mathbf{A}, \mathbf{B}, \mathbf{e}_{\beta}) = 2\, |\mathbf{A} + \dot{\iota}\omega\mathbf{I}| \, |\mathbf{B} - \dot{\iota}\omega\mathbf{I} + \dot{\iota}\omega\, \mathbf{e}_{\beta} {\mathbf{e}_{\beta}}^\top| \label{eq:s1}
\end{equation}
We use the property from Eq.~\ref{eq:det_prop}, to write Eq.~\ref{eq:s1} as,
\begin{equation}
    s(\omega; \mathbf{A}, \mathbf{B}, \mathbf{e}_{\beta}) = 2\, |\mathbf{A} + \dot{\iota}\omega\mathbf{I}| \, \bigl(|\mathbf{B} - \dot{\iota}\omega\mathbf{I}| + \dot{\iota}\omega\, {\mathbf{e}_{\beta}}^\top \mathrm{adj}(\mathbf{B} - \dot{\iota}\omega\mathbf{I}) \mathbf{e}_{\beta} \bigr)\label{eq:s2}
\end{equation}
Since only one of the entries in the vector $\mathbf{e}_{\beta}$ is 1, and the index is the same for both vectors (row or column), we can simplify $s(\omega; \mathbf{A}, \mathbf{B}, \mathbf{e}_{\beta})$ by defining a submatrix ($\mathbf{C} \in \R^{n-1 \times n-1}$) of $\mathbf{B}$, by excluding the row and column of $\mathbf{B}$ at the index $\beta$. Eq.~\ref{eq:s2} then becomes,
\begin{equation}
    s(\omega; \mathbf{A}, \mathbf{B}, \mathbf{e}_{\beta}) = 2\, |\mathbf{A} + \dot{\iota}\omega\mathbf{I}| \, 
    \bigr( |\mathbf{B} - \dot{\iota}\omega\mathbf{I}| + \dot{\iota} \omega |\mathbf{C} - \dot{\iota}\omega\mathbf{I}| \bigr) \label{eq:s3}
\end{equation}
On expanding the product, we get,
\begin{equation}
\begin{split}
    s(\omega; \mathbf{A}, \mathbf{B}, \mathbf{e}_{\beta}) &= 2\,|\mathbf{A} + \dot{\iota}\omega\mathbf{I}| \, 
    |\mathbf{B} - \dot{\iota}\omega\mathbf{I}| + \dot{\iota} \omega |\mathbf{A} + \dot{\iota}\omega\mathbf{I}||\mathbf{C} - \dot{\iota}\omega\mathbf{I}|\\
    & = 2\, |\mathbf{A} + \dot{\iota}\omega\mathbf{I}| \, 
    |\mathbf{B} - \dot{\iota}\omega\mathbf{I}|  + 2\,\dot{\iota} \omega |\mathbf{A} + \dot{\iota}\omega\mathbf{I}||\mathbf{C} - \dot{\iota}\omega\mathbf{I}|\\
    & = 2\, |\mathbf{A} + \dot{\iota}\omega\mathbf{I}| \, 
    \overline{|\mathbf{B} + \dot{\iota}\omega\mathbf{I}|} + 2\,\dot{\iota} \omega |\mathbf{A} + \dot{\iota}\omega\mathbf{I}||\mathbf{C} - \dot{\iota}\omega\mathbf{I}| \label{eq:f1}
\end{split}
\end{equation}
We split $s(\omega; \mathbf{A}, \mathbf{B}, \mathbf{e}_{\beta})$ into its real and imaginary parts as follows,
\begin{equation}
\begin{split}
    s_1(\omega; \mathbf{A}, \mathbf{B}, \mathbf{e}_{\beta}) &= 2\,\mathrm{Re}\left[|\mathbf{A} + \dot{\iota}\omega\mathbf{I}| \, \overline{|\mathbf{B} + \dot{\iota}\omega\mathbf{I} - \dot{\iota}\omega\, \mathbf{e}_{\beta} {\mathbf{e}_{\beta}}^\top|}\right]\\
    & = 2\,\mathrm{Re}\left[|\mathbf{A} + \dot{\iota}\omega\mathbf{I}| \, 
    \overline{|\mathbf{B} + \dot{\iota}\omega\mathbf{I}|} \right] + 2\,\mathrm{Re}\left[\dot{\iota} \omega |\mathbf{A} + \dot{\iota}\omega\mathbf{I}||\mathbf{C} - \dot{\iota}\omega\mathbf{I}|
    \right] \\ 
    & = 2\,\mathrm{Re}\left[ \overline{|\mathbf{A} + \dot{\iota}\omega\mathbf{I}|} \, 
    |\mathbf{B} + \dot{\iota}\omega\mathbf{I}| \right] - 2\,\omega \mathrm{Im}\left[|\mathbf{A} + \dot{\iota}\omega\mathbf{I}||\mathbf{C} - \dot{\iota}\omega\mathbf{I}|
    \right] \\ 
    & = 2\,\mathrm{Re}\left[ \overline{|\mathbf{A} + \dot{\iota}\omega\mathbf{I}|} \, 
    |\mathbf{B} + \dot{\iota}\omega\mathbf{I}| \right] + 2\,\omega \mathrm{Im}\left[\overline{ |\mathbf{A} + \dot{\iota}\omega\mathbf{I}|}|\mathbf{C} + \dot{\iota}\omega\mathbf{I}|
    \right] 
\end{split}
\end{equation}
\begin{equation}
\begin{split}
    s_2(\omega; \mathbf{A}, \mathbf{B}, \mathbf{e}_{\beta}) &= 2\,\mathrm{Im}\left[|\mathbf{A} + \dot{\iota}\omega\mathbf{I}| \, \overline{|\mathbf{B} + \dot{\iota}\omega\mathbf{I} - \dot{\iota}\omega\, \mathbf{e}_{\beta} {\mathbf{e}_{\beta}}^\top|}\right]\\
    & = 2\,\mathrm{Im}\left[ |\mathbf{A} + \dot{\iota}\omega\mathbf{I}| \, 
    \overline{|\mathbf{B} + \dot{\iota}\omega\mathbf{I}|} \right] + 2\,\mathrm{Im}\left[\dot{\iota} \omega |\mathbf{A} + \dot{\iota}\omega\mathbf{I}||\mathbf{C} - \dot{\iota}\omega\mathbf{I}|
    \right] \\ 
    & = 2\,\mathrm{Im}\left[ |\mathbf{A} + \dot{\iota}\omega\mathbf{I}| \, 
    \overline{|\mathbf{B} + \dot{\iota}\omega\mathbf{I}|} \right] + 2\,\omega \mathrm{Re}\left[\overline{|\mathbf{A} + \dot{\iota}\omega\mathbf{I}|} |\mathbf{C} + \dot{\iota}\omega\mathbf{I}|
    \right]
\end{split}
\end{equation}
so that,
\begin{equation}
\begin{split}
    s(\omega; \mathbf{A}, \mathbf{B}, \mathbf{e}_{\beta}) = s_1(\omega; \mathbf{A}, \mathbf{B}, \mathbf{e}_{\beta}) + \dot{\iota} s_2(\omega; \mathbf{A}, \mathbf{B}, \mathbf{e}_{\beta})
\end{split}
\end{equation}
Using the definitions of functions $g(\omega; \mathbf{A},\mathbf{C})$ (Eq.~\ref{eq:sum_g}) and $h(\omega; \mathbf{A},\mathbf{B})$ (Eq.~\ref{eq:sum_h}), we can write,
\begin{equation}
\begin{split}
    & s_1(\omega; \mathbf{A}, \mathbf{B}, \mathbf{e}_{\beta}) = h_1(\omega; \mathbf{A},\mathbf{B}) + g_2(\omega; \mathbf{A},\mathbf{C}) \\
    & s_2(\omega; \mathbf{A}, \mathbf{B}, \mathbf{e}_{\beta}) = -h_2(\omega; \mathbf{A},\mathbf{B}) + g_1(\omega; \mathbf{A},\mathbf{C})
\end{split}
\end{equation}
Finally, $s_1^\alpha(\mathbf{A}, \mathbf{B}, \mathbf{e}_{\beta})$, for $\alpha \in \{0, \ldots, n \}$, and $s_2^\alpha(\mathbf{A}, \mathbf{B}, \mathbf{e}_{\beta})$, for $\alpha \in \{0, \ldots, n -1\}$, are defined to be the coefficients of $\omega^{2\alpha}$ and $\omega^{2\alpha+1}$ from the polynomial expansion of $s_1(\omega; \mathbf{A}, \mathbf{B}, \mathbf{e}_{\beta})$ and $s_2(\omega; \mathbf{A}, \mathbf{B},  \mathbf{e}_{\beta})$, respectively.
\begin{equation}
\begin{split}
    &s_1^\alpha(\mathbf{A}, \mathbf{B}, \mathbf{e}_{\beta}) = h_1^\alpha(\mathbf{A},\mathbf{B}) + g_2^\alpha(\mathbf{A},\mathbf{C}) \\
    & s_2^\alpha(\mathbf{A}, \mathbf{B}, \mathbf{e}_{\beta}) = -h_2^\alpha(\mathbf{A},\mathbf{B}) + g_1^\alpha(\mathbf{A},\mathbf{C})
\end{split}
\end{equation}

\subsection{$t(\omega; \mathbf{A}, \mathbf{B}, \mathbf{e}_{\beta_1}, \mathbf{e}_{\beta_2})$} \label{function:t}
The function $t(\omega; \mathbf{A} \in \R^{n \times n}, \mathbf{B} \in \R^{n \times n}, \mathbf{e}_{\beta_1} \in \{0,1\}^{n \times 1}, \mathbf{e}_{\beta_2} \in \{0,1\}^{n \times 1})$ is defined as follows,
\begin{equation}
   t(\omega; \mathbf{A}, \mathbf{B}, \mathbf{e}_{\beta_1}, \mathbf{e}_{\beta_2}) = 2\, |\mathbf{A} + \dot{\iota}\omega\mathbf{I} - \dot{\iota}\omega\, \mathbf{e}_{\beta_1} {\mathbf{e}_{\beta_1}}^\top| \, \overline{|\mathbf{B} + \dot{\iota}\omega\mathbf{I} - \dot{\iota}\omega\, \mathbf{e}_{\beta_2} {\mathbf{e}_{\beta_2}}^\top|}
\end{equation}
Since all the matrices and vectors are purely real, we can write,
\begin{equation}
   t(\omega; \mathbf{A}, \mathbf{B}, \mathbf{e}_{\beta_1}, \mathbf{e}_{\beta_2}) = 2\, |\mathbf{A} + \dot{\iota}\omega\mathbf{I} - \dot{\iota}\omega\, \mathbf{e}_{\beta_2} {\mathbf{e}_{\beta_2}}^\top| \, |\mathbf{B} - \dot{\iota}\omega\mathbf{I} + \dot{\iota}\omega\, \mathbf{e}_{\beta_2} {\mathbf{e}_{\beta_2}}^\top| \label{eq:t1}
\end{equation}
We use the property from Eq.~\ref{eq:det_prop}, to write Eq.~\ref{eq:t1} as,
\begin{equation}
    t(\omega; \mathbf{A}, \mathbf{B}, \mathbf{e}_{\beta_1}, \mathbf{e}_{\beta_2}) = 2\,\bigl( |\mathbf{A} + \dot{\iota}\omega\mathbf{I}| -\dot{\iota}\omega {\mathbf{e}_{\beta_1}}^\top \mathrm{adj}(\mathbf{A} + \dot{\iota}\omega\mathbf{I}) \mathbf{e}_{\beta_1} \bigr)\bigl(|\mathbf{B} - \dot{\iota}\omega\mathbf{I}| + \dot{\iota}\omega\, {\mathbf{e}_{\beta_2}}^\top \mathrm{adj}(\mathbf{B} - \dot{\iota}\omega\mathbf{I}) \mathbf{e}_{\beta_2} \bigr)\label{eq:t2}
\end{equation}
Since only one of the entries in the vectors $\mathbf{e}_{\beta_1}$ and $\mathbf{e}_{\beta_2}$ are 1, and the index is the same for both vectors (row or column), we can simplify $t(\omega; \mathbf{A}, \mathbf{B}, \mathbf{e}_{\beta_1}, \mathbf{e}_{\beta_2})$ by defining submatrices $\mathbf{C} \in \R^{n-1 \times n-1}$ and $\mathbf{D} \in \R^{n-1 \times n-1}$, by excluding the $\beta_1$ and $\beta_2$ row and column of the matrices $\mathbf{A}$ and $\mathbf{B}$, respectively. Eq.~\ref{eq:t2} then becomes,
\begin{equation}
    t(\omega; \mathbf{A}, \mathbf{B}, \mathbf{e}_{\beta_1}, \mathbf{e}_{\beta_2}) = 2\,\bigl(|\mathbf{A} + \dot{\iota}\omega\mathbf{I}| -\dot{\iota}\omega |\mathbf{C} + \dot{\iota}\omega\mathbf{I}|\bigr)
    \bigl( |\mathbf{B} - \dot{\iota}\omega\mathbf{I}| + \dot{\iota} \omega |\mathbf{D} - \dot{\iota}\omega\mathbf{I}| \bigr) \label{eq:t3}
\end{equation}
On expanding the product, we get,
\begin{equation}
    t(\omega; \mathbf{A}, \mathbf{B}, \mathbf{e}_{\beta_1}, \mathbf{e}_{\beta_2}) = 2\,
    |\mathbf{A} + \dot{\iota}\omega\mathbf{I}| |\mathbf{B} - \dot{\iota}\omega\mathbf{I}| 
    + 2\, \omega^2 |\mathbf{C} + \dot{\iota}\omega\mathbf{I}| |\mathbf{D} - \dot{\iota}\omega\mathbf{I}|  -2\, \dot{\iota}\omega |\mathbf{C} + \dot{\iota}\omega\mathbf{I}| |\mathbf{B} - \dot{\iota}\omega\mathbf{I}|
    + 2\,\dot{\iota} \omega |\mathbf{A} + \dot{\iota}\omega\mathbf{I}| |\mathbf{D} - \dot{\iota}\omega\mathbf{I}|\bigr] \label{eq:t10}
\end{equation}
We split $t(\omega; \mathbf{A}, \mathbf{B}, \mathbf{e}_{\beta_1}, \mathbf{e}_{\beta_2})$ into its real and imaginary parts as follows,
\begin{equation}
\begin{split}
    t_1(\omega; \mathbf{A}, \mathbf{B}, \mathbf{e}_{\beta_1}, \mathbf{e}_{\beta_2}) &= \, 2\, \mathrm{Re}\left[ |\mathbf{A} + \dot{\iota}\omega\mathbf{I} - \dot{\iota}\omega\, \mathbf{e}_{\beta_1} {\mathbf{e}_{\beta_1}}^\top| \, \overline{|\mathbf{B} + \dot{\iota}\omega\mathbf{I} - \dot{\iota}\omega\, \mathbf{e}_{\beta_2} {\mathbf{e}_{\beta_2}}^\top|}\right] \\
    &= \, 2\,\mathrm{Re}\bigl[ 
    |\mathbf{A} + \dot{\iota}\omega\mathbf{I}| |\mathbf{B} - \dot{\iota}\omega\mathbf{I}| \bigr]
    + 2\,\mathrm{Re}\bigl[ \omega^2 |\mathbf{C} + \dot{\iota}\omega\mathbf{I}| |\mathbf{D} - \dot{\iota}\omega\mathbf{I}|\bigr] \\
    & \quad + 2\,\mathrm{Re}\bigl[-\dot{\iota}\omega |\mathbf{C} + \dot{\iota}\omega\mathbf{I}| |\mathbf{B} - \dot{\iota}\omega\mathbf{I}|\bigr]
    + 2\,\mathrm{Re}\bigl[\dot{\iota} \omega |\mathbf{A} + \dot{\iota}\omega\mathbf{I}| |\mathbf{D} - \dot{\iota}\omega\mathbf{I}|\bigr]
    \bigr]\\
    &= \, 2\,\mathrm{Re}\bigl[ 
    \overline{|\mathbf{A} + \dot{\iota}\omega\mathbf{I}|} |\mathbf{B} + \dot{\iota}\omega\mathbf{I}| \bigr]
    + 2\,\omega^2 \mathrm{Re}\bigl[ \overline{|\mathbf{C} + \dot{\iota}\omega\mathbf{I}|} |\mathbf{D} + \dot{\iota}\omega\mathbf{I}|\bigr] \\
    & \quad + 2\,\omega \mathrm{Im}\bigl[\overline{|\mathbf{B} + \dot{\iota}\omega\mathbf{I}|} |\mathbf{C} + \dot{\iota}\omega\mathbf{I}|\bigr]
    + 2\,\omega \mathrm{Im}\bigl[\overline{|\mathbf{A} + \dot{\iota}\omega\mathbf{I}|} |\mathbf{D} + \dot{\iota}\omega\mathbf{I}|\bigr]
    \bigr]
\end{split}
\end{equation}
\begin{equation}
\begin{split}
    t_2(\omega; \mathbf{A}, \mathbf{B}, \mathbf{e}_{\beta_1}, \mathbf{e}_{\beta_2}) &= \, 2\, \mathrm{Im}\left[ |\mathbf{A} + \dot{\iota}\omega\mathbf{I} - \dot{\iota}\omega\, \mathbf{e}_{\beta_1} {\mathbf{e}_{\beta_1}}^\top| \, \overline{|\mathbf{B} + \dot{\iota}\omega\mathbf{I} - \dot{\iota}\omega\, \mathbf{e}_{\beta_2} {\mathbf{e}_{\beta_2}}^\top|}\right] \\
    &= \, 2\,\mathrm{Im}\bigl[ 
    |\mathbf{A} + \dot{\iota}\omega\mathbf{I}| |\mathbf{B} - \dot{\iota}\omega\mathbf{I}| \bigr]
    + 2\,\mathrm{Im}\bigl[ \omega^2 |\mathbf{C} + \dot{\iota}\omega\mathbf{I}| |\mathbf{D} - \dot{\iota}\omega\mathbf{I}|\bigr] \\
    & \quad + 2\,\mathrm{Im}\bigl[-\dot{\iota}\omega |\mathbf{C} + \dot{\iota}\omega\mathbf{I}| |\mathbf{B} - \dot{\iota}\omega\mathbf{I}|\bigr]
    + 2\,\mathrm{Im}\bigl[\dot{\iota} \omega |\mathbf{A} + \dot{\iota}\omega\mathbf{I}| |\mathbf{D} - \dot{\iota}\omega\mathbf{I}|\bigr]
    \bigr]\\
    &= \, -2\,\mathrm{Im}\bigl[ 
    \overline{|\mathbf{A} + \dot{\iota}\omega\mathbf{I}|} |\mathbf{B} + \dot{\iota}\omega\mathbf{I}| \bigr]
    - 2\,\omega^2 \mathrm{Im}\bigl[ \overline{|\mathbf{C} + \dot{\iota}\omega\mathbf{I}|} |\mathbf{D} + \dot{\iota}\omega\mathbf{I}|\bigr] \\
    & \quad - 2\,\omega \mathrm{Re}\bigl[\overline{|\mathbf{B} + \dot{\iota}\omega\mathbf{I}|} |\mathbf{C} + \dot{\iota}\omega\mathbf{I}|\bigr]
    + 2\,\omega \mathrm{Re}\bigl[\overline{|\mathbf{A} + \dot{\iota}\omega\mathbf{I}|} |\mathbf{D} + \dot{\iota}\omega\mathbf{I}|\bigr]
    \bigr]
\end{split}
\end{equation}
so that,
\begin{equation}
\begin{split}
    t(\omega; \mathbf{A}, \mathbf{B}, \mathbf{e}_{\beta_1}, \mathbf{e}_{\beta_2}) = t_1(\omega; \mathbf{A}, \mathbf{B}, \mathbf{e}_{\beta_1}, \mathbf{e}_{\beta_2}) + \dot{\iota} t_2(\omega; \mathbf{A}, \mathbf{B}, \mathbf{e}_{\beta_1}, \mathbf{e}_{\beta_2})
\end{split}
\end{equation}
Using the definitions of functions $g(\omega; \mathbf{A},\mathbf{C})$ (Eq.~\ref{eq:sum_g}) and $h(\omega; \mathbf{A},\mathbf{B})$ (Eq.~\ref{eq:sum_h}), we can write,
\begin{equation}
\begin{split}
    & t_1(\omega; \mathbf{A}, \mathbf{B}, \mathbf{e}_{\beta_1}, \mathbf{e}_{\beta_2}) = h_1(\omega; \mathbf{A},\mathbf{B}) + \omega^2 h_1(\omega; \mathbf{C},\mathbf{D}) + g_2(\omega; \mathbf{B}, \mathbf{C}) + g_2(\omega; \mathbf{A}, \mathbf{D})\\
    & t_2(\omega; \mathbf{A}, \mathbf{B}, \mathbf{e}_{\beta_1}, \mathbf{e}_{\beta_2}) = -h_2(\omega; \mathbf{A},\mathbf{B}) -\omega^2 h_2 (\omega; \mathbf{C},\mathbf{D}) 
    - g_1(\omega; \mathbf{B},\mathbf{C}) + g_1(\omega; \mathbf{A},\mathbf{D})
\end{split}
\end{equation}
Finally, $t_1^\alpha(\mathbf{A}, \mathbf{B}, \mathbf{e}_{\beta_1}, \mathbf{e}_{\beta_2})$, for $\alpha \in \{0, \ldots, n \}$, and $t_2(\mathbf{A}, \mathbf{B}, \mathbf{e}_{\beta_1}, \mathbf{e}_{\beta_2})$, for $\alpha \in \{0, \ldots, n - 1\}$, are defined to be the coefficients of $\omega^{2\alpha}$ and $\omega^{2\alpha+1}$ from the polynomial expansion of $t_1(\omega; \mathbf{A}, \mathbf{B}, \mathbf{e}_{\beta_1}, \mathbf{e}_{\beta_2})$ and \\$t_2(\omega; \mathbf{A}, \mathbf{B}, \mathbf{e}_{\beta_1}, \mathbf{e}_{\beta_2})$, respectively.
\begin{equation}
\begin{split}
    &t_1^\alpha(\mathbf{A}, \mathbf{B}, \mathbf{e}_{\beta_1}, \mathbf{e}_{\beta_2}) = 
    h_1^\alpha(\mathbf{A},\mathbf{B}) + h_1^{\alpha-1}(\mathbf{C},\mathbf{D}) + g_2^\alpha(\mathbf{B}, \mathbf{C}) + g_2^\alpha(\mathbf{A}, \mathbf{D}) \\
    & t_2^\alpha(\mathbf{A}, \mathbf{B}, \mathbf{e}_{\beta_1}, \mathbf{e}_{\beta_2}) = -h_2^\alpha(\mathbf{A},\mathbf{B}) - h_2^{\alpha-1}(\mathbf{C},\mathbf{D}) 
    - g_1^\alpha(\mathbf{B},\mathbf{C}) + g_1^\alpha(\mathbf{A},\mathbf{D})
\end{split}
\end{equation}

\newpage
\appendix
\counterwithin{equation}{section}

\section{Basic identities} \label{appendix}
In this section, we outline basic matrix identities and definitions that we utilize.
\begin{itemize}
    \item If $\lambda$ are the eigenvalues of matrix $\mathbf{A} \in \R^{n\times n}$, then $\lambda^k$ are the eigenvalues of $\mathbf{A}^k$, where $k\in \mathbb{N}$.
    
    \item If $\lambda$ are the eigenvalues of matrix $\mathbf{A} \in \R^{n\times n}$, then $\lambda + \alpha$, where $\alpha \in \C$, are the eigenvalues of the matrix $\mathbf{A} + \alpha \mathbf{I}$.
    
    \item If $\lambda_i$ are the eigenvalues of matrix $\mathbf{A} \in \R^{n\times n}$, then we have,
    \begin{equation}
    \mathrm{det}(\mathbf{A})=\prod_{k=1}^{n} \lambda_k \label{eq:det_prod}
    \end{equation}
    \begin{equation}
    \Tr(\mathbf{A}) = \sum_{k=1}^n\lambda_k \label{eq:trace_sum}
    \end{equation}
    
    \item Properties of the adjugate of a matrix $\mathbf{A} \in \R^{n\times n}$,
    \begin{equation}
    \mathrm{adj}(\mathbf{A})^\top = \mathrm{adj}(\mathbf{A}^\top)
    \end{equation}
    \begin{equation}
    \mathrm{adj}(\overline{\mathbf{A} }) = \overline{\mathrm{adj}(\mathbf{A})}
    \end{equation}
    
    \item Laplace transform of a matrix exponential amounts to the resolvent of the matrix ($\mathbf{A} \in \R^{n\times n}$),
    \begin{equation}
        \mathcal{L}(e^{t\mathbf{A}}) = (s\I - \mathbf{A})^{-1} \label{eq:resolvent_appendix}
    \end{equation}
    
    \item For a matrix $\mathbf{A} \in \R^{n\times n}$, vectors $\mathbf{u} \in \R^{n \times 1}$ and $\mathbf{v} \in \R^{n \times 1}$, and scalar $\alpha \in \C$, we have,
    \begin{equation}
    \mathrm{det}(\mathbf{A}+\alpha\,\mathbf{u}\mathbf{v}^\top) = \mathrm{det}(\mathbf{A}) + \alpha\,\mathbf{v}^\top \mathrm{adj}(\mathbf{A}) \mathbf{u} \label{eq:det_prop}
    \end{equation}
    
    \item For a matrix $\mathbf{A} \in \R^{n\times n}$ with eigenvalues $\lambda_k$, the square modulus of the determinant of matrix $\mathbf{A}+\dot{\iota}\omega\mathbf{I}$ can be written as,
    \begin{equation}
    ||\mathbf{A}+\dot{\iota}\omega\mathbf{I}||^2 = \prod_{k=1}^{n}|\lambda_k+\dot{\iota}\omega|^2 = \prod_{k=1}^{n}(\lambda_k^2+\omega^2) \label{eq:imp2}
    \end{equation}

    \item For a matrix $\mathbf{A} \in \R^{n\times n}$ with eigenvalues $\lambda_k$, we have,
    \begin{equation}
    {\mathrm{det}(\omega\I-\mathbf{A})\, \mathrm{det}(\omega\I-\mathbf{A}^\top)} = \left( \prod_{j=1}^{n} (\omega - \lambda_j) \right) \left( \prod_{k=1}^{n} (\omega - \lambda_k) \right)= \prod_{j=1}^{n} (\omega - \lambda_j)^2 \label{eq:imp21}
    \end{equation} 
    
    \item Given two polynomials $a(x)=a_0+a_1x+a_2x^2+...+a_n x^n$ and $b(x)=b_0+b_1x+b_2x^2+...+b_m x^m$, the product $a(x)b(x)$ yields a new polynomial of degree $m+n$,
    \begin{equation}
    a(x)b(x) = \sum_{k=0}^{m+n} c_k x^k
    \end{equation}
    where $c_k$ can be written as,
    \begin{equation}
    c_k =  \sum_{i+j=k} a_i b_j  \label{eq:pol_prod}
    \end{equation}
    
    \item We can express a monic polynomial in $\omega$, in terms of elementary symmetric polynomials $e^i (\lambda_1,\lambda_2,\dots,\lambda_n)$ as,
    \begin{equation}
   \prod_{i=1}^n (\lambda_i + \omega) = \sum_{i=0}^{n} e^i \omega^{n-i} \label{eq:esp}
    \end{equation}
\end{itemize}

\setlength\bibitemsep{10pt}
\printbibliography[heading=bibnumbered]

% --------------------------------------------------------------
%     You don't have to mess with anything below this line.
% --------------------------------------------------------------